\documentclass[12pt,reqno]{article}

\usepackage[usenames]{color}

\usepackage[colorlinks=true,
linkcolor=webgreen,
filecolor=webbrown,
citecolor=webgreen]{hyperref}

\definecolor{webgreen}{rgb}{0,.5,0}
\definecolor{webbrown}{rgb}{.6,0,0}

\usepackage{fullpage}
\usepackage{float}

\usepackage{graphics,amsmath,amssymb}
\usepackage{graphicx}
\usepackage{amsthm}
\usepackage{amsfonts}
\usepackage{latexsym}
\usepackage{epsf}

\DeclareMathOperator\Res{Res}

\begin{document}

\theoremstyle{plain}
\newtheorem{theorem}{Theorem}
\newtheorem{corollary}[theorem]{Corollary}
\newtheorem{lemma}[theorem]{Lemma}
\newtheorem{proposition}[theorem]{Proposition}

\theoremstyle{definition}
\newtheorem{definition}[theorem]{Definition}
\newtheorem{example}[theorem]{Example}
\newtheorem{conjecture}[theorem]{Conjecture}

\theoremstyle{remark}
\newtheorem{remark}[theorem]{Remark}

\begin{center}
\vskip 1cm{\LARGE\bf The Formulas for the Distribution of the\\
\vskip .12in
3-Smooth, 5-Smooth, 7-Smooth and all other Smooth Numbers}
\vskip 1cm
\large
Raphael Schumacher\\
raphschu@ethz.ch\\
\end{center}

\vskip .2 in

\begin{abstract}
\noindent In this paper we present and prove rapidly convergent formulas for the distribution of the $3$-smooth, $5$-smooth, $7$-smooth and all other smooth numbers. One of these formulas is another version of a formula due to Hardy and Littlewood for the arithmetic function $N_{a,b}(x)$, which counts the number of positive integers of the form $a^{p}b^{q}$ less than or equal to $x$.
\end{abstract}

\section{Introduction}

\noindent Let $a\in\mathbb{N}_{\geq2}$ be a fixed natural number.\\
\noindent Let $N_{a}(x)$ denote the number of natural numbers of the form $a^{p}$ which are smaller or equal to $x$, where $p\in\mathbb{N}_{0}$.\\
\noindent By definition \cite{1,2}, we have that the $2$-smooth numbers are just the powers of $2$, namely
\begin{displaymath}
\begin{split}
S_{2}:=\{1,2,4,8,16,32,64,128,256,512,1024,\ldots\}
\end{split}
\end{displaymath}

\noindent and that the formula for their distribution is
\begin{displaymath}
\begin{split}
N_{2}(x)&=\frac{\log(x)}{\log(2)}+\frac{1}{2}-B_{1}\left(\left\{\frac{\log(x)}{\log(2)}\right\}\right).
\end{split}
\end{displaymath}

\noindent This follows directly from the more general formula
\begin{displaymath}
\begin{split}
N_{a}(x)&=\frac{\log(x)}{\log(a)}+\frac{1}{2}-B_{1}\left(\left\{\frac{\log(x)}{\log(a)}\right\}\right).
\end{split}
\end{displaymath}

\noindent Numbers of the form $2^{p}3^{q}$, where $p\in\mathbb{N}_{0}$ and $q\in\mathbb{N}_{0}$ are called $3$-smooth numbers \cite{1,2,3}, because these numbers are exactly the numbers which have no prime factors larger than $3$. We will denote the sequence of $3$-smooth numbers by $S_{2,3}$. Thus, we have that
\begin{displaymath}
\begin{split}
S_{2,3}:&=\left\{2^{p}3^{q}:p\in\mathbb{N}_{0},q\in\mathbb{N}_{0}\right\}\\
&=\left\{1,2,3,4,6,8,9,12,16,18,24,27,32,36,48,54,64,72,81,96,108,128,144,\ldots\right\}.
\end{split}
\end{displaymath}

\noindent More generally, let $a,b\in\mathbb{N}$ be fixed natural numbers such that $a<b$ and $\gcd(a,b)=1$.\\
\noindent Let $N_{a,b}(x)$ denote the number of natural numbers of the form $a^{p}b^{q}$ which are smaller or equal to $x$, where $p,q\in\mathbb{N}_{0}$. Furthermore, we denote by $\chi_{S_{a,b}}(x)$ the characteristic function of the natural numbers of the form $a^{p}b^{q}$.\\

\noindent In his first letter to Hardy \cite{4,5,6,7}, Ramanujan gave the formula
\begin{displaymath}
\begin{split}
N_{2,3}(x)&\approx\frac{1}{2}\frac{\log(2x)\log(3x)}{\log(2)\log(3)}+\frac{1}{2}\chi_{S_{2,3}}(x),
\end{split}
\end{displaymath}
\noindent which provides a very close approximation to the number $N_{2,3}(x)$ of $3$-smooth numbers less than or equal to $x$.

\noindent In his notebooks \cite{4,7}, Ramanujan later generalized this expression to all $a,b\in\mathbb{N}$ with $\gcd(a,b)=1$, namely
\begin{displaymath}
\begin{split}
N_{a,b}(x)&\approx\frac{1}{2}\frac{\log(ax)\log(bx)}{\log(a)\log(b)}+\frac{1}{2}\chi_{S_{a,b}}(x),
\end{split}
\end{displaymath}
\noindent which is again a very close approximation to $N_{a,b}(x)$.

\noindent The analog formula for $N_{a,b,c}(x)$ \cite{8,9,10} is
\begin{displaymath}
\begin{split}
N_{a,b,c}(x)&\approx\frac{\log\left(x\sqrt{abc}\right)^{3}}{6\log(a)\log(b)\log(c)}+\frac{1}{2}\chi_{S_{a,b,c}}(x),
\end{split}
\end{displaymath}
\noindent which is also a very good approximation to $N_{a,b,c}(x)$.\\

\noindent In the following sections, we present and prove rapidly convergent formulas for the functions $N_{a,b}(x)$ and $N_{2,3}(x)$, having the above Ramanujan approximations as their first term. These two formulas are other versions of a more rapidly convergent formula already found by Hardy and Littlewood around 1920 \cite{11,12}, as it was communicated to us by Emanuele Tron \cite{13}.\\
We also prove very rapidly convergent formulas for the distribution of the $5$-smooth, $7$-smooth and all other smooth numbers.\\
At the end of the paper, we give an exact formula for the counting function of the natural numbers of the form $a^{p^{2}}b^{q^{2}}$.\\
\noindent We have searched all resulting formulas (which are given in theorems and corollaries) in the literature and on the internet, but we could only find the Hardy-Littlewood formula \cite{11,12}. Therefore, we believe that all other results are new.

\section{The Formulas for $N_{a,b}(x)$ and the 3-Smooth Numbers Counting Function $N_{2,3}(x)$}

\noindent Let $a,b\in\mathbb{N}$ such that $a<b$ and $\gcd(a,b)=1$.\\
\noindent For $x\in\mathbb{R}_{0}^{+}$, we define the function $N_{a,b}(x)$ by
\begin{displaymath}
\begin{split}
N_{a,b}(x):&=\sum_{\substack{a^{p}b^{q}\leq x\\p\in\mathbb{N}_{0},q\in\mathbb{N}_{0}}}1.
\end{split}
\end{displaymath}

\noindent Moreover, we denote the set of natural numbers of the form $a^{p}b^{q}$ by $S_{a,b}$ and its characteristic function by $\chi_{S_{a,b}}(x)$, that is
\begin{displaymath}
\begin{split}
S_{a,b}:&=\left\{a^{p}b^{q}:p\in\mathbb{N}_{0},q\in\mathbb{N}_{0}\right\},\\
\chi_{S_{a,b}}(x):&=\begin{cases}1&\text{if $x\in S_{a,b}$}\\0&\text{if $x\notin S_{a,b}$}\end{cases}.
\end{split}
\end{displaymath}

\noindent We have that

\begin{displaymath}
\begin{split}
N_{a,b}(x)&=1+\sum_{k=0}^{\left\lfloor\log_{a}(x)\right\rfloor}\left\lfloor\log_{b}\left(\frac{x}{a^{k}}\right)\right\rfloor+\left\lfloor\log_{a}(x)\right\rfloor.
\end{split}
\end{displaymath}

\begin{theorem}(Our Formula for $N_{a,b}(x)$)\\
\noindent For every real number $x>1$, we have that
\begin{displaymath}
\begin{split}
N_{a,b}(x)&=\frac{1}{2}\frac{\log(ax)\log(bx)}{\log(a)\log(b)}+\frac{\log(a)}{12\log(b)}+\frac{\log(b)}{12\log(a)}-\frac{1}{4}-\frac{1}{2}B^{*}_{1}\left(\left\{\frac{\log(x)}{\log(a)}\right\}\right)\\
&\quad-\frac{1}{2}B^{*}_{1}\left(\left\{\frac{\log(x)}{\log(b)}\right\}\right)-\frac{\log(a)}{2\log(b)}B_{2}\left(\left\{\frac{\log(x)}{\log(a)}\right\}\right)-\frac{\log(b)}{2\log(a)}B_{2}\left(\left\{\frac{\log(x)}{\log(b)}\right\}\right)\\
&\quad+\frac{\log(a)\log(b)}{\pi^{2}}\sum_{n=1}^{\infty}\sum_{m=1}^{\infty}\frac{\cos\left(\frac{2\pi n\log(x)}{\log(a)}\right)-\cos\left(\frac{2\pi m\log(x)}{\log(b)}\right)}{m^{2}\log(a)^{2}-n^{2}\log(b)^{2}}+\frac{1}{2}\chi_{S_{a,b}}(x),
\end{split}
\end{displaymath}

\noindent where
\begin{displaymath}
\begin{split}
B^{*}_{1}(\{x\}):&=\begin{cases}\{x\}-\frac{1}{2}&\text{if $x\notin\mathbb{Z}$}\\0&\text{if $x\in\mathbb{Z}$}\end{cases},\\
B_{2}(\{x\}):&=\{x\}^{2}-\{x\}+\frac{1}{6}\;\;\forall x\in\mathbb{R}.
\end{split}
\end{displaymath}
\end{theorem}

\noindent The above formula converges rapidly.\\
\noindent As usual we denote by $\{x\}$ the fractional part of $x$.\\

\noindent This is just another version of the following

\begin{theorem}(The Hardy-Littlewood formula for $N_{a,b}(x)$)\cite{11,12}\\
\noindent For every real number $x\geq1$, we have that
\begin{displaymath}
\begin{split}
N_{a,b}(x)&=\frac{1}{2}\frac{\log(ax)\log(bx)}{\log(a)\log(b)}+\frac{\log(a)}{12\log(b)}+\frac{\log(b)}{12\log(a)}-\frac{1}{4}-B^{*}_{1}\left(\left\{\frac{\log(x)}{\log(a)}\right\}\right)-B^{*}_{1}\left(\left\{\frac{\log(x)}{\log(b)}\right\}\right)\\
&\quad-\frac{1}{2\pi}\sum_{k=1}^{\infty}\left(\frac{\cos\left(2\pi k\frac{\log(x)-\frac{1}{2}\log(b)}{\log(a)}\right)}{k\sin\left(\frac{\pi k\log(b)}{\log(a)}\right)}+\frac{\cos\left(2\pi k\frac{\log(x)-\frac{1}{2}\log(a)}{\log(b)}\right)}{k\sin\left(\frac{\pi k\log(a)}{\log(b)}\right)}\right)+\frac{1}{2}\chi_{S_{a,b}}(x),\end{split}
\end{displaymath}
\noindent where the series is to be interpreted as meaning \cite{12}
\begin{displaymath}
\begin{split}
&\sum_{k=1}^{\infty}\left(\frac{\cos\left(2\pi k\frac{\log(x)-\frac{1}{2}\log(b)}{\log(a)}\right)}{k\sin\left(\frac{\pi k\log(b)}{\log(a)}\right)}+\frac{\cos\left(2\pi k\frac{\log(x)-\frac{1}{2}\log(a)}{\log(b)}\right)}{k\sin\left(\frac{\pi k\log(a)}{\log(b)}\right)}\right)\\
=&\lim_{R\rightarrow\infty}\left(\sum_{k=1}^{\left\lfloor R\log(a)\right\rfloor}\frac{\cos\left(2\pi k\frac{\log(x)-\frac{1}{2}\log(b)}{\log(a)}\right)}{k\sin\left(\frac{\pi k\log(b)}{\log(a)}\right)}+\sum_{k=1}^{\left\lfloor R\log(b)\right\rfloor}\frac{\cos\left(2\pi k\frac{\log(x)-\frac{1}{2}\log(a)}{\log(b)}\right)}{k\sin\left(\frac{\pi k\log(a)}{\log(b)}\right)}\right),
\end{split}
\end{displaymath}
\noindent when $R\rightarrow\infty$ in an appropriate manner.
\end{theorem}

\noindent This formula converges very rapidly.\\

\noindent Setting $a=2$ and $b=3$, we get immediately two formulas for the distribution of the $3$-smooth numbers, namely

\begin{corollary}(Our Formula for the $3$-Smooth Numbers Counting Function $N_{2,3}(x)$)\\
\noindent For every real number $x>1$, we have that
\begin{displaymath}
\begin{split}
N_{2,3}(x)&=\frac{1}{2}\frac{\log(2x)\log(3x)}{\log(2)\log(3)}+\frac{\log(2)}{12\log(3)}+\frac{\log(3)}{12\log(2)}-\frac{1}{4}-\frac{1}{2}B^{*}_{1}\left(\left\{\frac{\log(x)}{\log(2)}\right\}\right)\\
&\quad-\frac{1}{2}B^{*}_{1}\left(\left\{\frac{\log(x)}{\log(3)}\right\}\right)-\frac{\log(2)}{2\log(3)}B_{2}\left(\left\{\frac{\log(x)}{\log(2)}\right\}\right)-\frac{\log(3)}{2\log(2)}B_{2}\left(\left\{\frac{\log(x)}{\log(3)}\right\}\right)\\
&\quad+\frac{\log(2)\log(3)}{\pi^{2}}\sum_{n=1}^{\infty}\sum_{m=1}^{\infty}\frac{\cos\left(\frac{2\pi n\log(x)}{\log(2)}\right)-\cos\left(\frac{2\pi m\log(x)}{\log(3)}\right)}{m^{2}\log(2)^{2}-n^{2}\log(3)^{2}}+\frac{1}{2}\chi_{S_{2,3}}(x).
\end{split}
\end{displaymath}
\end{corollary}

\begin{corollary}(The Hardy-Littlewood formula for $N_{2,3}(x)$)\cite{11,12}\\\noindent For every real number $x\geq1$, we have that
\begin{displaymath}
\begin{split}
N_{2,3}(x)&=\frac{1}{2}\frac{\log(2x)\log(3x)}{\log(2)\log(3)}+\frac{\log(2)}{12\log(3)}+\frac{\log(3)}{12\log(2)}-\frac{1}{4}-B^{*}_{1}\left(\left\{\frac{\log(x)}{\log(2)}\right\}\right)-B^{*}_{1}\left(\left\{\frac{\log(x)}{\log(3)}\right\}\right)\\
&\quad-\frac{1}{2\pi}\sum_{k=1}^{\infty}\left(\frac{\cos\left(2\pi k\frac{\log(x)-\frac{1}{2}\log(3)}{\log(2)}\right)}{k\sin\left(\frac{\pi k\log(3)}{\log(2)}\right)}+\frac{\cos\left(2\pi k\frac{\log(x)-\frac{1}{2}\log(2)}{\log(3)}\right)}{k\sin\left(\frac{\pi k\log(2)}{\log(3)}\right)}\right)+\frac{1}{2}\chi_{S_{2,3}}(x),
\end{split}
\end{displaymath}
\noindent where the series is to be interpreted as meaning \cite{12}
\begin{displaymath}
\begin{split}
&\sum_{k=1}^{\infty}\left(\frac{\cos\left(2\pi k\frac{\log(x)-\frac{1}{2}\log(3)}{\log(2)}\right)}{k\sin\left(\frac{\pi k\log(3)}{\log(2)}\right)}+\frac{\cos\left(2\pi k\frac{\log(x)-\frac{1}{2}\log(2)}{\log(3)}\right)}{k\sin\left(\frac{\pi k\log(2)}{\log(3)}\right)}\right)\\
=&\lim_{R\rightarrow\infty}\left(\sum_{k=1}^{\left\lfloor R\log(2)\right\rfloor}\frac{\cos\left(2\pi k\frac{\log(x)-\frac{1}{2}\log(3)}{\log(2)}\right)}{k\sin\left(\frac{\pi k\log(3)}{\log(2)}\right)}+\sum_{k=1}^{\left\lfloor R\log(3)\right\rfloor}\frac{\cos\left(2\pi k\frac{\log(x)-\frac{1}{2}\log(2)}{\log(3)}\right)}{k\sin\left(\frac{\pi k\log(2)}{\log(3)}\right)}\right),
\end{split}
\end{displaymath}
\noindent when $R\rightarrow\infty$ in an appropriate manner.
\end{corollary}

\noindent Using the computationally more efficient formula
\begin{displaymath}
\begin{split}
N_{2,3}(x)&=1+\sum_{k=0}^{\left\lfloor\log_{3}(x)\right\rfloor}\left\lfloor\log_{2}\left(\frac{x}{3^{k}}\right)\right\rfloor+\left\lfloor\log_{3}(x)\right\rfloor,
\end{split}
\end{displaymath}
\noindent we get the following two tables:\\
 
\begin{table}[htbp]
\begin{center}
\begin{tabular}{| c | c | c | c |}
\hline
$x$&$N_{2,3}(x)$&$\text{Our Formula for }N_{2,3}(x)$&$\text{Number of terms $(n,m)$}$\\
\hline
\bfseries $1$&$1$&$1.0510201857955517$&$(n,m)=(4,4)\text{\;at $x=1.1$}$\\
\hline
\bfseries $10$&$7$&$7.0071497373839231$&$(n,m)=(22,22)$\\
\hline
\bfseries $10^{2}$&$20$&$20.0045160354084706$&$(n,m)=(10,10)$\\
\hline
\bfseries $10^{3}$&$40$&40.0039084310672772&$(n,m)=(12,12)$\\
\hline
\bfseries $10^{4}$&$67$&$67.0408408937206653
$&$(n,m)=(20,20)$\\
\hline
\bfseries $10^{5}$&$101$&$101.05072154439969785$&$(n,m)=(28,28)$\\
\hline
\bfseries $10^{6}$&$142$&$142.01315000789587358$&$(n,m)=(70,70)$\\
\hline
\bfseries $10^{7}$&$190$&$190.00707389223323501$&$(n,m)=(110,110)$\\
\hline
\bfseries $10^{8}$&$244$&$244.00659912032029415$&$(n,m)=(140,140)$\\
\hline
\bfseries $10^{9}$&$306$&$306.00585869480145596$&$(n,m)=(160,160)$\\
\hline
\bfseries $10^{10}$&$376$&$376.02126583465866742$&$(n,m)=(170,170)$\\
\hline
\bfseries $10^{10^{2}}$&$35084$&$35084.0568926232894816675$&$(n,m)=(2000,2000)$\\
\hline
\bfseries $10^{10^{3}}$&$3483931$&$3483931.035272714689991309386$&$(n,m)=(4000,4000)$\\
\hline
\end{tabular}
\caption{Values of $N_{2,3}(x)$}\end{center}
\end{table}

\newpage

\begin{table}[htbp]
\begin{center}
\begin{tabular}{| c | c | c | c |}
\hline
$x$&$N_{2,3}(x)$&$\text{The Hardy-Littlewood Formula for }N_{2,3}(x)$&$\text{Number of terms $R$}$\\
\hline
\bfseries $1$&$1$&$1.00408281281244794423184044310637662236$&$R=1$\\
\hline
\bfseries $10$&$7$&$7.01039536792580652845911613960427072715$&$R=6$\\
\hline
\bfseries $10^{2}$&$20$&$20.00554687989157075178992137362449803027$&$R=10$\\
\hline
\bfseries $10^{3}$&$40$&$40.00416733658863125098651349198857667561$&$R=26$\\
\hline
\bfseries $10^{4}$&$67$&$67.04067163854917851848072234444363738009$&$R=32$\\
\hline
\bfseries $10^{5}$&$101$&$101.00383710643693392983460661037688277109$&$R=44$\\
\hline
\bfseries $10^{6}$&$142$&$142.00519665851176957826409909346411626717$&$R=60$\\
\hline
\bfseries $10^{7}$&$190$&$190.00431172466646336030921684292744206625$&$R=100$\\
\hline
\bfseries $10^{8}$&$244$&$244.00043300366963526250817238561664826018$&$R=122$\\
\hline
\bfseries $10^{9}$&$306$&$306.00450681431786167717856515798365069396$&$R=146$\\
\hline
\bfseries $10^{10}$&$376$&$376.02231447192801988487484982661961706561$&$R=160$\\
\hline
\bfseries $10^{10^{2}}$&$35084$&$35084.03451234481158685735036751788214481906$&$R=3000$\\
\hline
\bfseries $10^{10^{3}}$&$3483931$&$3483931.03067546896021243171738747049589388966$&$R=3000$\\
\hline
\bfseries $10^{10^{4}}$&$348149087$&$348149087.05625852937187129720297862230958308491$&$R=24000$\\
\hline
\bfseries $10^{10^{5}}$&$34812470748$&$34812470748.06400873722492550333469431071713138958$&$R=200000$\\
\hline
\end{tabular}
\caption{Values of $N_{2,3}(x)$}\end{center}
\end{table}

\begin{corollary}(Modified version of our formula for $N_{a,b}(x)$)\\
\noindent For every real number $x>1$, we have
\begin{displaymath}
\begin{split}
N_{a,b}(x)&=\frac{\log(x)^{2}}{2\log(a)\log(b)}+\frac{\log(x)}{2\log(a)}+\frac{\log(x)}{2\log(b)}+\frac{1}{4}+\frac{\log(a)}{12\log(b)}+\frac{\log(b)}{12\log(a)}-\frac{1}{2}B^{*}_{1}\left(\left\{\frac{\log(x)}{\log(a)}\right\}\right)\\
&\quad-\frac{1}{2}B^{*}_{1}\left(\left\{\frac{\log(x)}{\log(b)}\right\}\right)-\frac{\log(a)}{2\log(b)}B_{2}\left(\left\{\frac{\log(x)}{\log(a)}\right\}\right)-\frac{\log(b)}{2\log(a)}B_{2}\left(\left\{\frac{\log(x)}{\log(b)}\right\}\right)\\
&\quad+\frac{\log(a)\log(b)}{\pi^{2}}\sum_{n=1}^{\infty}\sum_{m=1}^{\infty}\frac{\cos\left(\frac{2\pi n\log(x)}{\log(a)}\right)-\cos\left(\frac{2\pi m\log(x)}{\log(b)}\right)}{m^{2}\log(a)^{2}-n^{2}\log(b)^{2}}+\frac{1}{2}\chi_{S_{a,b}}(x).
\end{split}
\end{displaymath}
\end{corollary}

\begin{corollary}(Modified Hardy-Littlewood formula for $N_{a,b}(x)$)\cite{11,12}\\
\noindent For every real number $x\geq1$, we have
\begin{displaymath}
\begin{split}
N_{a,b}(x)&=\frac{\log(x)^{2}}{2\log(a)\log(b)}+\frac{\log(x)}{2\log(a)}+\frac{\log(x)}{2\log(b)}+\frac{1}{4}+\frac{\log(a)}{12\log(b)}+\frac{\log(b)}{12\log(a)}-B^{*}_{1}\left(\left\{\frac{\log(x)}{\log(a)}\right\}\right)\\
&\quad-B^{*}_{1}\left(\left\{\frac{\log(x)}{\log(b)}\right\}\right)-\frac{1}{2\pi}\sum_{k=1}^{\infty}\left(\frac{\cos\left(2\pi k\frac{\log(x)-\frac{1}{2}\log(a)}{\log(b)}\right)}{k\sin\left(\frac{\pi k\log(a)}{\log(b)}\right)}+\frac{\cos\left(2\pi k\frac{\log(x)-\frac{1}{2}\log(b)}{\log(a)}\right)}{k\sin\left(\frac{\pi k\log(b)}{\log(a)}\right)}\right)+\frac{1}{2}\chi_{S_{a,b}}(x),
\end{split}
\end{displaymath}
\noindent where the series is interpreted as mentioned above.
\end{corollary}

\section{The Formula for the Distribution of the 5-Smooth Numbers}

\noindent Let $a,b,c\in\mathbb{N}$ such that $a<b<c$ and $\gcd(a,b,c)=1$.\\
\noindent For $x\in\mathbb{R}^{+}_{0}$, we define the function $N_{a,b,c}(x)$ by
\begin{displaymath}
\begin{split}
N_{a,b,c}(x):&=\sum_{\substack{a^{p}b^{q}c^{l}\leq x\\p\in\mathbb{N}_{0},q\in\mathbb{N}_{0},l\in\mathbb{N}_{0}}}1.
\end{split}
\end{displaymath}

\noindent We define also
\begin{displaymath}
\begin{split}
S_{a,b,c}:&=\left\{a^{p}b^{q}c^{l}:p\in\mathbb{N}_{0},q\in\mathbb{N}_{0},l\in\mathbb{N}_{0}\right\},\\
\chi_{S_{a,b,c}}(x):&=\begin{cases}1&\text{if $x\in S_{a,b,c}$}\\0&\text{if $x\notin S_{a,b,c}$}\end{cases}.
\end{split}
\end{displaymath}

\noindent Thus, we have that
\begin{displaymath}
\begin{split}
N_{a,b,c}(x)&=\sum_{k=0}^{\left\lfloor\log_{a}(x)\right\rfloor}\sum_{l=0}^{\left\lfloor\log_{b}\left(\frac{x}{a^{k}}\right)\right\rfloor}\left(\left\lfloor\log_{c}\left(\frac{x}{a^{k}b^{l}}\right)\right\rfloor+1\right).
\end{split}
\end{displaymath}

\noindent We have the following

\begin{theorem}(Formula for $N_{a,b,c}(x)$)\\
\noindent For every real number $x\geq1$, we have that
\begin{displaymath}
\begin{split}
N_{a,b,c}(x)&=\frac{\log(x)^{3}}{6\log(a)\log(b)\log(c)}+\frac{\log(x)^{2}}{4\log(a)\log(b)}+\frac{\log(x)^{2}}{4\log(a)\log(c)}+\frac{\log(x)^{2}}{4\log(b)\log(c)}+\frac{\log(x)}{4\log(a)}\\
&\quad+\frac{\log(x)}{4\log(b)}+\frac{\log(x)}{4\log(c)}+\frac{\log(a)\log(x)}{12\log(b)\log(c)}+\frac{\log(b)\log(x)}{12\log(a)\log(c)}+\frac{\log(c)\log(x)}{12\log(a)\log(b)}\\
&\quad+\frac{\log(a)}{24\log(b)}+\frac{\log(a)}{24\log(c)}+\frac{\log(b)}{24\log(a)}+\frac{\log(b)}{24\log(c)}+\frac{\log(c)}{24\log(a)}+\frac{\log(c)}{24\log(b)}+\frac{1}{8}\\
&\quad-B^{*}_{1}\left(\left\{\frac{\log(x)}{\log(a)}\right\}\right)-B^{*}_{1}\left(\left\{\frac{\log(x)}{\log(b)}\right\}\right)-B^{*}_{1}\left(\left\{\frac{\log(x)}{\log(c)}\right\}\right)\\
&\quad-\frac{1}{4\pi}\sum_{k=1}^{\infty}\left(\frac{\cos\left(2\pi k\frac{\log(x)-\frac{1}{2}\log(b)}{\log(a)}\right)}{k\sin\left(\frac{\pi k\log(b)}{\log(a)}\right)}+\frac{\cos\left(2\pi k\frac{\log(x)-\frac{1}{2}\log(a)}{\log(b)}\right)}{k\sin\left(\frac{\pi k\log(a)}{\log(b)}\right)}\right)\\
&\quad-\frac{1}{4\pi}\sum_{k=1}^{\infty}\left(\frac{\cos\left(2\pi k\frac{\log(x)-\frac{1}{2}\log(c)}{\log(b)}\right)}{k\sin\left(\frac{\pi k\log(c)}{\log(b)}\right)}+\frac{\cos\left(2\pi k\frac{\log(x)-\frac{1}{2}\log(b)}{\log(c)}\right)}{k\sin\left(\frac{\pi k\log(b)}{\log(c)}\right)}\right)\\
&\quad-\frac{1}{4\pi}\sum_{k=1}^{\infty}\left(\frac{\cos\left(2\pi k\frac{\log(x)-\frac{1}{2}\log(a)}{\log(c)}\right)}{k\sin\left(\frac{\pi k\log(a)}{\log(c)}\right)}+\frac{\cos\left(2\pi k\frac{\log(x)-\frac{1}{2}\log(c)}{\log(a)}\right)}{k\sin\left(\frac{\pi k\log(c)}{\log(a)}\right)}\right)\\
\end{split}
\end{displaymath}
\begin{displaymath}
\begin{split}
\hspace{1.5cm}&\quad-\frac{1}{8\pi}\sum_{k=1}^{\infty}\left(\frac{\sin\left(2\pi k\frac{\log(x)+\frac{1}{2}\log(b)-\frac{1}{2}\log(c)}{\log(a)}\right)}{k\sin\left(\frac{\pi k\log(b)}{\log(a)}\right)\sin\left(\frac{\pi k\log(c)}{\log(a)}\right)}+\frac{\sin\left(2\pi k\frac{\log(x)-\frac{1}{2}\log(b)+\frac{1}{2}\log(c)}{\log(a)}\right)}{k\sin\left(\frac{\pi k\log(b)}{\log(a)}\right)\sin\left(\frac{\pi k\log(c)}{\log(a)}\right)}\right)\\
&\quad-\frac{1}{8\pi}\sum_{k=1}^{\infty}\left(\frac{\sin\left(2\pi k\frac{\log(x)+\frac{1}{2}\log(a)-\frac{1}{2}\log(c)}{\log(b)}\right)}{k\sin\left(\frac{\pi k\log(a)}{\log(b)}\right)\sin\left(\frac{\pi k\log(c)}{\log(b)}\right)}+\frac{\sin\left(2\pi k\frac{\log(x)-\frac{1}{2}\log(a)+\frac{1}{2}\log(c)}{\log(b)}\right)}{k\sin\left(\frac{\pi k\log(a)}{\log(b)}\right)\sin\left(\frac{\pi k\log(c)}{\log(b)}\right)}\right)\\
&\quad-\frac{1}{8\pi}\sum_{k=1}^{\infty}\left(\frac{\sin\left(2\pi k\frac{\log(x)+\frac{1}{2}\log(a)-\frac{1}{2}\log(b)}{\log(c)}\right)}{k\sin\left(\frac{\pi k\log(a)}{\log(c)}\right)\sin\left(\frac{\pi k\log(b)}{\log(c)}\right)}+\frac{\sin\left(2\pi k\frac{\log(x)-\frac{1}{2}\log(a)+\frac{1}{2}\log(b)}{\log(c)}\right)}{k\sin\left(\frac{\pi k\log(a)}{\log(c)}\right)\sin\left(\frac{\pi k\log(b)}{\log(c)}\right)}\right)+\frac{1}{2}\chi_{S_{a,b,c}}(x).
\end{split}
\end{displaymath}
\end{theorem}

\noindent This formula converges very rapidly.\\

\noindent In the above formula, the series are to be interpreted as meaning
\begin{displaymath}
\begin{split}
&\sum_{k=1}^{\infty}\left(\frac{\cos\left(2\pi k\frac{\log(x)-\frac{1}{2}\log(b)}{\log(a)}\right)}{k\sin\left(\frac{\pi k\log(b)}{\log(a)}\right)}+\frac{\cos\left(2\pi k\frac{\log(x)-\frac{1}{2}\log(a)}{\log(b)}\right)}{k\sin\left(\frac{\pi k\log(a)}{\log(b)}\right)}\right)\\
=&\lim_{R\rightarrow\infty}\left(\sum_{k=1}^{\left\lfloor R\log(a)\right\rfloor}\frac{\cos\left(2\pi k\frac{\log(x)-\frac{1}{2}\log(b)}{\log(a)}\right)}{k\sin\left(\frac{\pi k\log(b)}{\log(a)}\right)}+\sum_{k=1}^{\left\lfloor R\log(b)\right\rfloor}\frac{\cos\left(2\pi k\frac{\log(x)-\frac{1}{2}\log(a)}{\log(b)}\right)}{k\sin\left(\frac{\pi k\log(a)}{\log(b)}\right)}\right)
\end{split}
\end{displaymath}
\noindent and
\begin{displaymath}
\begin{split}
&\sum_{k=1}^{\infty}\left(\frac{\sin\left(2\pi k\frac{\log(x)+\frac{1}{2}\log(b)-\frac{1}{2}\log(c)}{\log(a)}\right)}{k\sin\left(\frac{\pi k\log(b)}{\log(a)}\right)\sin\left(\frac{\pi k\log(c)}{\log(a)}\right)}+\frac{\sin\left(2\pi k\frac{\log(x)-\frac{1}{2}\log(b)+\frac{1}{2}\log(c)}{\log(a)}\right)}{k\sin\left(\frac{\pi k\log(b)}{\log(a)}\right)\sin\left(\frac{\pi k\log(c)}{\log(a)}\right)}\right)\\
=&\lim_{R\rightarrow\infty}\left(\sum_{k=1}^{\left\lfloor R\log(a)\right\rfloor}\left(\frac{\sin\left(2\pi k\frac{\log(x)+\frac{1}{2}\log(b)-\frac{1}{2}\log(c)}{\log(a)}\right)}{k\sin\left(\frac{\pi k\log(b)}{\log(a)}\right)\sin\left(\frac{\pi k\log(c)}{\log(a)}\right)}+\frac{\sin\left(2\pi k\frac{\log(x)-\frac{1}{2}\log(b)+\frac{1}{2}\log(c)}{\log(a)}\right)}{k\sin\left(\frac{\pi k\log(b)}{\log(a)}\right)\sin\left(\frac{\pi k\log(c)}{\log(a)}\right)}\right)\right),
\end{split}
\end{displaymath}
\noindent when $R\rightarrow\infty$ in an appropriate manner.\\

\noindent Setting $a=2$, $b=3$ and $c=5$ and interpreting the series, like before, as meaning (for example)
\begin{displaymath}
\begin{split}
&\sum_{k=1}^{\infty}\left(\frac{\cos\left(2\pi k\frac{\log(x)-\frac{1}{2}\log(3)}{\log(2)}\right)}{k\sin\left(\frac{\pi k\log(3)}{\log(2)}\right)}+\frac{\cos\left(2\pi k\frac{\log(x)-\frac{1}{2}\log(2)}{\log(3)}\right)}{k\sin\left(\frac{\pi k\log(2)}{\log(3)}\right)}\right)\\
=&\lim_{R\rightarrow\infty}\left(\sum_{k=1}^{\left\lfloor R\log(2)\right\rfloor}\frac{\cos\left(2\pi k\frac{\log(x)-\frac{1}{2}\log(3)}{\log(2)}\right)}{k\sin\left(\frac{\pi k\log(3)}{\log(2)}\right)}+\sum_{k=1}^{\left\lfloor R\log(3)\right\rfloor}\frac{\cos\left(2\pi k\frac{\log(x)-\frac{1}{2}\log(2)}{\log(3)}\right)}{k\sin\left(\frac{\pi k\log(2)}{\log(3)}\right)}\right)
\end{split}
\end{displaymath}
\noindent and
\begin{displaymath}
\begin{split}
&\sum_{k=1}^{\infty}\left(\frac{\sin\left(2\pi k\frac{\log(x)+\frac{1}{2}\log(3)-\frac{1}{2}\log(5)}{\log(2)}\right)}{k\sin\left(\frac{\pi k\log(3)}{\log(2)}\right)\sin\left(\frac{\pi k\log(5)}{\log(2)}\right)}+\frac{\sin\left(2\pi k\frac{\log(x)-\frac{1}{2}\log(3)+\frac{1}{2}\log(5)}{\log(2)}\right)}{k\sin\left(\frac{\pi k\log(3)}{\log(2)}\right)\sin\left(\frac{\pi k\log(5)}{\log(2)}\right)}\right)\\
=&\lim_{R\rightarrow\infty}\left(\sum_{k=1}^{\left\lfloor R\log(2)\right\rfloor}\left(\frac{\sin\left(2\pi k\frac{\log(x)+\frac{1}{2}\log(3)-\frac{1}{2}\log(5)}{\log(2)}\right)}{k\sin\left(\frac{\pi k\log(3)}{\log(2)}\right)\sin\left(\frac{\pi k\log(5)}{\log(2)}\right)}+\frac{\sin\left(2\pi k\frac{\log(x)-\frac{1}{2}\log(3)+\frac{1}{2}\log(5)}{\log(2)}\right)}{k\sin\left(\frac{\pi k\log(3)}{\log(2)}\right)\sin\left(\frac{\pi k\log(5)}{\log(2)}\right)}\right)\right),
\end{split}
\end{displaymath}
\noindent when $R\rightarrow\infty$ in an appropriate manner, we get for the sequence
\begin{displaymath}
\begin{split}
S_{2,3,5}:&=\left\{2^{p}3^{q}5^{l}:p\in\mathbb{N}_{0},q\in\mathbb{N}_{0},l\in\mathbb{N}_{0}\right\}\\
&=\left\{1,2,3,4,5,6,8,9,10,12,15,16,18,20,24,25,27,30,32,36,40,45,48,\ldots\right\},
\end{split}
\end{displaymath}
\noindent of $5$-smooth numbers (regular numbers or Hamming numbers) \cite{9,10}, the following formula

\begin{corollary}(Formula for the $5$-Smooth Numbers Counting Function $N_{2,3,5}(x)$)\\
\noindent For every real number $x\geq1$, we have that
\begin{displaymath}
\begin{split}
N_{2,3,5}(x)&=\frac{\log(x)^{3}}{6\log(2)\log(3)\log(5)}+\frac{\log(x)^{2}}{4\log(2)\log(3)}+\frac{\log(x)^{2}}{4\log(2)\log(5)}+\frac{\log(x)^{2}}{4\log(3)\log(5)}+\frac{\log(x)}{4\log(2)}\\
&\quad+\frac{\log(x)}{4\log(3)}+\frac{\log(x)}{4\log(5)}+\frac{\log(2)\log(x)}{12\log(3)\log(5)}+\frac{\log(3)\log(x)}{12\log(2)\log(5)}+\frac{\log(5)\log(x)}{12\log(2)\log(3)}\\
&\quad+\frac{\log(2)}{24\log(3)}+\frac{\log(2)}{24\log(5)}+\frac{\log(3)}{24\log(2)}+\frac{\log(3)}{24\log(5)}+\frac{\log(5)}{24\log(2)}+\frac{\log(5)}{24\log(3)}+\frac{1}{8}\\
&\quad-B^{*}_{1}\left(\left\{\frac{\log(x)}{\log(2)}\right\}\right)-B^{*}_{1}\left(\left\{\frac{\log(x)}{\log(3)}\right\}\right)-B^{*}_{1}\left(\left\{\frac{\log(x)}{\log(5)}\right\}\right)\\
&\quad-\frac{1}{4\pi}\sum_{k=1}^{\infty}\left(\frac{\cos\left(2\pi k\frac{\log(x)-\frac{1}{2}\log(3)}{\log(2)}\right)}{k\sin\left(\frac{\pi k\log(3)}{\log(2)}\right)}+\frac{\cos\left(2\pi k\frac{\log(x)-\frac{1}{2}\log(2)}{\log(3)}\right)}{k\sin\left(\frac{\pi k\log(2)}{\log(3)}\right)}\right)\\
&\quad-\frac{1}{4\pi}\sum_{k=1}^{\infty}\left(\frac{\cos\left(2\pi k\frac{\log(x)-\frac{1}{2}\log(5)}{\log(3)}\right)}{k\sin\left(\frac{\pi k\log(5)}{\log(3)}\right)}+\frac{\cos\left(2\pi k\frac{\log(x)-\frac{1}{2}\log(3)}{\log(5)}\right)}{k\sin\left(\frac{\pi k\log(3)}{\log(5)}\right)}\right)\\
&\quad-\frac{1}{4\pi}\sum_{k=1}^{\infty}\left(\frac{\cos\left(2\pi k\frac{\log(x)-\frac{1}{2}\log(2)}{\log(5)}\right)}{k\sin\left(\frac{\pi k\log(2)}{\log(5)}\right)}+\frac{\cos\left(2\pi k\frac{\log(x)-\frac{1}{2}\log(5)}{\log(2)}\right)}{k\sin\left(\frac{\pi k\log(5)}{\log(2)}\right)}\right)\\
&\quad-\frac{1}{8\pi}\sum_{k=1}^{\infty}\left(\frac{\sin\left(2\pi k\frac{\log(x)+\frac{1}{2}\log(3)-\frac{1}{2}\log(5)}{\log(2)}\right)}{k\sin\left(\frac{\pi k\log(3)}{\log(2)}\right)\sin\left(\frac{\pi k\log(5)}{\log(2)}\right)}+\frac{\sin\left(2\pi k\frac{\log(x)-\frac{1}{2}\log(3)+\frac{1}{2}\log(5)}{\log(2)}\right)}{k\sin\left(\frac{\pi k\log(3)}{\log(2)}\right)\sin\left(\frac{\pi k\log(5)}{\log(2)}\right)}\right)\\
\end{split}
\end{displaymath}
\begin{displaymath}
\begin{split}
\hspace{1.5cm}&\quad-\frac{1}{8\pi}\sum_{k=1}^{\infty}\left(\frac{\sin\left(2\pi k\frac{\log(x)+\frac{1}{2}\log(2)-\frac{1}{2}\log(5)}{\log(3)}\right)}{k\sin\left(\frac{\pi k\log(2)}{\log(3)}\right)\sin\left(\frac{\pi k\log(5)}{\log(3)}\right)}+\frac{\sin\left(2\pi k\frac{\log(x)-\frac{1}{2}\log(2)+\frac{1}{2}\log(5)}{\log(3)}\right)}{k\sin\left(\frac{\pi k\log(2)}{\log(3)}\right)\sin\left(\frac{\pi k\log(5)}{\log(3)}\right)}\right)\\
&\quad-\frac{1}{8\pi}\sum_{k=1}^{\infty}\left(\frac{\sin\left(2\pi k\frac{\log(x)+\frac{1}{2}\log(2)-\frac{1}{2}\log(3)}{\log(5)}\right)}{k\sin\left(\frac{\pi k\log(2)}{\log(5)}\right)\sin\left(\frac{\pi k\log(3)}{\log(5)}\right)}+\frac{\sin\left(2\pi k\frac{\log(x)-\frac{1}{2}\log(2)+\frac{1}{2}\log(3)}{\log(5)}\right)}{k\sin\left(\frac{\pi k\log(2)}{\log(5)}\right)\sin\left(\frac{\pi k\log(3)}{\log(5)}\right)}\right)+\frac{1}{2}\chi_{S_{2,3,5}}(x).
\end{split}
\end{displaymath}
\end{corollary}

\noindent This formula converges very rapidly.

\noindent Using this formula for the $5$-Smooth Numbers Counting Function $N_{2,3,5}(x)$, we get the following table:

\begin{table}[htbp]
\begin{center}
\begin{tabular}{| c | c | c | c |}
\hline
$x$&$N_{2,3,5}(x)$&$\text{Formula for }N_{2,3,5}(x)$&$\text{Number of terms $R$}$\\
\hline
\bfseries $1$&$1$&$1.0191146914343678209209456$&$R=3$\\
\hline
\bfseries $10$&$9$&$9.0066388420020729763649195$&$R=11$\\
\hline
\bfseries $10^{2}$&$34$&$34.01798108016701636663657078$&$R=32$\\
\hline
\bfseries $10^{3}$&$86$&$86.01831146911104727455077198$&$R=40$\\
\hline
\bfseries $10^{4}$&$175$&$175.01259815271196528318821070$&$R=52$\\
\hline
\bfseries $10^{5}$&$313$&$313.01116052291470126065468770$&$R=100$\\
\hline
\bfseries $10^{6}$&$507$&$507.04384962202822061525989835$&$R=104$\\
\hline
\bfseries $10^{7}$&$768$&$768.05762686767314864195183397$&$R=110$\\
\hline
\bfseries $10^{8}$&$1105$&$1105.00435666776355760375109758$&$R=260$\\
\hline
\bfseries $10^{9}$&$1530$&$1530.00198789289107971841182114$&$R=300$\\
\hline
\bfseries $10^{10}$&$2053$&$2053.01709151724653660944693303$&$R=306$\\
\hline
\bfseries $10^{10^{2}}$&$1697191$&$1697191.10060827971167051326275935$&$R=20000$\\
\hline
\end{tabular}
\caption{Values of $N_{2,3,5}(x)$}\end{center}
\end{table}

\section{The Formula for the Distribution of the 7-Smooth Numbers}

\noindent Let $a,b,c,d\in\mathbb{N}$ such that $a<b<c<d$ and $\gcd(a,b,c,d)=1$.\\
\noindent For $x\in\mathbb{R}^{+}_{0}$, we define the function $N_{a,b,c,d}(x)$ by
\begin{displaymath}
\begin{split}
N_{a,b,c,d}(x):&=\sum_{\substack{a^{p}b^{q}c^{l}d^{f}\leq x\\p\in\mathbb{N}_{0},q\in\mathbb{N}_{0},l\in\mathbb{N}_{0},f\in\mathbb{N}_{0}}}1.
\end{split}
\end{displaymath}

\noindent We define also
\begin{displaymath}
\begin{split}
S_{a,b,c,d}:&=\left\{a^{p}b^{q}c^{l}d^{f}:p\in\mathbb{N}_{0},q\in\mathbb{N}_{0},l\in\mathbb{N}_{0},f\in\mathbb{N}_{0}\right\},\\
\chi_{S_{a,b,c,d}}(x):&=\begin{cases}1&\text{if $x\in S_{a,b,c,d}$}\\0&\text{if $x\notin S_{a,b,c,d}$}\end{cases}.
\end{split}
\end{displaymath}

\noindent Thus, we have that
\begin{displaymath}
\begin{split}
N_{a,b,c,d}(x)&=\sum_{k=0}^{\left\lfloor\log_{a}(x)\right\rfloor}\sum_{l=0}^{\left\lfloor\log_{b}\left(\frac{x}{a^{k}}\right)\right\rfloor}\sum_{m=0}^{\left\lfloor\log_{c}\left(\frac{x}{a^{k}b^{l}}\right)\right\rfloor}\left(\left\lfloor\log_{d}\left(\frac{x}{a^{k}b^{l}c^{m}}\right)\right\rfloor+1\right).
\end{split}
\end{displaymath}

\noindent We have the following

\begin{theorem}(Formula for $N_{a,b,c,d}(x)$)\\
\noindent For every real number $x\geq1$, we have that
\begin{displaymath}
\begin{split}
N_{a,b,c,d}(x)&=\frac{\log(x)^{4}}{24\log(a)\log(b)\log(c)\log(d)}+\frac{\log(x)^{3}}{12\log(a)\log(b)\log(c)}+\frac{\log(x)^{3}}{12\log(a)\log(b)\log(d)}\\
&\quad+\frac{\log(x)^{3}}{12\log(a)\log(c)\log(d)}+\frac{\log(x)^{3}}{12\log(b)\log(c)\log(d)}+\frac{\log(a)\log(x)^{2}}{24\log(b)\log(c)\log(d)}\\
&\quad+\frac{\log(b)\log(x)^{2}}{24\log(a)\log(c)\log(d)}+\frac{\log(c)\log(x)^{2}}{24\log(a)\log(b)\log(d)}+\frac{\log(d)\log(x)^{2}}{24\log(a)\log(b)\log(c)}\\
&\quad+\frac{\log(x)^{2}}{8\log(a)\log(b)}+\frac{\log(x)^{2}}{8\log(a)\log(c)}+\frac{\log(x)^{2}}{8\log(a)\log(d)}+\frac{\log(x)^{2}}{8\log(b)\log(c)}+\frac{\log(x)^{2}}{8\log(b)\log(d)}\\
&\quad+\frac{\log(x)^{2}}{8\log(c)\log(d)}+\frac{\log(x)}{8\log(a)}+\frac{\log(x)}{8\log(b)}+\frac{\log(x)}{8\log(c)}+\frac{\log(x)}{8\log(d)}+\frac{\log(a)\log(x)}{24\log(b)\log(c)}\\
&\quad+\frac{\log(a)\log(x)}{24\log(b)\log(d)}+\frac{\log(a)\log(x)}{24\log(c)\log(d)}+\frac{\log(b)\log(x)}{24\log(a)\log(c)}+\frac{\log(b)\log(x)}{24\log(a)\log(d)}+\frac{\log(b)\log(x)}{24\log(c)\log(d)}\\
&\quad+\frac{\log(c)\log(x)}{24\log(a)\log(b)}+\frac{\log(c)\log(x)}{24\log(a)\log(d)}+\frac{\log(c)\log(x)}{24\log(b)\log(d)}+\frac{\log(d)\log(x)}{24\log(a)\log(b)}+\frac{\log(d)\log(x)}{24\log(a)\log(c)}\\
&\quad+\frac{\log(d)\log(x)}{24\log(b)\log(c)}+\frac{1}{16}+\frac{\log(a)}{48\log(b)}+\frac{\log(a)}{48\log(c)}+\frac{\log(a)}{48\log(d)}+\frac{\log(b)}{48\log(a)}+\frac{\log(b)}{48\log(c)}\\
&\quad+\frac{\log(b)}{48\log(d)}+\frac{\log(c)}{48\log(a)}+\frac{\log(c)}{48\log(b)}+\frac{\log(c)}{48\log(d)}+\frac{\log(d)}{48\log(a)}+\frac{\log(d)}{48\log(b)}+\frac{\log(d)}{48\log(c)}\\
&\quad+\frac{\log(a)\log(b)}{144\log(c)\log(d)}+\frac{\log(a)\log(c)}{144\log(b)\log(d)}+\frac{\log(a)\log(d)}{144\log(b)\log(c)}+\frac{\log(b)\log(c)}{144\log(a)\log(d)}\\
&\quad+\frac{\log(b)\log(d)}{144\log(a)\log(c)}+\frac{\log(c)\log(d)}{144\log(a)\log(b)}-\frac{\log(a)^{3}}{720\log(b)\log(c)\log(d)}-\frac{\log(b)^{3}}{720\log(a)\log(c)\log(d)}\\
&\quad-\frac{\log(c)^{3}}{720\log(a)\log(b)\log(d)}-\frac{\log(d)^{3}}{720\log(a)\log(b)\log(c)}-\frac{7}{8}B^{*}_{1}\left(\left\{\frac{\log(x)}{\log(a)}\right\}\right)\\
\end{split}
\end{displaymath}
\begin{displaymath}
\begin{split}
\hspace{1.1cm}&\quad-\frac{7}{8}B^{*}_{1}\left(\left\{\frac{\log(x)}{\log(b)}\right\}\right)-\frac{7}{8}B^{*}_{1}\left(\left\{\frac{\log(x)}{\log(c)}\right\}\right)-\frac{7}{8}B^{*}_{1}\left(\left\{\frac{\log(x)}{\log(d)}\right\}\right)\\
&\quad-\frac{1}{8\pi}\sum_{k=1}^{\infty}\left(\frac{\cos\left(2\pi k\frac{\log(x)-\frac{1}{2}\log(b)}{\log(a)}\right)}{k\sin\left(\frac{\pi k\log(b)}{\log(a)}\right)}+\frac{\cos\left(2\pi k\frac{\log(x)-\frac{1}{2}\log(a)}{\log(b)}\right)}{k\sin\left(\frac{\pi k\log(a)}{\log(b)}\right)}\right)\\
&\quad-\frac{1}{8\pi}\sum_{k=1}^{\infty}\left(\frac{\cos\left(2\pi k\frac{\log(x)-\frac{1}{2}\log(c)}{\log(a)}\right)}{k\sin\left(\frac{\pi k\log(c)}{\log(a)}\right)}+\frac{\cos\left(2\pi k\frac{\log(x)-\frac{1}{2}\log(a)}{\log(c)}\right)}{k\sin\left(\frac{\pi k\log(a)}{\log(c)}\right)}\right)\\
&\quad-\frac{1}{8\pi}\sum_{k=1}^{\infty}\left(\frac{\cos\left(2\pi k\frac{\log(x)-\frac{1}{2}\log(d)}{\log(a)}\right)}{k\sin\left(\frac{\pi k\log(d)}{\log(a)}\right)}+\frac{\cos\left(2\pi k\frac{\log(x)-\frac{1}{2}\log(a)}{\log(d)}\right)}{k\sin\left(\frac{\pi k\log(a)}{\log(d)}\right)}\right)\\
&\quad-\frac{1}{8\pi}\sum_{k=1}^{\infty}\left(\frac{\cos\left(2\pi k\frac{\log(x)-\frac{1}{2}\log(c)}{\log(b)}\right)}{k\sin\left(\frac{\pi k\log(c)}{\log(b)}\right)}+\frac{\cos\left(2\pi k\frac{\log(x)-\frac{1}{2}\log(b)}{\log(c)}\right)}{k\sin\left(\frac{\pi k\log(b)}{\log(c)}\right)}\right)\\
&\quad-\frac{1}{8\pi}\sum_{k=1}^{\infty}\left(\frac{\cos\left(2\pi k\frac{\log(x)-\frac{1}{2}\log(d)}{\log(b)}\right)}{k\sin\left(\frac{\pi k\log(d)}{\log(b)}\right)}+\frac{\cos\left(2\pi k\frac{\log(x)-\frac{1}{2}\log(b)}{\log(d)}\right)}{k\sin\left(\frac{\pi k\log(b)}{\log(d)}\right)}\right)\\
&\quad-\frac{1}{8\pi}\sum_{k=1}^{\infty}\left(\frac{\cos\left(2\pi k\frac{\log(x)-\frac{1}{2}\log(d)}{\log(c)}\right)}{k\sin\left(\frac{\pi k\log(d)}{\log(c)}\right)}+\frac{\cos\left(2\pi k\frac{\log(x)-\frac{1}{2}\log(c)}{\log(d)}\right)}{k\sin\left(\frac{\pi k\log(c)}{\log(d)}\right)}\right)\\
&\quad-\frac{1}{16\pi}\sum_{k=1}^{\infty}\left(\frac{\sin\left(2\pi k\frac{\log(x)+\frac{1}{2}\log(b)-\frac{1}{2}\log(c)}{\log(a)}\right)}{k\sin\left(\frac{\pi k\log(b)}{\log(a)}\right)\sin\left(\frac{\pi k\log(c)}{\log(a)}\right)}+\frac{\sin\left(2\pi k\frac{\log(x)-\frac{1}{2}\log(b)+\frac{1}{2}\log(c)}{\log(a)}\right)}{k\sin\left(\frac{\pi k\log(b)}{\log(a)}\right)\sin\left(\frac{\pi k\log(c)}{\log(a)}\right)}\right)\\
&\quad-\frac{1}{16\pi}\sum_{k=1}^{\infty}\left(\frac{\sin\left(2\pi k\frac{\log(x)+\frac{1}{2}\log(c)-\frac{1}{2}\log(d)}{\log(a)}\right)}{k\sin\left(\frac{\pi k\log(c)}{\log(a)}\right)\sin\left(\frac{\pi k\log(d)}{\log(a)}\right)}+\frac{\sin\left(2\pi k\frac{\log(x)-\frac{1}{2}\log(c)+\frac{1}{2}\log(d)}{\log(a)}\right)}{k\sin\left(\frac{\pi k\log(c)}{\log(a)}\right)\sin\left(\frac{\pi k\log(d)}{\log(a)}\right)}\right)\\
&\quad-\frac{1}{16\pi}\sum_{k=1}^{\infty}\left(\frac{\sin\left(2\pi k\frac{\log(x)+\frac{1}{2}\log(b)-\frac{1}{2}\log(d)}{\log(a)}\right)}{k\sin\left(\frac{\pi k\log(b)}{\log(a)}\right)\sin\left(\frac{\pi k\log(d)}{\log(a)}\right)}+\frac{\sin\left(2\pi k\frac{\log(x)-\frac{1}{2}\log(b)+\frac{1}{2}\log(d)}{\log(a)}\right)}{k\sin\left(\frac{\pi k\log(b)}{\log(a)}\right)\sin\left(\frac{\pi k\log(d)}{\log(a)}\right)}\right)\\
&\quad-\frac{1}{16\pi}\sum_{k=1}^{\infty}\left(\frac{\sin\left(2\pi k\frac{\log(x)+\frac{1}{2}\log(a)-\frac{1}{2}\log(c)}{\log(b)}\right)}{k\sin\left(\frac{\pi k\log(a)}{\log(b)}\right)\sin\left(\frac{\pi k\log(c)}{\log(b)}\right)}+\frac{\sin\left(2\pi k\frac{\log(x)-\frac{1}{2}\log(a)+\frac{1}{2}\log(c)}{\log(b)}\right)}{k\sin\left(\frac{\pi k\log(a)}{\log(b)}\right)\sin\left(\frac{\pi k\log(c)}{\log(b)}\right)}\right)\\
&\quad-\frac{1}{16\pi}\sum_{k=1}^{\infty}\left(\frac{\sin\left(2\pi k\frac{\log(x)+\frac{1}{2}\log(c)-\frac{1}{2}\log(d)}{\log(b)}\right)}{k\sin\left(\frac{\pi k\log(c)}{\log(b)}\right)\sin\left(\frac{\pi k\log(d)}{\log(b)}\right)}+\frac{\sin\left(2\pi k\frac{\log(x)-\frac{1}{2}\log(c)+\frac{1}{2}\log(d)}{\log(b)}\right)}{k\sin\left(\frac{\pi k\log(c)}{\log(b)}\right)\sin\left(\frac{\pi k\log(d)}{\log(b)}\right)}\right)\\
\end{split}
\end{displaymath}
\begin{displaymath}
\begin{split}
\hspace{1.75cm}&\quad-\frac{1}{16\pi}\sum_{k=1}^{\infty}\left(\frac{\sin\left(2\pi k\frac{\log(x)+\frac{1}{2}\log(a)-\frac{1}{2}\log(d)}{\log(b)}\right)}{k\sin\left(\frac{\pi k\log(a)}{\log(b)}\right)\sin\left(\frac{\pi k\log(d)}{\log(b)}\right)}+\frac{\sin\left(2\pi k\frac{\log(x)-\frac{1}{2}\log(a)+\frac{1}{2}\log(d)}{\log(b)}\right)}{k\sin\left(\frac{\pi k\log(a)}{\log(b)}\right)\sin\left(\frac{\pi k\log(d)}{\log(b)}\right)}\right)\\
&\quad-\frac{1}{16\pi}\sum_{k=1}^{\infty}\left(\frac{\sin\left(2\pi k\frac{\log(x)+\frac{1}{2}\log(a)-\frac{1}{2}\log(b)}{\log(c)}\right)}{k\sin\left(\frac{\pi k\log(a)}{\log(c)}\right)\sin\left(\frac{\pi k\log(b)}{\log(c)}\right)}+\frac{\sin\left(2\pi k\frac{\log(x)-\frac{1}{2}\log(a)+\frac{1}{2}\log(b)}{\log(c)}\right)}{k\sin\left(\frac{\pi k\log(a)}{\log(c)}\right)\sin\left(\frac{\pi k\log(b)}{\log(c)}\right)}\right)\\
&\quad-\frac{1}{16\pi}\sum_{k=1}^{\infty}\left(\frac{\sin\left(2\pi k\frac{\log(x)+\frac{1}{2}\log(b)-\frac{1}{2}\log(d)}{\log(c)}\right)}{k\sin\left(\frac{\pi k\log(b)}{\log(c)}\right)\sin\left(\frac{\pi k\log(d)}{\log(c)}\right)}+\frac{\sin\left(2\pi k\frac{\log(x)-\frac{1}{2}\log(b)+\frac{1}{2}\log(d)}{\log(c)}\right)}{k\sin\left(\frac{\pi k\log(b)}{\log(c)}\right)\sin\left(\frac{\pi k\log(d)}{\log(c)}\right)}\right)\\
&\quad-\frac{1}{16\pi}\sum_{k=1}^{\infty}\left(\frac{\sin\left(2\pi k\frac{\log(x)+\frac{1}{2}\log(a)-\frac{1}{2}\log(d)}{\log(c)}\right)}{k\sin\left(\frac{\pi k\log(a)}{\log(c)}\right)\sin\left(\frac{\pi k\log(d)}{\log(c)}\right)}+\frac{\sin\left(2\pi k\frac{\log(x)-\frac{1}{2}\log(a)+\frac{1}{2}\log(d)}{\log(c)}\right)}{k\sin\left(\frac{\pi k\log(a)}{\log(c)}\right)\sin\left(\frac{\pi k\log(d)}{\log(c)}\right)}\right)\\
&\quad-\frac{1}{16\pi}\sum_{k=1}^{\infty}\left(\frac{\sin\left(2\pi k\frac{\log(x)+\frac{1}{2}\log(a)-\frac{1}{2}\log(b)}{\log(d)}\right)}{k\sin\left(\frac{\pi k\log(a)}{\log(d)}\right)\sin\left(\frac{\pi k\log(b)}{\log(d)}\right)}+\frac{\sin\left(2\pi k\frac{\log(x)-\frac{1}{2}\log(a)+\frac{1}{2}\log(b)}{\log(d)}\right)}{k\sin\left(\frac{\pi k\log(a)}{\log(d)}\right)\sin\left(\frac{\pi k\log(b)}{\log(d)}\right)}\right)\\
&\quad-\frac{1}{16\pi}\sum_{k=1}^{\infty}\left(\frac{\sin\left(2\pi k\frac{\log(x)+\frac{1}{2}\log(b)-\frac{1}{2}\log(c)}{\log(d)}\right)}{k\sin\left(\frac{\pi k\log(b)}{\log(d)}\right)\sin\left(\frac{\pi k\log(c)}{\log(d)}\right)}+\frac{\sin\left(2\pi k\frac{\log(x)-\frac{1}{2}\log(b)+\frac{1}{2}\log(c)}{\log(d)}\right)}{k\sin\left(\frac{\pi k\log(b)}{\log(d)}\right)\sin\left(\frac{\pi k\log(c)}{\log(d)}\right)}\right)\\
&\quad-\frac{1}{16\pi}\sum_{k=1}^{\infty}\left(\frac{\sin\left(2\pi k\frac{\log(x)+\frac{1}{2}\log(a)-\frac{1}{2}\log(c)}{\log(d)}\right)}{k\sin\left(\frac{\pi k\log(a)}{\log(d)}\right)\sin\left(\frac{\pi k\log(c)}{\log(d)}\right)}+\frac{\sin\left(2\pi k\frac{\log(x)-\frac{1}{2}\log(a)+\frac{1}{2}\log(c)}{\log(d)}\right)}{k\sin\left(\frac{\pi k\log(a)}{\log(d)}\right)\sin\left(\frac{\pi k\log(c)}{\log(d)}\right)}\right)\\
&\quad+\frac{1}{8\pi}\sum_{k=1}^{\infty}\frac{\cos\left(\frac{\pi k\log(b)}{\log(a)}\right)\cos\left(\frac{\pi k\log(c)}{\log(a)}\right)\cos\left(\frac{\pi k\log(d)}{\log(a)}\right)\cos\left(2\pi k\frac{\log(x)}{\log(a)}\right)}{k\sin\left(\frac{\pi k\log(b)}{\log(a)}\right)\sin\left(\frac{\pi k\log(c)}{\log(a)}\right)\sin\left(\frac{\pi k\log(d)}{\log(a)}\right)}\\
&\quad+\frac{1}{8\pi}\sum_{k=1}^{\infty}\frac{\cos\left(\frac{\pi k\log(a)}{\log(b)}\right)\cos\left(\frac{\pi k\log(c)}{\log(b)}\right)\cos\left(\frac{\pi k\log(d)}{\log(b)}\right)\cos\left(2\pi k\frac{\log(x)}{\log(b)}\right)}{k\sin\left(\frac{\pi k\log(a)}{\log(b)}\right)\sin\left(\frac{\pi k\log(c)}{\log(b)}\right)\sin\left(\frac{\pi k\log(d)}{\log(b)}\right)}\\
&\quad+\frac{1}{8\pi}\sum_{k=1}^{\infty}\frac{\cos\left(\frac{\pi k\log(a)}{\log(c)}\right)\cos\left(\frac{\pi k\log(b)}{\log(c)}\right)\cos\left(\frac{\pi k\log(d)}{\log(c)}\right)\cos\left(2\pi k\frac{\log(x)}{\log(c)}\right)}{k\sin\left(\frac{\pi k\log(a)}{\log(c)}\right)\sin\left(\frac{\pi k\log(b)}{\log(c)}\right)\sin\left(\frac{\pi k\log(d)}{\log(c)}\right)}\\
&\quad+\frac{1}{8\pi}\sum_{k=1}^{\infty}\frac{\cos\left(\frac{\pi k\log(a)}{\log(d)}\right)\cos\left(\frac{\pi k\log(b)}{\log(d)}\right)\cos\left(\frac{\pi k\log(c)}{\log(d)}\right)\cos\left(2\pi k\frac{\log(x)}{\log(d)}\right)}{k\sin\left(\frac{\pi k\log(a)}{\log(d)}\right)\sin\left(\frac{\pi k\log(b)}{\log(d)}\right)\sin\left(\frac{\pi k\log(c)}{\log(d)}\right)}+\frac{1}{2}\chi_{S_{a,b,c,d}}(x).
\end{split}
\end{displaymath}
\end{theorem}

\noindent This formula converges again very rapidly.\\
\noindent The series appearing in this formula are all interpreted like before.\\

\noindent Setting $a=2$, $b=3$, $c=5$ and $d=7$, we get for the sequence
\begin{displaymath}
\begin{split}
S_{2,3,5,7}:&=\left\{2^{p}3^{q}5^{l}7^{f}:p\in\mathbb{N}_{0},q\in\mathbb{N}_{0},l\in\mathbb{N}_{0},f\in\mathbb{N}_{0}\right\}\\
&=\left\{1,2,3,4,5,6,7,8,9,10,12,14,15,16,18,20,21,24,25,27,28,30,32,35,36,40,42,45,48,\ldots\right\},
\end{split}
\end{displaymath}
\noindent of $7$-smooth numbers (Humble numbers or "highly composite numbers") \cite{1,2,14}, immediately the following

\begin{corollary}(Formula for the $7$-Smooth Numbers Counting Function $N_{2,3,5,7}(x)$)\\
\noindent For every real number $x\geq1$, we have that
\begin{displaymath}
\begin{split}
N_{2,3,5,7}(x)&=\frac{\log(x)^{4}}{24\log(2)\log(3)\log(5)\log(7)}+\frac{\log(x)^{3}}{12\log(2)\log(3)\log(5)}+\frac{\log(x)^{3}}{12\log(2)\log(3)\log(7)}\\
&\quad+\frac{\log(x)^{3}}{12\log(2)\log(5)\log(7)}+\frac{\log(x)^{3}}{12\log(3)\log(5)\log(7)}+\frac{\log(2)\log(x)^{2}}{24\log(3)\log(5)\log(7)}\\
&\quad+\frac{\log(3)\log(x)^{2}}{24\log(2)\log(5)\log(7)}+\frac{\log(5)\log(x)^{2}}{24\log(2)\log(3)\log(7)}+\frac{\log(7)\log(x)^{2}}{24\log(2)\log(3)\log(5)}\\
&\quad+\frac{\log(x)^{2}}{8\log(2)\log(3)}+\frac{\log(x)^{2}}{8\log(2)\log(5)}+\frac{\log(x)^{2}}{8\log(2)\log(7)}+\frac{\log(x)^{2}}{8\log(3)\log(5)}+\frac{\log(x)^{2}}{8\log(2)\log(7)}\\
&\quad+\frac{\log(x)^{2}}{8\log(5)\log(7)}+\frac{\log(x)}{8\log(2)}+\frac{\log(x)}{8\log(3)}+\frac{\log(x)}{8\log(5)}+\frac{\log(x)}{8\log(7)}+\frac{\log(2)\log(x)}{24\log(3)\log(5)}\\
&\quad+\frac{\log(2)\log(x)}{24\log(3)\log(7)}+\frac{\log(2)\log(x)}{24\log(5)\log(7)}+\frac{\log(3)\log(x)}{24\log(2)\log(5)}+\frac{\log(3)\log(x)}{24\log(2)\log(7)}+\frac{\log(3)\log(x)}{24\log(5)\log(7)}\\
&\quad+\frac{\log(5)\log(x)}{24\log(2)\log(3)}+\frac{\log(5)\log(x)}{24\log(2)\log(7)}+\frac{\log(5)\log(x)}{24\log(3)\log(7)}+\frac{\log(7)\log(x)}{24\log(2)\log(3)}+\frac{\log(7)\log(x)}{24\log(2)\log(5)}\\
&\quad+\frac{\log(7)\log(x)}{24\log(3)\log(5)}+\frac{1}{16}+\frac{\log(2)}{48\log(3)}+\frac{\log(2)}{48\log(5)}+\frac{\log(2)}{48\log(7)}+\frac{\log(3)}{48\log(2)}+\frac{\log(3)}{48\log(5)}\\
&\quad+\frac{\log(3)}{48\log(7)}+\frac{\log(5)}{48\log(2)}+\frac{\log(5)}{48\log(3)}+\frac{\log(5)}{48\log(7)}+\frac{\log(7)}{48\log(2)}+\frac{\log(7)}{48\log(3)}+\frac{\log(7)}{48\log(5)}\\
&\quad+\frac{\log(2)\log(3)}{144\log(5)\log(7)}+\frac{\log(2)\log(5)}{144\log(3)\log(7)}+\frac{\log(2)\log(7)}{144\log(3)\log(5)}+\frac{\log(3)\log(5)}{144\log(2)\log(7)}\\
&\quad+\frac{\log(3)\log(7)}{144\log(2)\log(5)}+\frac{\log(5)\log(7)}{144\log(2)\log(3)}-\frac{\log(2)^{3}}{720\log(3)\log(5)\log(7)}-\frac{\log(3)^{3}}{720\log(2)\log(5)\log(7)}\\
&\quad-\frac{\log(5)^{3}}{720\log(2)\log(3)\log(7)}-\frac{\log(7)^{3}}{720\log(2)\log(3)\log(5)}-\frac{7}{8}B^{*}_{1}\left(\left\{\frac{\log(x)}{\log(2)}\right\}\right)\\
&\quad-\frac{7}{8}B^{*}_{1}\left(\left\{\frac{\log(x)}{\log(3)}\right\}\right)-\frac{7}{8}B^{*}_{1}\left(\left\{\frac{\log(x)}{\log(5)}\right\}\right)-\frac{7}{8}B^{*}_{1}\left(\left\{\frac{\log(x)}{\log(7)}\right\}\right)\\
&\quad-\frac{1}{8\pi}\sum_{k=1}^{\infty}\left(\frac{\cos\left(2\pi k\frac{\log(x)-\frac{1}{2}\log(3)}{\log(2)}\right)}{k\sin\left(\frac{\pi k\log(3)}{\log(2)}\right)}+\frac{\cos\left(2\pi k\frac{\log(x)-\frac{1}{2}\log(2)}{\log(3)}\right)}{k\sin\left(\frac{\pi k\log(2)}{\log(3)}\right)}\right)\\
\end{split}
\end{displaymath}
\begin{displaymath}
\begin{split}
\hspace{1.15cm}&\quad-\frac{1}{8\pi}\sum_{k=1}^{\infty}\left(\frac{\cos\left(2\pi k\frac{\log(x)-\frac{1}{2}\log(5)}{\log(2)}\right)}{k\sin\left(\frac{\pi k\log(5)}{\log(2)}\right)}+\frac{\cos\left(2\pi k\frac{\log(x)-\frac{1}{2}\log(2)}{\log(5)}\right)}{k\sin\left(\frac{\pi k\log(2)}{\log(5)}\right)}\right)\\
&\quad-\frac{1}{8\pi}\sum_{k=1}^{\infty}\left(\frac{\cos\left(2\pi k\frac{\log(x)-\frac{1}{2}\log(7)}{\log(2)}\right)}{k\sin\left(\frac{\pi k\log(7)}{\log(2)}\right)}+\frac{\cos\left(2\pi k\frac{\log(x)-\frac{1}{2}\log(2)}{\log(7)}\right)}{k\sin\left(\frac{\pi k\log(2)}{\log(7)}\right)}\right)\\
&\quad-\frac{1}{8\pi}\sum_{k=1}^{\infty}\left(\frac{\cos\left(2\pi k\frac{\log(x)-\frac{1}{2}\log(5)}{\log(3)}\right)}{k\sin\left(\frac{\pi k\log(5)}{\log(3)}\right)}+\frac{\cos\left(2\pi k\frac{\log(x)-\frac{1}{2}\log(3)}{\log(5)}\right)}{k\sin\left(\frac{\pi k\log(3)}{\log(5)}\right)}\right)\\
&\quad-\frac{1}{8\pi}\sum_{k=1}^{\infty}\left(\frac{\cos\left(2\pi k\frac{\log(x)-\frac{1}{2}\log(7)}{\log(3)}\right)}{k\sin\left(\frac{\pi k\log(7)}{\log(3)}\right)}+\frac{\cos\left(2\pi k\frac{\log(x)-\frac{1}{2}\log(3)}{\log(7)}\right)}{k\sin\left(\frac{\pi k\log(3)}{\log(7)}\right)}\right)\\
&\quad-\frac{1}{8\pi}\sum_{k=1}^{\infty}\left(\frac{\cos\left(2\pi k\frac{\log(x)-\frac{1}{2}\log(7)}{\log(5)}\right)}{k\sin\left(\frac{\pi k\log(7)}{\log(5)}\right)}+\frac{\cos\left(2\pi k\frac{\log(x)-\frac{1}{2}\log(5)}{\log(7)}\right)}{k\sin\left(\frac{\pi k\log(5)}{\log(7)}\right)}\right)\\
&\quad-\frac{1}{16\pi}\sum_{k=1}^{\infty}\left(\frac{\sin\left(2\pi k\frac{\log(x)+\frac{1}{2}\log(3)-\frac{1}{2}\log(5)}{\log(2)}\right)}{k\sin\left(\frac{\pi k\log(3)}{\log(2)}\right)\sin\left(\frac{\pi k\log(5)}{\log(2)}\right)}+\frac{\sin\left(2\pi k\frac{\log(x)-\frac{1}{2}\log(3)+\frac{1}{2}\log(5)}{\log(2)}\right)}{k\sin\left(\frac{\pi k\log(3)}{\log(2)}\right)\sin\left(\frac{\pi k\log(5)}{\log(2)}\right)}\right)\\
&\quad-\frac{1}{16\pi}\sum_{k=1}^{\infty}\left(\frac{\sin\left(2\pi k\frac{\log(x)+\frac{1}{2}\log(5)-\frac{1}{2}\log(7)}{\log(2)}\right)}{k\sin\left(\frac{\pi k\log(5)}{\log(2)}\right)\sin\left(\frac{\pi k\log(7)}{\log(2)}\right)}+\frac{\sin\left(2\pi k\frac{\log(x)-\frac{1}{2}\log(5)+\frac{1}{2}\log(7)}{\log(2)}\right)}{k\sin\left(\frac{\pi k\log(5)}{\log(2)}\right)\sin\left(\frac{\pi k\log(7)}{\log(2)}\right)}\right)\\
&\quad-\frac{1}{16\pi}\sum_{k=1}^{\infty}\left(\frac{\sin\left(2\pi k\frac{\log(x)+\frac{1}{2}\log(3)-\frac{1}{2}\log(7)}{\log(2)}\right)}{k\sin\left(\frac{\pi k\log(3)}{\log(2)}\right)\sin\left(\frac{\pi k\log(7)}{\log(2)}\right)}+\frac{\sin\left(2\pi k\frac{\log(x)-\frac{1}{2}\log(3)+\frac{1}{2}\log(7)}{\log(2)}\right)}{k\sin\left(\frac{\pi k\log(3)}{\log(2)}\right)\sin\left(\frac{\pi k\log(7)}{\log(2)}\right)}\right)\\
&\quad-\frac{1}{16\pi}\sum_{k=1}^{\infty}\left(\frac{\sin\left(2\pi k\frac{\log(x)+\frac{1}{2}\log(2)-\frac{1}{2}\log(5)}{\log(3)}\right)}{k\sin\left(\frac{\pi k\log(2)}{\log(3)}\right)\sin\left(\frac{\pi k\log(5)}{\log(3)}\right)}+\frac{\sin\left(2\pi k\frac{\log(x)-\frac{1}{2}\log(2)+\frac{1}{2}\log(5)}{\log(3)}\right)}{k\sin\left(\frac{\pi k\log(2)}{\log(3)}\right)\sin\left(\frac{\pi k\log(5)}{\log(3)}\right)}\right)\\
&\quad-\frac{1}{16\pi}\sum_{k=1}^{\infty}\left(\frac{\sin\left(2\pi k\frac{\log(x)+\frac{1}{2}\log(5)-\frac{1}{2}\log(7)}{\log(3)}\right)}{k\sin\left(\frac{\pi k\log(5)}{\log(3)}\right)\sin\left(\frac{\pi k\log(7)}{\log(3)}\right)}+\frac{\sin\left(2\pi k\frac{\log(x)-\frac{1}{2}\log(5)+\frac{1}{2}\log(7)}{\log(3)}\right)}{k\sin\left(\frac{\pi k\log(5)}{\log(3)}\right)\sin\left(\frac{\pi k\log(7)}{\log(3)}\right)}\right)\\
&\quad-\frac{1}{16\pi}\sum_{k=1}^{\infty}\left(\frac{\sin\left(2\pi k\frac{\log(x)+\frac{1}{2}\log(2)-\frac{1}{2}\log(7)}{\log(3)}\right)}{k\sin\left(\frac{\pi k\log(2)}{\log(3)}\right)\sin\left(\frac{\pi k\log(7)}{\log(3)}\right)}+\frac{\sin\left(2\pi k\frac{\log(x)-\frac{1}{2}\log(2)+\frac{1}{2}\log(7)}{\log(3)}\right)}{k\sin\left(\frac{\pi k\log(2)}{\log(3)}\right)\sin\left(\frac{\pi k\log(7)}{\log(3)}\right)}\right)\\
&\quad-\frac{1}{16\pi}\sum_{k=1}^{\infty}\left(\frac{\sin\left(2\pi k\frac{\log(x)+\frac{1}{2}\log(2)-\frac{1}{2}\log(3)}{\log(5)}\right)}{k\sin\left(\frac{\pi k\log(2)}{\log(5)}\right)\sin\left(\frac{\pi k\log(3)}{\log(5)}\right)}+\frac{\sin\left(2\pi k\frac{\log(x)-\frac{1}{2}\log(2)+\frac{1}{2}\log(3)}{\log(5)}\right)}{k\sin\left(\frac{\pi k\log(2)}{\log(5)}\right)\sin\left(\frac{\pi k\log(3)}{\log(5)}\right)}\right)\\
\end{split}
\end{displaymath}
\begin{displaymath}
\begin{split}
\hspace{1.75cm}&\quad-\frac{1}{16\pi}\sum_{k=1}^{\infty}\left(\frac{\sin\left(2\pi k\frac{\log(x)+\frac{1}{2}\log(3)-\frac{1}{2}\log(7)}{\log(5)}\right)}{k\sin\left(\frac{\pi k\log(3)}{\log(5)}\right)\sin\left(\frac{\pi k\log(7)}{\log(5)}\right)}+\frac{\sin\left(2\pi k\frac{\log(x)-\frac{1}{2}\log(3)+\frac{1}{2}\log(7)}{\log(5)}\right)}{k\sin\left(\frac{\pi k\log(3)}{\log(5)}\right)\sin\left(\frac{\pi k\log(7)}{\log(5)}\right)}\right)\\
&\quad-\frac{1}{16\pi}\sum_{k=1}^{\infty}\left(\frac{\sin\left(2\pi k\frac{\log(x)+\frac{1}{2}\log(2)-\frac{1}{2}\log(7)}{\log(5)}\right)}{k\sin\left(\frac{\pi k\log(2)}{\log(5)}\right)\sin\left(\frac{\pi k\log(7)}{\log(5)}\right)}+\frac{\sin\left(2\pi k\frac{\log(x)-\frac{1}{2}\log(2)+\frac{1}{2}\log(7)}{\log(5)}\right)}{k\sin\left(\frac{\pi k\log(2)}{\log(5)}\right)\sin\left(\frac{\pi k\log(7)}{\log(5)}\right)}\right)\\
&\quad-\frac{1}{16\pi}\sum_{k=1}^{\infty}\left(\frac{\sin\left(2\pi k\frac{\log(x)+\frac{1}{2}\log(2)-\frac{1}{2}\log(3)}{\log(7)}\right)}{k\sin\left(\frac{\pi k\log(2)}{\log(7)}\right)\sin\left(\frac{\pi k\log(3)}{\log(7)}\right)}+\frac{\sin\left(2\pi k\frac{\log(x)-\frac{1}{2}\log(2)+\frac{1}{2}\log(3)}{\log(7)}\right)}{k\sin\left(\frac{\pi k\log(2)}{\log(7)}\right)\sin\left(\frac{\pi k\log(3)}{\log(7)}\right)}\right)\\
&\quad-\frac{1}{16\pi}\sum_{k=1}^{\infty}\left(\frac{\sin\left(2\pi k\frac{\log(x)+\frac{1}{2}\log(3)-\frac{1}{2}\log(5)}{\log(7)}\right)}{k\sin\left(\frac{\pi k\log(3)}{\log(7)}\right)\sin\left(\frac{\pi k\log(5)}{\log(7)}\right)}+\frac{\sin\left(2\pi k\frac{\log(x)-\frac{1}{2}\log(3)+\frac{1}{2}\log(5)}{\log(7)}\right)}{k\sin\left(\frac{\pi k\log(3)}{\log(7)}\right)\sin\left(\frac{\pi k\log(5)}{\log(7)}\right)}\right)\\
&\quad-\frac{1}{16\pi}\sum_{k=1}^{\infty}\left(\frac{\sin\left(2\pi k\frac{\log(x)+\frac{1}{2}\log(2)-\frac{1}{2}\log(5)}{\log(7)}\right)}{k\sin\left(\frac{\pi k\log(2)}{\log(7)}\right)\sin\left(\frac{\pi k\log(5)}{\log(7)}\right)}+\frac{\sin\left(2\pi k\frac{\log(x)-\frac{1}{2}\log(2)+\frac{1}{2}\log(5)}{\log(7)}\right)}{k\sin\left(\frac{\pi k\log(2)}{\log(7)}\right)\sin\left(\frac{\pi k\log(5)}{\log(7)}\right)}\right)\\
&\quad+\frac{1}{8\pi}\sum_{k=1}^{\infty}\frac{\cos\left(\frac{\pi k\log(3)}{\log(2)}\right)\cos\left(\frac{\pi k\log(5)}{\log(2)}\right)\cos\left(\frac{\pi k\log(7)}{\log(2)}\right)\cos\left(2\pi k\frac{\log(x)}{\log(2)}\right)}{k\sin\left(\frac{\pi k\log(3)}{\log(2)}\right)\sin\left(\frac{\pi k\log(5)}{\log(2)}\right)\sin\left(\frac{\pi k\log(7)}{\log(2)}\right)}\\
&\quad+\frac{1}{8\pi}\sum_{k=1}^{\infty}\frac{\cos\left(\frac{\pi k\log(2)}{\log(3)}\right)\cos\left(\frac{\pi k\log(5)}{\log(3)}\right)\cos\left(\frac{\pi k\log(7)}{\log(3)}\right)\cos\left(2\pi k\frac{\log(x)}{\log(3)}\right)}{k\sin\left(\frac{\pi k\log(2)}{\log(3)}\right)\sin\left(\frac{\pi k\log(5)}{\log(3)}\right)\sin\left(\frac{\pi k\log(7)}{\log(3)}\right)}\\
&\quad+\frac{1}{8\pi}\sum_{k=1}^{\infty}\frac{\cos\left(\frac{\pi k\log(2)}{\log(5)}\right)\cos\left(\frac{\pi k\log(3)}{\log(5)}\right)\cos\left(\frac{\pi k\log(7)}{\log(5)}\right)\cos\left(2\pi k\frac{\log(x)}{\log(5)}\right)}{k\sin\left(\frac{\pi k\log(2)}{\log(5)}\right)\sin\left(\frac{\pi k\log(3)}{\log(5)}\right)\sin\left(\frac{\pi k\log(7)}{\log(5)}\right)}\\
&\quad+\frac{1}{8\pi}\sum_{k=1}^{\infty}\frac{\cos\left(\frac{\pi k\log(2)}{\log(7)}\right)\cos\left(\frac{\pi k\log(3)}{\log(7)}\right)\cos\left(\frac{\pi k\log(5)}{\log(7)}\right)\cos\left(2\pi k\frac{\log(x)}{\log(7)}\right)}{k\sin\left(\frac{\pi k\log(2)}{\log(7)}\right)\sin\left(\frac{\pi k\log(3)}{\log(7)}\right)\sin\left(\frac{\pi k\log(5)}{\log(7)}\right)}+\frac{1}{2}\chi_{S_{2,3,5,7}}(x).
\end{split}
\end{displaymath}
\end{corollary}

\noindent This formula converges very rapidly.\\
\noindent Every series is interpreted as mentioned above.

\newpage

\noindent Using this formula for the $7$-Smooth Numbers Counting Function $N_{2,3,5,7}(x)$, we get the following table:

\begin{table}[htbp]
\begin{center}
\begin{tabular}{| c | c | c | c |}
\hline
$x$&$N_{2,3,5,7}(x)$&$\text{Formula for }N_{2,3,5,7}(x)$&$\text{Number of terms $R$}$\\
\hline
\bfseries $1$&$1$&$1.030388812940249824617233653730019551$&$R=3$\\
\hline
\bfseries $10$&$10$&$10.01263249440259984789405319823431872556$&$R=3$\\
\hline
\bfseries $10^{2}$&$46$&$46.03668521491726375130238293886497852216$&$R=20$\\
\hline
\bfseries $10^{3}$&$141$&$141.01285390547424275647701138240776403195$&$R=80$\\
\hline
\bfseries $10^{4}$&$338$&$338.0186997720522261698185344005048234745$&$R=80$\\
\hline
\bfseries $10^{5}$&$694$&$694.00540895426731024839939099335158382934$&$R=100$\\
\hline
\bfseries $10^{6}$&$1273$&$1273.02115574787663113791230619711970129327$&$R=1500$\\
\hline
\bfseries $10^{7}$&$2155$&$2155.01133325568473975698180880511876853632$&$R=1500$\\
\hline
\bfseries $10^{8}$&$3427$&$3427.01611847162744035197962908126411814549$&$R=1500$\\
\hline
\bfseries $10^{9}$&$5194$&$5194.03771424320772544603355297308020543638$&$R=1600$\\
\hline
\bfseries $10^{10}$&$7575$&$7575.01767118495435682818874877606239707862$&$R=9000$\\
\hline
\end{tabular}
\caption{Values of $N_{2,3,5,7}(x)$}\end{center}
\end{table}

\section{The Formula for the Distribution of all Smooth\\
Numbers}

\noindent Let $a_{1},a_{2},a_{3},\ldots,a_{n}\in\mathbb{N}$ such that $a_{1}<a_{2}<a_{3}<\ldots<a_{n}$ and $\gcd(a_{1},a_{2},a_{3},\ldots,a_{n})=1$.\\
\noindent For $x\in\mathbb{R}^{+}_{0}$, we define the function $N_{a_{1},a_{2},a_{3},\ldots,a_{n}}(x)$ by
\begin{displaymath}
\begin{split}
N_{a_{1},a_{2},a_{3},\ldots,a_{n}}(x):&=\sum_{\substack{a_{1}^{q_{1}}a_{2}^{q_{2}}a_{3}^{q_{3}}\cdots a_{n}^{q_{n}}\leq x\\q_{1}\in\mathbb{N}_{0},q_{2}\in\mathbb{N}_{0},q_{3}\in\mathbb{N}_{0},\ldots,q_{n}\in\mathbb{N}_{0}}}1.
\end{split}
\end{displaymath}

\noindent We define also
\begin{displaymath}
\begin{split}
S_{a_{1},a_{2},a_{3},\ldots,a_{n}}:&=\left\{a_{1}^{q_{1}}a_{2}^{q_{2}}a_{3}^{q_{3}}\cdots a_{n}^{q_{n}}:q_{1}\in\mathbb{N}_{0},q_{2}\in\mathbb{N}_{0},q_{3}\in\mathbb{N}_{0},\ldots,q_{n}\in\mathbb{N}_{0}\right\},\\
\chi_{S_{a_{1},a_{2},a_{3},\ldots,a_{n}}}(x):&=\begin{cases}1&\text{if $x\in S_{a_{1},a_{2},a_{3},\ldots,a_{n}}$}\\0&\text{if $x\notin S_{a_{1},a_{2},a_{3},\ldots,a_{n}}$}\end{cases}.
\end{split}
\end{displaymath}

\noindent Thus, we have that
\begin{displaymath}
\begin{split}
N_{a_{1},a_{2},a_{3},\ldots,a_{n}}(x)&=\sum_{k_{1}=0}^{\left\lfloor\log_{a_{1}}(x)\right\rfloor}\sum_{k_{2}=0}^{\left\lfloor\log_{a_{2}}\left(\frac{x}{a_{1}^{k_{1}}}\right)\right\rfloor}\ldots\sum_{k_{n-1}=0}^{\left\lfloor\log_{a_{n-1}}\left(\frac{x}{a_{1}^{k_{1}}a_{2}^{k_{2}}\cdots a_{n-2}^{k_{n-2}}}\right)\right\rfloor}\left(\left\lfloor\log_{a_{n}}\left(\frac{x}{a_{1}^{k_{1}}a_{2}^{k_{2}}a_{3}^{k_{3}}\cdots a_{n-1}^{k_{n-1}}}\right)\right\rfloor+1\right).
\end{split}
\end{displaymath}

\noindent Expressions of this form for $N_{a_{1},a_{2},a_{3},\ldots,a_{n}}(x)$ are called "Klauder-Ness Expressions" \cite{15,16}.

\noindent We have the following

\begin{theorem}(Formula for $N_{a_{1},a_{2},a_{3},\ldots,a_{n}}(x)$)\\
\noindent For every real number $x\geq1$, we have that
\begin{displaymath}
\begin{split}
N_{a_{1},a_{2},a_{3},\ldots,a_{n}}(x)&=\Res_{s=0}\left(\frac{x^{s}}{s\prod_{k=1}^{n}\left(1-\frac{1}{a_{k}^{s}}\right)}\right)-\frac{1}{2^{n-1}}\sum_{k=1}^{n}B^{*}_{1}\left(\left\{\frac{\log(x)}{\log(a_{k})}\right\}\right)\\
&\quad+\frac{1}{2^{n-1}\pi}\sum_{m=1}^{n}\sum_{r=1}^{n-1}\sum_{\substack{
i_{1}<i_{2}<i_{3}<\ldots<i_{r}\\
\{i_{1},i_{2},i_{3},\ldots,i_{r}\}\\
\subset\{a_{1},a_{2},a_{3},\ldots,\widehat{a_{m}},\ldots,a_{n}\}
}}\sum_{k=1}^{\infty}\frac{\sin\left(2\pi k\frac{\log(x)}{\log(a_{m})}-\frac{\pi r}{2}\right)}{k}\prod_{l=1}^{r}\cot\left(\frac{\pi k\log(i_{l})}{\log(a_{m})}\right)\\
&\quad+\frac{1}{2}\chi_{S_{a_{1},a_{2},a_{3},\ldots,a_{n}}}(x),
\end{split}
\end{displaymath}
\noindent where the series are to be interpreted as meaning\begin{displaymath}
\begin{split}
&\sum_{k=1}^{\infty}\frac{\sin\left(2\pi k\frac{\log(x)}{\log(a_{m})}-\frac{\pi r}{2}\right)}{k}\prod_{l=1}^{r}\cot\left(\frac{\pi k\log(i_{l})}{\log(a_{m})}\right)\\
=&\lim_{R\rightarrow\infty}\left(\sum_{k=1}^{\left\lfloor R\log(a_{m})\right\rfloor}\frac{\sin\left(2\pi k\frac{\log(x)}{\log(a_{m})}-\frac{\pi r}{2}\right)}{k}\prod_{l=1}^{r}\cot\left(\frac{\pi k\log(i_{l})}{\log(a_{m})}\right)\right),
\end{split}
\end{displaymath}
\noindent when $R\rightarrow\infty$ in an appropriate manner.
\end{theorem}

\noindent This formula converges again very rapidly.\\

\begin{proof}
We have that
\begin{displaymath}
\begin{split}
\sum_{k=1}^{\infty}\frac{\chi_{S_{a_{1},a_{2},a_{3},\ldots,a_{n}}}(k)}{k^{s}}&=\prod_{k=1}^{n}\left(\sum_{m=0}^{\infty}\frac{1}{a_{k}^{ms}}\right)\\
&=\prod_{k=1}^{n}\frac{1}{1-e^{-\log(a_{k})s}}.
\end{split}
\end{displaymath}
Therefore, by Perron's formula, we get that
\begin{displaymath}
\begin{split}
N_{a_{1},a_{2},a_{3},\ldots,a_{n}}(x)&=\frac{1}{2\pi i}\int_{\gamma}\left(\prod_{k=1}^{n}\frac{1}{1-e^{-\log(a_{k})s}}\right)\frac{x^{s}}{s}ds\\
&=\Res_{s=0}\left(\frac{x^{s}}{s\prod_{k=1}^{n}\left(1-\frac{1}{a_{k}^{s}}\right)}\right)\\
&\quad+\sum_{m=1}^{n}\sum_{k=1}^{\infty}\left(\Res_{s=\frac{2\pi ik}{\log(a_{m})}}\left(\frac{x^{s}}{s\prod_{k=1}^{n}\left(1-\frac{1}{a_{k}^{s}}\right)}\right)+\Res_{s=-\frac{2\pi ik}{\log(a_{m})}}\left(\frac{x^{s}}{s\prod_{k=1}^{n}\left(1-\frac{1}{a_{k}^{s}}\right)}\right)\right)\\
&\quad+\frac{1}{2}\chi_{S_{a_{1},a_{2},a_{3},\ldots,a_{n}}}(x),
\end{split}
\end{displaymath}
where $\gamma=\text{line from $1-i\infty$ to $1+i\infty$}$.\\
Using the relation
\begin{displaymath}
\begin{split}
\lim_{s\rightarrow\pm\frac{2\pi ik}{\log(a_{m})}}\left(\frac{s\mp\frac{2\pi ik}{\log(a_{m})}}{1-e^{-\log(a_{m})s}}\right)&=\frac{1}{\log(a_{m})}\;\;\forall k\in\mathbb{N},
\end{split}
\end{displaymath}
we can compute the "$a_{m}$ -Residues" to
\begin{displaymath}
\begin{split}
\Res_{s=\frac{2\pi ik}{\log(a_{m})}}\left(\frac{x^{s}}{s\prod_{k=1}^{n}\left(1-\frac{1}{a_{k}^{s}}\right)}\right)&=-\frac{i\prod_{\substack{l=1\\l\neq m}}^{n}a_{l}^{\frac{2\pi ik}{\log(a_{m})}}x^{\frac{2\pi ik}{\log(a_{m})}}}{2\pi k\prod_{\substack{l=1\\l\neq m}}^{n}\left(a_{l}^{\frac{2\pi ik}{\log(a_{m})}}-1\right)}\;\;\text{for all $k\in\mathbb{N}$}\\
\Res_{s=-\frac{2\pi ik}{\log(a_{m})}}\left(\frac{x^{s}}{s\prod_{k=1}^{n}\left(1-\frac{1}{a_{k}^{s}}\right)}\right)&=-\frac{(-1)^{n}ix^{-\frac{2\pi ik}{\log(a_{m})}}}{2\pi k\prod_{\substack{l=1\\l\neq m}}^{n}\left(a_{l}^{\frac{2\pi ik}{\log(a_{m})}}-1\right)}\;\;\text{for all $k\in\mathbb{N}$}.
\end{split}
\end{displaymath}
Using the relations
\begin{displaymath}
\begin{split}
\sin\left(2\pi k\frac{\log(x)}{\log(a_{m})}\right)&=\frac{1}{2}ix^{-\frac{2\pi ik}{\log(a_{m})}}-\frac{1}{2}ix^{\frac{2\pi ik}{\log(a_{m})}}\\
\cos\left(2\pi k\frac{\log(x)}{\log(a_{m})}\right)&=\frac{1}{2}x^{-\frac{2\pi ik}{\log(a_{m})}}+\frac{1}{2}x^{\frac{2\pi ik}{\log(a_{m})}}\\
\cot\left(\frac{\pi k\log(i_{l})}{\log(a_{m})}\right)&=i\frac{i_{l}^{\frac{2\pi ik}{\log(a_{m})}}+1}{i_{l}^{\frac{2\pi ik}{\log(a_{m})}}-1}
\end{split}
\end{displaymath}
\begin{displaymath}
\begin{split}
\sin\left(2\pi k\frac{\log(x)}{\log(a_{m})}-\frac{\pi r}{2}\right)&=\begin{cases}(-1)^{\frac{r}{2}}\sin\left(2\pi k\frac{\log(x)}{\log(a_{m})}\right),&\text{if $r=\text{even}$}\\(-1)^{\frac{r+1}{2}}\cos\left(2\pi k\frac{\log(x)}{\log(a_{m})}\right),&\text{if $r=\text{odd}$}
\end{cases}\\
&=\begin{cases}(-1)^{\frac{r}{2}}\left(\frac{1}{2}ix^{-\frac{2\pi ik}{\log(a_{m})}}-\frac{1}{2}ix^{\frac{2\pi ik}{\log(a_{m})}}\right),&\text{if $r=\text{even}$}\\(-1)^{\frac{r+1}{2}}\left(\frac{1}{2}x^{-\frac{2\pi ik}{\log(a_{m})}}+\frac{1}{2}x^{\frac{2\pi ik}{\log(a_{m})}}\right),&\text{if $r=\text{odd}$}\end{cases}
\end{split}
\end{displaymath}
and
\begin{displaymath}
\begin{split}
\sin\left(2\pi k\frac{\log(x)}{\log(a_{m})}-\frac{\pi r}{2}\right)&=(-1)^{\frac{1}{2}(r\;\text{mod $4$}+r\;\text{mod $2$})}\left(\frac{1}{2}i^{(r+1)\text{mod $2$}}x^{-\frac{2\pi ik}{\log(a_{m})}}+(-1)^{r+1}\frac{1}{2}i^{(r+1)\text{mod $2$}}x^{\frac{2\pi ik}{\log(a_{m})}}\right),
\end{split}
\end{displaymath}
we establish (by expanding everything out) that
\begin{displaymath}
\begin{split}
&\sum_{r=1}^{n-1}\sum_{\substack{
i_{1}<i_{2}<i_{3}<\ldots<i_{r}\\
\{i_{1},i_{2},i_{3},\ldots,i_{r}\}\\
\subset\{a_{1},a_{2},a_{3},\ldots,\widehat{a_{m}},\ldots,a_{n}\}
}}\sin\left(2\pi k\frac{\log(x)}{\log(a_{m})}-\frac{\pi r}{2}\right)\prod_{l=1}^{r}\cot\left(\frac{\pi k\log(i_{l})}{\log(a_{m})}\right)+\sin\left(2\pi k\frac{\log(x)}{\log(a_{m})}\right)\\
&=\sum_{r=1}^{n-1}\sum_{\substack{
i_{1}<i_{2}<i_{3}<\ldots<i_{r}\\
\{i_{1},i_{2},i_{3},\ldots,i_{r}\}\\
\subset\{a_{1},a_{2},a_{3},\ldots,\widehat{a_{m}},\ldots,a_{n}\}
}}(-1)^{\frac{1}{2}(r\;\text{mod $4$}+r\;\text{mod $2$})}\left(\frac{1}{2}i^{(r+1)\text{mod $2$}}x^{-\frac{2\pi ik}{\log(a_{m})}}+(-1)^{r+1}\frac{1}{2}i^{(r+1)\text{mod $2$}}x^{\frac{2\pi ik}{\log(a_{m})}}\right)\\
&\quad\cdot\prod_{l=1}^{r}\left(i\frac{i_{l}^{\frac{2\pi ik}{\log(a_{m})}}+1}{i_{l}^{\frac{2\pi ik}{\log(a_{m})}}-1}\right)+\left(\frac{1}{2}ix^{-\frac{2\pi ik}{\log(a_{m})}}-\frac{1}{2}ix^{\frac{2\pi ik}{\log(a_{m})}}\right)\\
&=\sum_{r=1}^{n-1}\sum_{\substack{
i_{1}<i_{2}<i_{3}<\ldots<i_{r}\\
\{i_{1},i_{2},i_{3},\ldots,i_{r}\}\\
\subset\{a_{1},a_{2},a_{3},\ldots,\widehat{a_{m}},\ldots,a_{n}\}
}}(-1)^{\frac{1}{2}(r\;\text{mod $4$}+r\;\text{mod $2$})}\frac{1}{2}i^{r+(r+1)\text{mod $2$}}\prod_{l=1}^{r}\left(\frac{i_{l}^{\frac{2\pi ik}{\log(a_{m})}}+1}{i_{l}^{\frac{2\pi ik}{\log(a_{m})}}-1}\right)x^{-\frac{2\pi ik}{\log(a_{m})}}+\frac{1}{2}ix^{-\frac{2\pi ik}{\log(a_{m})}}\\
&\quad+\sum_{r=1}^{n-1}\sum_{\substack{
i_{1}<i_{2}<i_{3}<\ldots<i_{r}\\
\{i_{1},i_{2},i_{3},\ldots,i_{r}\}\\
\subset\{a_{1},a_{2},a_{3},\ldots,\widehat{a_{m}},\ldots,a_{n}\}
}}(-1)^{r+1+\frac{1}{2}(r\;\text{mod $4$}+r\;\text{mod $2$})}\frac{1}{2}i^{r+(r+1)\text{mod $2$}}\prod_{l=1}^{r}\left(\frac{i_{l}^{\frac{2\pi ik}{\log(a_{m})}}+1}{i_{l}^{\frac{2\pi ik}{\log(a_{m})}}-1}\right)x^{\frac{2\pi ik}{\log(a_{m})}}-\frac{1}{2}ix^{\frac{2\pi ik}{\log(a_{m})}}\\
&=(-1)^{n-1}\frac{1}{2}ix^{-\frac{2\pi ik}{\log(a_{m})}}\left(\frac{\sum_{\substack{\{\epsilon_{1},\epsilon_{2},\epsilon_{3},\ldots,\epsilon_{n-1}\}\\
\subset\{\pm1,\pm1,\pm1,\ldots,\pm1\}}}\prod_{\substack{l=1\\l\neq m}}^{n}\epsilon_{k}\left(a_{l}^{\frac{2\pi ik}{\log(a_{m})}}+\epsilon_{k}\right)}{\prod_{\substack{l=1\\l\neq m}}^{n}\left(a_{l}^{\frac{2\pi ik}{\log(a_{m})}}-1\right)}\right)\\
&\quad-\frac{1}{2}ix^{\frac{2\pi ik}{\log(a_{m})}}\left(\frac{\sum_{\substack{\{\epsilon_{1},\epsilon_{2},\epsilon_{3},\ldots,\epsilon_{n-1}\}\\
\subset\{\pm1,\pm1,\pm1,\ldots,\pm1\}}}\prod_{\substack{l=1\\l\neq m}}^{n}\left(a_{l}^{\frac{2\pi ik}{\log(a_{m})}}+\epsilon_{k}\right)}{\prod_{\substack{l=1\\l\neq m}}^{n}\left(a_{l}^{\frac{2\pi ik}{\log(a_{m})}}-1\right)}\right)\\
\end{split}
\end{displaymath}
\begin{displaymath}
\begin{split}
&=(-1)^{n-1}\frac{1}{2}ix^{-\frac{2\pi ik}{\log(a_{m})}}\left(\frac{S_{1}\left(a_{1}^{\frac{2\pi ik}{\log(a_{m})}},a_{2}^{\frac{2\pi ik}{\log(a_{m})}},a_{3}^{\frac{2\pi ik}{\log(a_{m})}},\ldots,\widehat{a_{m}^{\frac{2\pi ik}{\log(a_{m})}}},\ldots,a_{n}^{\frac{2\pi ik}{\log(a_{m})}}\right)}{\prod_{\substack{l=1\\l\neq m}}^{n}\left(a_{l}^{\frac{2\pi ik}{\log(a_{m})}}-1\right)}\right)\\
&\quad-\frac{1}{2}ix^{\frac{2\pi ik}{\log(a_{m})}}\left(\frac{S_{2}\left(a_{1}^{\frac{2\pi ik}{\log(a_{m})}},a_{2}^{\frac{2\pi ik}{\log(a_{m})}},a_{3}^{\frac{2\pi ik}{\log(a_{m})}},\ldots,\widehat{a_{m}^{\frac{2\pi ik}{\log(a_{m})}}},\ldots,a_{n}^{\frac{2\pi ik}{\log(a_{m})}}\right)}{\prod_{\substack{l=1\\l\neq m}}^{n}\left(a_{l}^{\frac{2\pi ik}{\log(a_{m})}}-1\right)}\right)\\
&=(-1)^{n-1}\frac{1}{2}ix^{-\frac{2\pi ik}{\log(a_{m})}}\left(\frac{2^{n-1}}{\prod_{\substack{l=1\\l\neq m}}^{n}\left(a_{l}^{\frac{2\pi ik}{\log(a_{m})}}-1\right)}\right)-\frac{1}{2}ix^{\frac{2\pi ik}{\log(a_{m})}}\left(\frac{2^{n-1}\prod_{\substack{l=1\\l\neq m}}^{n}a_{l}^{\frac{2\pi ik}{\log(a_{m})}}}{\prod_{\substack{l=1\\l\neq m}}^{n}\left(a_{l}^{\frac{2\pi ik}{\log(a_{m})}}-1\right)}\right)\\
&=(-1)^{n+1}\frac{2^{n-2}i\left(x^{-\frac{2\pi ik}{\log(a_{m})}}+(-1)^{n}\prod_{\substack{l=1\\l\neq m}}^{n}a_{l}^{\frac{2\pi ik}{\log(a_{m})}}x^{\frac{2\pi ik}{\log(a_{m})}}\right)}{\prod_{\substack{l=1\\l\neq m}}^{n}\left(a_{l}^{\frac{2\pi ik}{\log(a_{m})}}-1\right)}.\hspace{12cm}
\end{split}
\end{displaymath}
In the above calculation, we have used the two algebraic identities
\begin{displaymath}
\begin{split}
S_{1}(x_{1},x_{2},x_{3},\ldots,x_{n})&=2^{n},\\
S_{2}(x_{1},x_{2},x_{3},\ldots,x_{n})&=2^{n}\prod_{k=1}^{n}x_{k},
\end{split}
\end{displaymath}
where
\begin{displaymath}
\begin{split}
S_{1}(x_{1},x_{2},x_{3},\ldots,x_{n})&=\sum_{\substack{\{\epsilon_{1},\epsilon_{2},\epsilon_{3},\ldots,\epsilon_{n}\}\\
\subset\{\pm1,\pm1,\pm1,\ldots,\pm1\}}}\prod_{k=1}^{n}\epsilon_{k}(x_{k}+\epsilon_{k})\\
S_{2}(x_{1},x_{2},x_{3},\ldots,x_{n})&=\sum_{\substack{\{\epsilon_{1},\epsilon_{2},\epsilon_{3},\ldots,\epsilon_{n}\}\\
\subset\{\pm1,\pm1,\pm1,\ldots,\pm1\}}}\prod_{k=1}^{n}(x_{k}+\epsilon_{k}).
\end{split}
\end{displaymath}
These two identities follow by induction, because for $n=1$, we have that
\begin{displaymath}
\begin{split}
S_{1}(x_{1})&=(x_{1}+1)-(x_{1}-1)=2\\
S_{2}(x_{1})&=(x_{1}+1)+(x_{1}-1)=2x_{1}.
\end{split}
\end{displaymath}
These two identities are exactly the claimed formulas for $S_{1}(x_{1})$ and $S_{2}(x_{1})$.\\
Supposing now that the statement is also true for $S_{1}(x_{1},x_{2},x_{3},\ldots,x_{n-1})$ and $S_{2}(x_{1},x_{2},x_{3},\ldots,x_{n-1})$, we prove by induction
\begin{displaymath}
\begin{split}
S_{1}(x_{1},x_{2},x_{3},\ldots,x_{n-1},x_{n})&=(x_{n}+1)S_{1}(x_{1},x_{2},x_{3},\ldots,x_{n-1})-(x_{n}-1)S_{1}(x_{1},x_{2},x_{3},\ldots,x_{n-1})\\
&=\left(x_{n}+1-x_{n}+1\right)S_{1}(x_{1},x_{2},x_{3},\ldots,x_{n-1})\\
&=2S_{1}(x_{1},x_{2},x_{3},\ldots,x_{n-1})\\
&=2\cdot2^{n-1}\\
&=2^{n},\\
S_{2}(x_{1},x_{2},x_{3},\ldots,x_{n-1},x_{n})&=(x_{n}+1)S_{2}(x_{1},x_{2},x_{3},\ldots,x_{n-1})+(x_{n}-1)S_{2}(x_{1},x_{2},x_{3},\ldots,x_{n-1})\\
&=\left(x_{n}+1+x_{n}-1\right)S_{2}(x_{1},x_{2},x_{3},\ldots,x_{n-1})\\
&=2x_{n}S_{2}(x_{1},x_{2},x_{3},\ldots,x_{n-1})\\
&=2x_{n}2^{n-1}\prod_{k=1}^{n-1}x_{k}\\
&=2^{n}\prod_{k=1}^{n}x_{k}.
\end{split}
\end{displaymath}
These are the claimed statements for $S_{1}(x_{1},x_{2},x_{3},\ldots,x_{n-1},x_{n})$ and $S_{2}(x_{1},x_{2},x_{3},\ldots,x_{n-1},x_{n})$. Therefore, the inductive proof is finished.

\vspace{0.4cm}

\noindent The above established identity implies that
\begin{displaymath}
\begin{split}
&\Res_{s=\frac{2\pi ik}{\log(a_{m})}}\left(\frac{x^{s}}{s\prod_{k=1}^{n}\left(1-\frac{1}{a_{k}^{s}}\right)}\right)+\Res_{s=-\frac{2\pi ik}{\log(a_{m})}}\left(\frac{x^{s}}{s\prod_{k=1}^{n}\left(1-\frac{1}{a_{k}^{s}}\right)}\right)\\
&=(-1)^{n+1}\frac{i\left(x^{-\frac{2\pi ik}{\log(a_{m})}}+(-1)^{n}\prod_{\substack{l=1\\l\neq m}}^{n}a_{l}^{\frac{2\pi ik}{\log(a_{m})}}x^{\frac{2\pi ik}{\log(a_{m})}}\right)}{2\pi k\prod_{\substack{l=1\\l\neq m}}^{n}\left(a_{l}^{\frac{2\pi ik}{\log(a_{m})}}-1\right)}\\
&=\frac{1}{2^{n-1}\pi k}\left(\sum_{r=1}^{n-1}\sum_{\substack{
i_{1}<i_{2}<i_{3}<\ldots<i_{r}\\
\{i_{1},i_{2},i_{3},\ldots,i_{r}\}\\
\subset\{a_{1},a_{2},a_{3},\ldots,\widehat{a_{m}},\ldots,a_{n}\}
}}\sin\left(2\pi k\frac{\log(x)}{\log(a_{m})}-\frac{\pi r}{2}\right)\prod_{l=1}^{r}\cot\left(\frac{\pi k\log(i_{l})}{\log(a_{m})}\right)+\sin\left(2\pi k\frac{\log(x)}{\log(a_{m})}\right)\right)\\
&\hspace{6cm}\text{for all $k\in\mathbb{N}$ and for all $1\leq m\leq n$}.
\end{split}
\end{displaymath}
Summing up all the Residues, we get our formula for $N_{a_{1},a_{2},a_{3},\ldots,a_{n}}(x)$.
\end{proof}

\begin{remark}
The first few identities of the family of identities, which we encountered in the above proof, are
\begin{displaymath}
\begin{split}
S_{1}(x_{1})&=(x_{1}+1)-(x_{1}-1)=2\\
\\
S_{2}(x_{1})&=(x_{1}+1)+(x_{1}-1)=2x_{1}\\
\\
S_{1}(x_{1},x_{2})&=(x_{1}+1)(x_{2}+1)-(x_{1}-1)(x_{2}+1)-(x_{1}+1)(x_{2}-1)+(x_{1}-1)(x_{2}-1)=4\\
\\
S_{2}(x_{1},x_{2})&=(x_{1}+1)(x_{2}+1)+(x_{1}-1)(x_{2}+1)+(x_{1}+1)(x_{2}-1)+(x_{1}-1)(x_{2}-1)=4x_{1}x_{2}\\
\\
S_{1}(x_{1},x_{2},x_{3})&=(x_{1}+1)(x_{2}+1)(x_{3}+1)-(x_{1}-1)(x_{2}+1)(x_{3}+1)-(x_{1}+1)(x_{2}-1)(x_{3}+1)\\
&\quad-(x_{1}+1)(x_{2}+1)(x_{3}-1)+(x_{1}+1)(x_{2}-1)(x_{3}-1)+(x_{1}-1)(x_{2}+1)(x_{3}-1)\\
&\quad+(x_{1}-1)(x_{2}-1)(x_{3}+1)-(x_{1}-1)(x_{2}-1)(x_{3}-1)=8\\
\\
S_{2}(x_{1},x_{2},x_{3})&=(x_{1}+1)(x_{2}+1)(x_{3}+1)+(x_{1}-1)(x_{2}+1)(x_{3}+1)+(x_{1}+1)(x_{2}-1)(x_{3}+1)\\
&\quad+(x_{1}+1)(x_{2}+1)(x_{3}-1)+
(x_{1}+1)(x_{2}-1)(x_{3}-1)+(x_{1}-1)(x_{2}+1)(x_{3}-1)\\
&\quad+(x_{1}-1)(x_{2}-1)(x_{3}+1)+(x_{1}-1)(x_{2}-1)(x_{3}-1)=8x_{1}x_{2}x_{3}.
\end{split}
\end{displaymath}
And so on.
\end{remark}

\vspace{0.4cm}

\noindent Setting $a_{1}=2$, $a_{2}=3$, $a_{3}=5$, $a_{4}=7$, \ldots, $a_{k}=p_{k}=\text{$k$-th prime number}$, \ldots, $a_{n}=p_{n}=\text{$n$-th prime number}$ in the above theorem, we get for the sequence
\begin{displaymath}
\begin{split}
S_{2,3,5,7,\ldots,p_{n}}:&=\left\{2^{q_{1}}3^{q_{2}}5^{q_{3}}7^{q_{4}}\cdots p_{n}^{q_{n}}:q_{1}\in\mathbb{N}_{0},q_{2}\in\mathbb{N}_{0},q_{3}\in\mathbb{N}_{0},\ldots,q_{n}\in\mathbb{N}_{0}\right\},
\end{split}
\end{displaymath}
\noindent of $p_{n}$-smooth numbers \cite{1,2}, immediately the following

\begin{corollary}(Formula for the $p_{n}$-Smooth Numbers Counting Function $N_{2,3,5,7,\ldots,p_{n}}(x)$)\\
\noindent For every real number $x\geq1$, we have that
\begin{displaymath}
\begin{split}
N_{2,3,5,7,\ldots,p_{n}}(x)&=\Res_{s=0}\left(\frac{x^{s}}{s\prod_{k=1}^{n}\left(1-\frac{1}{p_{k}^{s}}\right)}\right)-\frac{1}{2^{n-1}}\sum_{k=1}^{n}B^{*}_{1}\left(\left\{\frac{\log(x)}{\log(p_{k})}\right\}\right)\\
&\quad+\frac{1}{2^{n-1}\pi}\sum_{m=1}^{n}\sum_{r=1}^{n-1}\sum_{\substack{
i_{1}<i_{2}<i_{3}<\ldots<i_{r}\\
\{i_{1},i_{2},i_{3},\ldots,i_{r}\}\\
\subset\{2,3,5,7,\ldots,\widehat{p_{m}},\ldots,p_{n}\}}}\sum_{k=1}^{\infty}\frac{\sin\left(2\pi k\frac{\log(x)}{\log(p_{m})}-\frac{\pi r}{2}\right)}{k}\prod_{l=1}^{r}\cot\left(\frac{\pi k\log(i_{l})}{\log(p_{m})}\right)\\
&\quad+\frac{1}{2}\chi_{S_{2,3,5,7,\ldots,p_{n}}}(x),
\end{split}
\end{displaymath}
\noindent where the series are to be interpreted as meaning\begin{displaymath}
\begin{split}
&\sum_{k=1}^{\infty}\frac{\sin\left(2\pi k\frac{\log(x)}{\log(p_{m})}-\frac{\pi r}{2}\right)}{k}\prod_{l=1}^{r}\cot\left(\frac{\pi k\log(i_{l})}{\log(p_{m})}\right)\\
=&\lim_{R\rightarrow\infty}\left(\sum_{k=1}^{\left\lfloor R\log(p_{m})\right\rfloor}\frac{\sin\left(2\pi k\frac{\log(x)}{\log(p_{m})}-\frac{\pi r}{2}\right)}{k}\prod_{l=1}^{r}\cot\left(\frac{\pi k\log(i_{l})}{\log(p_{m})}\right)\right),
\end{split}
\end{displaymath}
\noindent when $R\rightarrow\infty$ in an appropriate manner.
\end{corollary}

\noindent This formula converges also very rapidly.\\

\noindent Therefore, we have

\begin{corollary}(The Hardy-Littlewood formula for $N_{a,b}(x)$ and $N_{2,3}(x)$)\cite{11,12}\\
\noindent For every real number $x\geq1$, we have that
\begin{displaymath}
\begin{split}
N_{a,b}(x)&=\frac{\log(x)^{2}}{2\log(a)\log(b)}+\frac{\log(x)}{2\log(a)}+\frac{\log(x)}{2\log(b)}+\frac{1}{4}+\frac{\log(a)}{12\log(b)}+\frac{\log(b)}{12\log(a)}\\
&\quad-\frac{1}{2}B^{*}_{1}\left(\left\{\frac{\log(x)}{\log(a)}\right\}\right)-\frac{1}{2}B^{*}_{1}\left(\left\{\frac{\log(x)}{\log(b)}\right\}\right)-\frac{1}{2\pi}\sum_{k=1}^{\infty}\frac{\cos\left(\frac{\pi k\log(b)}{\log(a)}\right)\cos\left(2\pi k\frac{\log(x)}{\log(a)}\right)}{k\sin\left(\frac{\pi k\log(b)}{\log(a)}\right)}\\
&\quad-\frac{1}{2\pi}\sum_{k=1}^{\infty}\frac{\cos\left(\frac{\pi k\log(a)}{\log(b)}\right)\cos\left(2\pi k\frac{\log(x)}{\log(b)}\right)}{k\sin\left(\frac{\pi k\log(a)}{\log(b)}\right)}+\frac{1}{2}\chi_{S_{a,b}}(x)
\end{split}
\end{displaymath}

\noindent and

\begin{displaymath}
\begin{split}
N_{2,3}(x)&=\frac{\log(x)^{2}}{2\log(2)\log(3)}+\frac{\log(x)}{2\log(2)}+\frac{\log(x)}{2\log(3)}+\frac{1}{4}+\frac{\log(2)}{12\log(3)}+\frac{\log(3)}{12\log(2)}\\
&\quad-\frac{1}{2}B^{*}_{1}\left(\left\{\frac{\log(x)}{\log(2)}\right\}\right)-\frac{1}{2}B^{*}_{1}\left(\left\{\frac{\log(x)}{\log(3)}\right\}\right)-\frac{1}{2\pi}\sum_{k=1}^{\infty}\frac{\cos\left(\frac{\pi k\log(3)}{\log(2)}\right)\cos\left(2\pi k\frac{\log(x)}{\log(2)}\right)}{k\sin\left(\frac{\pi k\log(3)}{\log(2)}\right)}\\
&\quad-\frac{1}{2\pi}\sum_{k=1}^{\infty}\frac{\cos\left(\frac{\pi k\log(2)}{\log(3)}\right)\cos\left(2\pi k\frac{\log(x)}{\log(3)}\right)}{k\sin\left(\frac{\pi k\log(2)}{\log(3)}\right)}
+\frac{1}{2}\chi_{S_{2,3}}(x).
\end{split}
\end{displaymath}
\end{corollary}

\begin{proof} The proof that we give here is Hardy's proof \cite{11} of the formula for $N_{a,b}(x)$.\\
We have that
\begin{displaymath}
\begin{split}
\sum_{k=1}^{\infty}\frac{\chi_{S_{a,b}}(k)}{k^{s}}&=\left(\sum_{m_{1}=0}^{\infty}\frac{1}{a^{m_{1}s}}\right)\left(\sum_{m_{2}=0}^{\infty}\frac{1}{b^{m_{2}s}}\right)\\
&=\frac{1}{\left(1-e^{-\log(a)s}\right)\left(1-e^{-\log(b)s}\right)}.
\end{split}
\end{displaymath}
Therefore, by Perron's formula, we get that
\begin{displaymath}
\begin{split}
N_{a,b}(x)&=\frac{1}{2\pi i}\int_{\gamma}\frac{x^{s}}{s\left(1-e^{-\log(a)s}\right)\left(1-e^{-\log(b)s}\right)}ds\\
&=\Res_{s=0}\left(\frac{x^{s}}{s\left(1-e^{-\log(a)s}\right)\left(1-e^{-\log(b)s}\right)}\right)+\sum_{k=1}^{\infty}\Res_{s=\frac{2\pi ik}{\log(a)}}\left(\frac{x^{s}}{s\left(1-e^{-\log(a)s}\right)\left(1-e^{-\log(b)s}\right)}\right)\\
&\quad+\sum_{k=1}^{\infty}\Res_{s=-\frac{2\pi ik}{\log(a)}}\left(\frac{x^{s}}{s\left(1-e^{-\log(a)s}\right)\left(1-e^{-\log(b)s}\right)}\right)\\
&\quad+\sum_{k=1}^{\infty}\Res_{s=\frac{2\pi ik}{\log(b)}}\left(\frac{x^{s}}{s\left(1-e^{-\log(a)s}\right)\left(1-e^{-\log(b)s}\right)}\right)\\
&\quad+\sum_{k=1}^{\infty}\Res_{s=-\frac{2\pi ik}{\log(b)}}\left(\frac{x^{s}}{s\left(1-e^{-\log(a)s}\right)\left(1-e^{-\log(b)s}\right)}\right)+\frac{1}{2}\chi_{S_{a,b}}(x),
\end{split}
\end{displaymath}
where $\gamma=\text{line from $1-i\infty$ to $1+i\infty$}$.\\
Moreover, we have that
\begin{displaymath}
\begin{split}
\Res_{s=0}\left(\frac{x^{s}}{s\left(1-e^{-\log(a)s}\right)\left(1-e^{-\log(b)s}\right)}\right)&=\frac{\log(x)^{2}}{2\log(a)\log(b)}+\frac{\log(x)}{2\log(a)}+\frac{\log(x)}{2\log(b)}\\
&\quad+\frac{1}{4}+\frac{\log(a)}{12\log(b)}+\frac{\log(b)}{12\log(a)}
\end{split}
\end{displaymath}
and that
\begin{displaymath}
\begin{split}
\Res_{s=\frac{2\pi ik}{\log(a)}}\left(\frac{x^{s}}{s\left(1-e^{-\log(a)s}\right)\left(1-e^{-\log(b)s}\right)}\right)&=-\frac{ib^{\frac{2\pi ik}{\log(a)}}x^{\frac{2\pi ik}{\log(a)}}}{2\pi k\left(b^{\frac{2\pi ik}{\log(a)}}-1\right)}\;\;\text{for all $k\in\mathbb{N}$}\\
\Res_{s=-\frac{2\pi ik}{\log(a)}}\left(\frac{x^{s}}{s\left(1-e^{-\log(a)s}\right)\left(1-e^{-\log(b)s}\right)}\right)&=-\frac{ix^{-\frac{2\pi ik}{\log(a)}}}{2\pi k\left(b^{\frac{2\pi ik}{\log(a)}}-1\right)}\;\;\text{for all $k\in\mathbb{N}$}
\end{split}
\end{displaymath}
\begin{displaymath}
\begin{split}
\Res_{s=\frac{2\pi ik}{\log(b)}}\left(\frac{x^{s}}{s\left(1-e^{-\log(a)s}\right)\left(1-e^{-\log(b)s}\right)}\right)&=-\frac{ia^{\frac{2\pi ik}{\log(b)}}x^{\frac{2\pi ik}{\log(b)}}}{2\pi k\left(a^{\frac{2\pi ik}{\log(b)}}-1\right)}\;\;\text{for all $k\in\mathbb{N}$}\\
\Res_{s=-\frac{2\pi ik}{\log(b)}}\left(\frac{x^{s}}{s\left(1-e^{-\log(a)s}\right)\left(1-e^{-\log(b)s}\right)}\right)&=-\frac{ix^{-\frac{2\pi ik}{\log(b)}}}{2\pi k\left(a^{\frac{2\pi ik}{\log(b)}}-1\right)}\;\;\text{for all $k\in\mathbb{N}$}.
\end{split}
\end{displaymath}
Using the relations
\begin{displaymath}
\begin{split}
\sin\left(2\pi k\frac{\log(x)}{\log(a)}\right)&=\frac{1}{2}ix^{-\frac{2\pi ik}{\log(a)}}-\frac{1}{2}ix^{\frac{2\pi ik}{\log(a)}}\\
\cos\left(2\pi k\frac{\log(x)}{\log(a)}\right)&=\frac{1}{2}x^{-\frac{2\pi ik}{\log(a)}}+\frac{1}{2}x^{\frac{2\pi ik}{\log(a)}}\\
\cot\left(\frac{\pi k\log(b)}{\log(a)}\right)&=i\frac{b^{\frac{2\pi ik}{\log(a)}}+1}{b^{\frac{2\pi ik}{\log(a)}}-1},
\end{split}
\end{displaymath}
we establish (by expanding everything out) the following identity
\begin{displaymath}
\begin{split}
&-\cot\left(\frac{\pi k\log(b)}{\log(a)}\right)\cos\left(2\pi k\frac{\log(x)}{\log(a)}\right)+\sin\left(2\pi k\frac{\log(x)}{\log(a)}\right)\\
&=-\left(i\frac{b^{\frac{2\pi ik}{\log(a)}}+1}{b^{\frac{2\pi ik}{\log(a)}}-1}\right)\left(\frac{1}{2}x^{-\frac{2\pi ik}{\log(a)}}+\frac{1}{2}x^{\frac{2\pi ik}{\log(a)}}\right)+\left(\frac{1}{2}ix^{-\frac{2\pi ik}{\log(a)}}-\frac{1}{2}ix^{\frac{2\pi ik}{\log(a)}}\right)\\
&=-\frac{i\left(x^{-\frac{2\pi ik}{\log(a)}}+b^{\frac{2\pi ik}{\log(a)}}x^{\frac{2\pi ik}{\log(a)}}\right)}{b^{\frac{2\pi ik}{\log(a)}}-1}.
\end{split}
\end{displaymath}
Therefore, for the first Residues (the "$a$ -Residues"), we have that
\begin{displaymath}
\begin{split}
&\Res_{s=\frac{2\pi ik}{\log(a)}}\left(\frac{x^{s}}{s\left(1-e^{-\log(a)s}\right)\left(1-e^{-\log(b)s}\right)}\right)+\Res_{s=-\frac{2\pi ik}{\log(a)}}\left(\frac{x^{s}}{s\left(1-e^{-\log(a)s}\right)\left(1-e^{-\log(b)s}\right)}\right)\\&=-\frac{i\left(x^{-\frac{2\pi ik}{\log(a)}}+b^{\frac{2\pi ik}{\log(a)}}x^{\frac{2\pi ik}{\log(a)}}\right)}{2\pi k\left(b^{\frac{2\pi ik}{\log(a)}}-1\right)}\\
&=\frac{1}{2\pi k}\left(-\cot\left(\frac{\pi k\log(b)}{\log(a)}\right)\cos\left(2\pi k\frac{\log(x)}{\log(a)}\right)+\sin\left(2\pi k\frac{\log(x)}{\log(a)}\right)\right)\;\;\text{for all $k\in\mathbb{N}$}.
\end{split}
\end{displaymath}
Exchanging $a$ and $b$ ("permuting $a$ and $b$"), we get also the other Residues (the "$b$ -Residues"), namely
\begin{displaymath}
\begin{split}
&\Res_{s=\frac{2\pi ik}{\log(b)}}\left(\frac{x^{s}}{s\left(1-e^{-\log(a)s}\right)\left(1-e^{-\log(b)s}\right)}\right)+\Res_{s=-\frac{2\pi ik}{\log(b)}}\left(\frac{x^{s}}{s\left(1-e^{-\log(a)s}\right)\left(1-e^{-\log(b)s}\right)}\right)\\&=-\frac{i\left(x^{-\frac{2\pi ik}{\log(b)}}+a^{\frac{2\pi ik}{\log(b)}}x^{\frac{2\pi ik}{\log(b)}}\right)}{2\pi k\left(a^{\frac{2\pi ik}{\log(b)}}-1\right)}\\
&=\frac{1}{2\pi k}\left(-\cot\left(\frac{\pi k\log(a)}{\log(b)}\right)\cos\left(2\pi k\frac{\log(x)}{\log(b)}\right)+\sin\left(2\pi k\frac{\log(x)}{\log(b)}\right)\right)\;\;\text{for all $k\in\mathbb{N}$}.
\end{split}
\end{displaymath}
Summing everything up, we get our formula for $N_{a,b}(x)$. Setting $a=2$ and $b=3$, we get also the formula for $N_{2,3}(x)$.
\end{proof}

\begin{corollary}(The Formulas for $N_{a,b,c}(x)$ and $N_{2,3,5}(x)$)\\
\noindent For every real number $x\geq1$, we have that
\begin{displaymath}
\begin{split}
N_{a,b,c}(x)&=\frac{\log(x)^{3}}{6\log(a)\log(b)\log(c)}+\frac{\log(x)^{2}}{4\log(a)\log(b)}+\frac{\log(x)^{2}}{4\log(a)\log(c)}+\frac{\log(x)^{2}}{4\log(b)\log(c)}+\frac{\log(x)}{4\log(a)}\\
&\quad+\frac{\log(x)}{4\log(b)}+\frac{\log(x)}{4\log(c)}+\frac{\log(a)\log(x)}{12\log(b)\log(c)}+\frac{\log(b)\log(x)}{12\log(a)\log(c)}+\frac{\log(c)\log(x)}{12\log(a)\log(b)}\\
&\quad+\frac{\log(a)}{24\log(b)}+\frac{\log(a)}{24\log(c)}+\frac{\log(b)}{24\log(a)}+\frac{\log(b)}{24\log(c)}+\frac{\log(c)}{24\log(a)}+\frac{\log(c)}{24\log(b)}+\frac{1}{8}\\
&\quad-\frac{1}{4}B^{*}_{1}\left(\left\{\frac{\log(x)}{\log(a)}\right\}\right)-\frac{1}{4}B^{*}_{1}\left(\left\{\frac{\log(x)}{\log(b)}\right\}\right)-\frac{1}{4}B^{*}_{1}\left(\left\{\frac{\log(x)}{\log(c)}\right\}\right)\\
&\quad-\frac{1}{4\pi}\sum_{k=1}^{\infty}\frac{\cos\left(\frac{\pi k\log(b)}{\log(a)}\right)\cos\left(2\pi k\frac{\log(x)}{\log(a)}\right)}{k\sin\left(\frac{\pi k\log(b)}{\log(a)}\right)}-\frac{1}{4\pi}\sum_{k=1}^{\infty}\frac{\cos\left(\frac{\pi k\log(a)}{\log(b)}\right)\cos\left(2\pi k\frac{\log(x)}{\log(b)}\right)}{k\sin\left(\frac{\pi k\log(a)}{\log(b)}\right)}\\
&\quad-\frac{1}{4\pi}\sum_{k=1}^{\infty}\frac{\cos\left(\frac{\pi k\log(c)}{\log(b)}\right)\cos\left(2\pi k\frac{\log(x)}{\log(b)}\right)}{k\sin\left(\frac{\pi k\log(c)}{\log(b)}\right)}-\frac{1}{4\pi}\sum_{k=1}^{\infty}\frac{\cos\left(\frac{\pi k\log(b)}{\log(c)}\right)\cos\left(2\pi k\frac{\log(x)}{\log(c)}\right)}{k\sin\left(\frac{\pi k\log(b)}{\log(c)}\right)}\\
&\quad-\frac{1}{4\pi}\sum_{k=1}^{\infty}\frac{\cos\left(\frac{\pi k\log(a)}{\log(c)}\right)\cos\left(2\pi k\frac{\log(x)}{\log(c)}\right)}{k\sin\left(\frac{\pi k\log(a)}{\log(c)}\right)}-\frac{1}{4\pi}\sum_{k=1}^{\infty}\frac{\cos\left(\frac{\pi k\log(c)}{\log(a)}\right)\cos\left(2\pi k\frac{\log(x)}{\log(a)}\right)}{k\sin\left(\frac{\pi k\log(c)}{\log(a)}\right)}\\
&\quad-\frac{1}{4\pi}\sum_{k=1}^{\infty}\frac{\cos\left(\frac{\pi k\log(b)}{\log(a)}\right)\cos\left(\frac{\pi k\log(c)}{\log(a)}\right)\sin\left(2\pi k\frac{\log(x)}{\log(a)}\right)}{k\sin\left(\frac{\pi k\log(b)}{\log(a)}\right)\sin\left(\frac{\pi k\log(c)}{\log(a)}\right)}\\
&\quad-\frac{1}{4\pi}\sum_{k=1}^{\infty}\frac{\cos\left(\frac{\pi k\log(a)}{\log(b)}\right)\cos\left(\frac{\pi k\log(c)}{\log(b)}\right)\sin\left(2\pi k\frac{\log(x)}{\log(b)}\right)}{k\sin\left(\frac{\pi k\log(a)}{\log(b)}\right)\sin\left(\frac{\pi k\log(c)}{\log(b)}\right)}\\
\end{split}
\end{displaymath}
\begin{displaymath}
\begin{split}
\hspace{-1.35cm}&\quad-\frac{1}{4\pi}\sum_{k=1}^{\infty}\frac{\cos\left(\frac{\pi k\log(a)}{\log(c)}\right)\cos\left(\frac{\pi k\log(b)}{\log(c)}\right)\sin\left(2\pi k\frac{\log(x)}{\log(c)}\right)}{k\sin\left(\frac{\pi k\log(a)}{\log(c)}\right)\sin\left(\frac{\pi k\log(b)}{\log(c)}\right)}+\frac{1}{2}\chi_{S_{a,b,c}}(x)
\end{split}
\end{displaymath}

\noindent and

\begin{displaymath}
\begin{split}
N_{2,3,5}(x)&=\frac{\log(x)^{3}}{6\log(2)\log(3)\log(5)}+\frac{\log(x)^{2}}{4\log(2)\log(3)}+\frac{\log(x)^{2}}{4\log(2)\log(5)}+\frac{\log(x)^{2}}{4\log(3)\log(5)}+\frac{\log(x)}{4\log(2)}\\
&\quad+\frac{\log(x)}{4\log(3)}+\frac{\log(x)}{4\log(5)}+\frac{\log(2)\log(x)}{12\log(3)\log(5)}+\frac{\log(3)\log(x)}{12\log(2)\log(5)}+\frac{\log(5)\log(x)}{12\log(2)\log(3)}\\
&\quad+\frac{\log(2)}{24\log(3)}+\frac{\log(2)}{24\log(5)}+\frac{\log(3)}{24\log(2)}+\frac{\log(3)}{24\log(5)}+\frac{\log(5)}{24\log(2)}+\frac{\log(5)}{24\log(3)}+\frac{1}{8}\\
&\quad-\frac{1}{4}B^{*}_{1}\left(\left\{\frac{\log(x)}{\log(2)}\right\}\right)-\frac{1}{4}B^{*}_{1}\left(\left\{\frac{\log(x)}{\log(3)}\right\}\right)-\frac{1}{4}B^{*}_{1}\left(\left\{\frac{\log(x)}{\log(5)}\right\}\right)\\
&\quad-\frac{1}{4\pi}\sum_{k=1}^{\infty}\frac{\cos\left(\frac{\pi k\log(3)}{\log(2)}\right)\cos\left(2\pi k\frac{\log(x)}{\log(2)}\right)}{k\sin\left(\frac{\pi k\log(3)}{\log(2)}\right)}-\frac{1}{4\pi}\sum_{k=1}^{\infty}\frac{\cos\left(\frac{\pi k\log(2)}{\log(3)}\right)\cos\left(2\pi k\frac{\log(x)}{\log(3)}\right)}{k\sin\left(\frac{\pi k\log(2)}{\log(3)}\right)}\\
&\quad-\frac{1}{4\pi}\sum_{k=1}^{\infty}\frac{\cos\left(\frac{\pi k\log(5)}{\log(3)}\right)\cos\left(2\pi k\frac{\log(x)}{\log(3)}\right)}{k\sin\left(\frac{\pi k\log(5)}{\log(3)}\right)}-\frac{1}{4\pi}\sum_{k=1}^{\infty}\frac{\cos\left(\frac{\pi k\log(3)}{\log(5)}\right)\cos\left(2\pi k\frac{\log(x)}{\log(5)}\right)}{k\sin\left(\frac{\pi k\log(3)}{\log(5)}\right)}\\
&\quad-\frac{1}{4\pi}\sum_{k=1}^{\infty}\frac{\cos\left(\frac{\pi k\log(2)}{\log(5)}\right)\cos\left(2\pi k\frac{\log(x)}{\log(5)}\right)}{k\sin\left(\frac{\pi k\log(2)}{\log(5)}\right)}-\frac{1}{4\pi}\sum_{k=1}^{\infty}\frac{\cos\left(\frac{\pi k\log(5)}{\log(2)}\right)\cos\left(2\pi k\frac{\log(x)}{\log(2)}\right)}{k\sin\left(\frac{\pi k\log(5)}{\log(2)}\right)}\\
&\quad-\frac{1}{4\pi}\sum_{k=1}^{\infty}\frac{\cos\left(\frac{\pi k\log(3)}{\log(2)}\right)\cos\left(\frac{\pi k\log(5)}{\log(2)}\right)\sin\left(2\pi k\frac{\log(x)}{\log(2)}\right)}{k\sin\left(\frac{\pi k\log(3)}{\log(2)}\right)\sin\left(\frac{\pi k\log(5)}{\log(2)}\right)}\\
&\quad-\frac{1}{4\pi}\sum_{k=1}^{\infty}\frac{\cos\left(\frac{\pi k\log(2)}{\log(3)}\right)\cos\left(\frac{\pi k\log(5)}{\log(3)}\right)\sin\left(2\pi k\frac{\log(x)}{\log(3)}\right)}{k\sin\left(\frac{\pi k\log(2)}{\log(3)}\right)\sin\left(\frac{\pi k\log(5)}{\log(3)}\right)}\\
&\quad-\frac{1}{4\pi}\sum_{k=1}^{\infty}\frac{\cos\left(\frac{\pi k\log(2)}{\log(5)}\right)\cos\left(\frac{\pi k\log(3)}{\log(5)}\right)\sin\left(2\pi k\frac{\log(x)}{\log(5)}\right)}{k\sin\left(\frac{\pi k\log(2)}{\log(5)}\right)\sin\left(\frac{\pi k\log(3)}{\log(5)}\right)}+\frac{1}{2}\chi_{S_{2,3,5}}(x).
\end{split}
\end{displaymath}
\end{corollary}

\begin{proof}
We have that
\begin{displaymath}
\begin{split}
\sum_{k=1}^{\infty}\frac{\chi_{S_{a,b,c}}(k)}{k^{s}}&=\left(\sum_{m_{1}=0}^{\infty}\frac{1}{a^{m_{1}s}}\right)\left(\sum_{m_{2}=0}^{\infty}\frac{1}{b^{m_{2}s}}\right)\left(\sum_{m_{3}=0}^{\infty}\frac{1}{c^{m_{3}s}}\right)\\
&=\frac{1}{\left(1-e^{-\log(a)s}\right)\left(1-e^{-\log(b)s}\right)\left(1-e^{-\log(c)s}\right)}.
\end{split}
\end{displaymath}
Therefore, by Perron's formula, we get that
\begin{displaymath}
\begin{split}
N_{a,b,c}(x)&=\frac{1}{2\pi i}\int_{\gamma}\frac{x^{s}}{s\left(1-e^{-\log(a)s}\right)\left(1-e^{-\log(b)s}\right)\left(1-e^{-\log(c)s}\right)}ds\\
&=\Res_{s=0}\left(\frac{x^{s}}{s\left(1-e^{-\log(a)s}\right)\left(1-e^{-\log(b)s}\right)\left(1-e^{-\log(c)s}\right)}\right)\\
&\quad+\sum_{k=1}^{\infty}\Res_{s=\frac{2\pi ik}{\log(a)}}\left(\frac{x^{s}}{s\left(1-e^{-\log(a)s}\right)\left(1-e^{-\log(b)s}\right)\left(1-e^{-\log(c)s}\right)}\right)\\
&\quad+\sum_{k=1}^{\infty}\Res_{s=-\frac{2\pi ik}{\log(a)}}\left(\frac{x^{s}}{s\left(1-e^{-\log(a)s}\right)\left(1-e^{-\log(b)s}\right)\left(1-e^{-\log(c)s}\right)}\right)\\
&\quad+\sum_{k=1}^{\infty}\Res_{s=\frac{2\pi ik}{\log(b)}}\left(\frac{x^{s}}{s\left(1-e^{-\log(a)s}\right)\left(1-e^{-\log(b)s}\right)\left(1-e^{-\log(c)s}\right)}\right)\\
&\quad+\sum_{k=1}^{\infty}\Res_{s=-\frac{2\pi ik}{\log(b)}}\left(\frac{x^{s}}{s\left(1-e^{-\log(a)s}\right)\left(1-e^{-\log(b)s}\right)\left(1-e^{-\log(c)s}\right)}\right)\\
&\quad+\sum_{k=1}^{\infty}\Res_{s=\frac{2\pi ik}{\log(c)}}\left(\frac{x^{s}}{s\left(1-e^{-\log(a)s}\right)\left(1-e^{-\log(b)s}\right)\left(1-e^{-\log(c)s}\right)}\right)\\
&\quad+\sum_{k=1}^{\infty}\Res_{s=-\frac{2\pi ik}{\log(c)}}\left(\frac{x^{s}}{s\left(1-e^{-\log(a)s}\right)\left(1-e^{-\log(b)s}\right)\left(1-e^{-\log(c)s}\right)}\right)+\frac{1}{2}\chi_{S_{a,b,c}}(x),
\end{split}
\end{displaymath}
where $\gamma=\text{line from $1-i\infty$ to $1+i\infty$}$.\\
Furthermore, we have that
\begin{displaymath}
\begin{split}
&\Res_{s=0}\left(\frac{x^{s}}{s\left(1-e^{-\log(a)s}\right)\left(1-e^{-\log(b)s}\right)\left(1-e^{-\log(c)s}\right)}\right)\\
&=\frac{\log(x)^{3}}{6\log(a)\log(b)\log(c)}+\frac{\log(x)^{2}}{4\log(a)\log(b)}+\frac{\log(x)^{2}}{4\log(a)\log(c)}+\frac{\log(x)^{2}}{4\log(b)\log(c)}+\frac{\log(x)}{4\log(a)}\\
&\quad+\frac{\log(x)}{4\log(b)}+\frac{\log(x)}{4\log(c)}+\frac{\log(a)\log(x)}{12\log(b)\log(c)}+\frac{\log(b)\log(x)}{12\log(a)\log(c)}+\frac{\log(c)\log(x)}{12\log(a)\log(b)}\\
&\quad+\frac{\log(a)}{24\log(b)}+\frac{\log(a)}{24\log(c)}+\frac{\log(b)}{24\log(a)}+\frac{\log(b)}{24\log(c)}+\frac{\log(c)}{24\log(a)}+\frac{\log(c)}{24\log(b)}+\frac{1}{8}
\end{split}
\end{displaymath}
and that
\begin{displaymath}
\begin{split}
\Res_{s=\frac{2\pi ik}{\log(a)}}\left(\frac{x^{s}}{s\left(1-e^{-\log(a)s}\right)\left(1-e^{-\log(b)s}\right)\left(1-e^{-\log(c)s}\right)}\right)&=-\frac{ib^{\frac{2\pi ik}{\log(a)}}c^{\frac{2\pi ik}{\log(a)}}x^{\frac{2\pi ik}{\log(a)}}}{2\pi k\left(b^{\frac{2\pi ik}{\log(a)}}-1\right)\left(c^{\frac{2\pi ik}{\log(a)}}-1\right)}\;\;\forall k\in\mathbb{N}\\
\Res_{s=-\frac{2\pi ik}{\log(a)}}\left(\frac{x^{s}}{s\left(1-e^{-\log(a)s}\right)\left(1-e^{-\log(b)s}\right)\left(1-e^{-\log(c)s}\right)}\right)&=\frac{ix^{-\frac{2\pi ik}{\log(a)}}}{2\pi k\left(b^{\frac{2\pi ik}{\log(a)}}-1\right)\left(c^{\frac{2\pi ik}{\log(a)}}-1\right)}\;\;\forall k\in\mathbb{N}.
\end{split}
\end{displaymath}
Exactly similar expressions hold also for the other Residues under exchanging $a$ with $b$, and $a$ with $c$ ("permuting  $a,b,c$").
Using again the relations
\begin{displaymath}
\begin{split}
\sin\left(2\pi k\frac{\log(x)}{\log(a)}\right)&=\frac{1}{2}ix^{-\frac{2\pi ik}{\log(a)}}-\frac{1}{2}ix^{\frac{2\pi ik}{\log(a)}}\\
\cos\left(2\pi k\frac{\log(x)}{\log(a)}\right)&=\frac{1}{2}x^{-\frac{2\pi ik}{\log(a)}}+\frac{1}{2}x^{\frac{2\pi ik}{\log(a)}}\\
\cot\left(\frac{\pi k\log(b)}{\log(a)}\right)&=i\frac{b^{\frac{2\pi ik}{\log(a)}}+1}{b^{\frac{2\pi ik}{\log(a)}}-1},
\end{split}
\end{displaymath}
we establish (by expanding everything out) the following identity
\begin{displaymath}
\begin{split}
&-\cot\left(\frac{\pi k\log(b)}{\log(a)}\right)\cot\left(\frac{\pi k\log(c)}{\log(a)}\right)\sin\left(2\pi k\frac{\log(x)}{\log(a)}\right)-\cot\left(\frac{\pi k\log(b)}{\log(a)}\right)\cos\left(2\pi k\frac{\log(x)}{\log(a)}\right)\\
&-\cot\left(\frac{\pi k\log(c)}{\log(a)}\right)\cos\left(2\pi k\frac{\log(x)}{\log(a)}\right)+\sin\left(2\pi k\frac{\log(x)}{\log(a)}\right)\\
&=-\left(i\frac{b^{\frac{2\pi ik}{\log(a)}}+1}{b^{\frac{2\pi ik}{\log(a)}}-1}\right)\left(i\frac{c^{\frac{2\pi ik}{\log(a)}}+1}{c^{\frac{2\pi ik}{\log(a)}}-1}\right)\left(\frac{1}{2}ix^{-\frac{2\pi ik}{\log(a)}}-\frac{1}{2}ix^{\frac{2\pi ik}{\log(a)}}\right)-\left(i\frac{b^{\frac{2\pi ik}{\log(a)}}+1}{b^{\frac{2\pi ik}{\log(a)}}-1}\right)\left(\frac{1}{2}x^{-\frac{2\pi ik}{\log(a)}}+\frac{1}{2}x^{\frac{2\pi ik}{\log(a)}}\right)\\
&-\left(i\frac{c^{\frac{2\pi ik}{\log(a)}}+1}{c^{\frac{2\pi ik}{\log(a)}}-1}\right)\left(\frac{1}{2}x^{-\frac{2\pi ik}{\log(a)}}+\frac{1}{2}x^{\frac{2\pi ik}{\log(a)}}\right)+\left(\frac{1}{2}ix^{-\frac{2\pi ik}{\log(a)}}-\frac{1}{2}ix^{\frac{2\pi ik}{\log(a)}}\right)\\
&=\frac{2i\left(x^{-\frac{2\pi ik}{\log(a)}}-b^{\frac{2\pi ik}{\log(a)}}c^{\frac{2\pi ik}{\log(a)}}x^{\frac{2\pi ik}{\log(a)}}\right)}{\left(b^{\frac{2\pi ik}{\log(a)}}-1\right)\left(c^{\frac{2\pi ik}{\log(a)}}-1\right)}.
\end{split}
\end{displaymath}
Therefore, for the first Residues (the "$a$ -Residues"), we have that
\begin{displaymath}
\begin{split}
&\Res_{s=\frac{2\pi ik}{\log(a)}}\left(\frac{x^{s}}{s\left(1-e^{-\log(a)s}\right)\left(1-e^{-\log(b)s}\right)\left(1-e^{-\log(c)s}\right)}\right)\\
&+\Res_{s=-\frac{2\pi ik}{\log(a)}}\left(\frac{x^{s}}{s\left(1-e^{-\log(a)s}\right)\left(1-e^{-\log(b)s}\right)\left(1-e^{-\log(c)s}\right)}\right)\\
&=\frac{i\left(x^{-\frac{2\pi ik}{\log(a)}}-b^{\frac{2\pi ik}{\log(a)}}c^{\frac{2\pi ik}{\log(a)}}x^{\frac{2\pi ik}{\log(a)}}\right)}{2\pi k\left(b^{\frac{2\pi ik}{\log(a)}}-1\right)\left(c^{\frac{2\pi ik}{\log(a)}}-1\right)}\\
&=\frac{1}{4\pi k}\bigg(-\cot\left(\frac{\pi k\log(b)}{\log(a)}\right)\cot\left(\frac{\pi k\log(c)}{\log(a)}\right)\sin\left(2\pi k\frac{\log(x)}{\log(a)}\right)-\cot\left(\frac{\pi k\log(b)}{\log(a)}\right)\cos\left(2\pi k\frac{\log(x)}{\log(a)}\right)\\
&\hspace{1.6cm}-\cot\left(\frac{\pi k\log(c)}{\log(a)}\right)\cos\left(2\pi k\frac{\log(x)}{\log(a)}\right)+\sin\left(2\pi k\frac{\log(x)}{\log(a)}\right)\bigg)\;\;\forall k\in\mathbb{N}.
\end{split}
\end{displaymath}
Exchanging $a$ with $b$, and $a$ with $c$ ("permuting $a,b,c$"), we get also the other Residues (the "$b$ -Residues" and the "$c$ -Residues"), which have exactly the same structure.
Summing everything up, we get the formula for $N_{a,b,c}(x)$. Setting $a=2$, $b=3$ and $c=5$, we get also the formula for $N_{2,3,5}(x)$.
\end{proof}

\begin{corollary}(The Formulas for $N_{a,b,c,d}(x)$ and $N_{2,3,5,7}(x)$)\\
\noindent For every real number $x\geq1$, we have that
\begin{displaymath}
\begin{split}
N_{a,b,c,d}(x)&=\frac{\log(x)^{4}}{24\log(a)\log(b)\log(c)\log(d)}+\frac{\log(x)^{3}}{12\log(a)\log(b)\log(c)}+\frac{\log(x)^{3}}{12\log(a)\log(b)\log(d)}\\
&\quad+\frac{\log(x)^{3}}{12\log(a)\log(c)\log(d)}+\frac{\log(x)^{3}}{12\log(b)\log(c)\log(d)}+\frac{\log(a)\log(x)^{2}}{24\log(b)\log(c)\log(d)}\\
&\quad+\frac{\log(b)\log(x)^{2}}{24\log(a)\log(c)\log(d)}+\frac{\log(c)\log(x)^{2}}{24\log(a)\log(b)\log(d)}+\frac{\log(d)\log(x)^{2}}{24\log(a)\log(b)\log(c)}\\
&\quad+\frac{\log(x)^{2}}{8\log(a)\log(b)}+\frac{\log(x)^{2}}{8\log(a)\log(c)}+\frac{\log(x)^{2}}{8\log(a)\log(d)}+\frac{\log(x)^{2}}{8\log(b)\log(c)}+\frac{\log(x)^{2}}{8\log(b)\log(d)}\\
&\quad+\frac{\log(x)^{2}}{8\log(c)\log(d)}+\frac{\log(x)}{8\log(a)}+\frac{\log(x)}{8\log(b)}+\frac{\log(x)}{8\log(c)}+\frac{\log(x)}{8\log(d)}+\frac{\log(a)\log(x)}{24\log(b)\log(c)}\\
&\quad+\frac{\log(a)\log(x)}{24\log(b)\log(d)}+\frac{\log(a)\log(x)}{24\log(c)\log(d)}+\frac{\log(b)\log(x)}{24\log(a)\log(c)}+\frac{\log(b)\log(x)}{24\log(a)\log(d)}+\frac{\log(b)\log(x)}{24\log(c)\log(d)}\\
&\quad+\frac{\log(c)\log(x)}{24\log(a)\log(b)}+\frac{\log(c)\log(x)}{24\log(a)\log(d)}+\frac{\log(c)\log(x)}{24\log(b)\log(d)}+\frac{\log(d)\log(x)}{24\log(a)\log(b)}+\frac{\log(d)\log(x)}{24\log(a)\log(c)}\\
&\quad+\frac{\log(d)\log(x)}{24\log(b)\log(c)}+\frac{1}{16}+\frac{\log(a)}{48\log(b)}+\frac{\log(a)}{48\log(c)}+\frac{\log(a)}{48\log(d)}+\frac{\log(b)}{48\log(a)}+\frac{\log(b)}{48\log(c)}\\
&\quad+\frac{\log(b)}{48\log(d)}+\frac{\log(c)}{48\log(a)}+\frac{\log(c)}{48\log(b)}+\frac{\log(c)}{48\log(d)}+\frac{\log(d)}{48\log(a)}+\frac{\log(d)}{48\log(b)}+\frac{\log(d)}{48\log(c)}\\
&\quad+\frac{\log(a)\log(b)}{144\log(c)\log(d)}+\frac{\log(a)\log(c)}{144\log(b)\log(d)}+\frac{\log(a)\log(d)}{144\log(b)\log(c)}+\frac{\log(b)\log(c)}{144\log(a)\log(d)}\\
&\quad+\frac{\log(b)\log(d)}{144\log(a)\log(c)}+\frac{\log(c)\log(d)}{144\log(a)\log(b)}-\frac{\log(a)^{3}}{720\log(b)\log(c)\log(d)}-\frac{\log(b)^{3}}{720\log(a)\log(c)\log(d)}\\
&\quad-\frac{\log(c)^{3}}{720\log(a)\log(b)\log(d)}-\frac{\log(d)^{3}}{720\log(a)\log(b)\log(c)}-\frac{1}{8}B^{*}_{1}\left(\left\{\frac{\log(x)}{\log(a)}\right\}\right)\\
&\quad-\frac{1}{8}B^{*}_{1}\left(\left\{\frac{\log(x)}{\log(b)}\right\}\right)-\frac{1}{8}B^{*}_{1}\left(\left\{\frac{\log(x)}{\log(c)}\right\}\right)-\frac{1}{8}B^{*}_{1}\left(\left\{\frac{\log(x)}{\log(d)}\right\}\right)\\
&\quad-\frac{1}{8\pi}\sum_{k=1}^{\infty}\frac{\cos\left(\frac{\pi k\log(b)}{\log(a)}\right)\cos\left(2\pi k\frac{\log(x)}{\log(a)}\right)}{k\sin\left(\frac{\pi k\log(b)}{\log(a)}\right)}-\frac{1}{8\pi}\sum_{k=1}^{\infty}\frac{\cos\left(\frac{\pi k\log(a)}{\log(b)}\right)\cos\left(2\pi k\frac{\log(x)}{\log(b)}\right)}{k\sin\left(\frac{\pi k\log(a)}{\log(b)}\right)}\\
&\quad-\frac{1}{8\pi}\sum_{k=1}^{\infty}\frac{\cos\left(\frac{\pi k\log(c)}{\log(a)}\right)\cos\left(2\pi k\frac{\log(x)}{\log(a)}\right)}{k\sin\left(\frac{\pi k\log(c)}{\log(a)}\right)}-\frac{1}{8\pi}\sum_{k=1}^{\infty}\frac{\cos\left(\frac{\pi k\log(a)}{\log(c)}\right)\cos\left(2\pi k\frac{\log(x)}{\log(c)}\right)}{k\sin\left(\frac{\pi k\log(a)}{\log(c)}\right)}\\
\end{split}
\end{displaymath}
\begin{displaymath}
\begin{split}
\hspace{1.38cm}&\quad-\frac{1}{8\pi}\sum_{k=1}^{\infty}\frac{\cos\left(\frac{\pi k\log(d)}{\log(a)}\right)\cos\left(2\pi k\frac{\log(x)}{\log(a)}\right)}{k\sin\left(\frac{\pi k\log(d)}{\log(a)}\right)}-\frac{1}{8\pi}\sum_{k=1}^{\infty}\frac{\cos\left(\frac{\pi k\log(a)}{\log(d)}\right)\cos\left(2\pi k\frac{\log(x)}{\log(d)}\right)}{k\sin\left(\frac{\pi k\log(a)}{\log(d)}\right)}\\
&\quad-\frac{1}{8\pi}\sum_{k=1}^{\infty}\frac{\cos\left(\frac{\pi k\log(c)}{\log(b)}\right)\cos\left(2\pi k\frac{\log(x)}{\log(b)}\right)}{k\sin\left(\frac{\pi k\log(c)}{\log(b)}\right)}-\frac{1}{8\pi}\sum_{k=1}^{\infty}\frac{\cos\left(\frac{\pi k\log(b)}{\log(c)}\right)\cos\left(2\pi k\frac{\log(x)}{\log(c)}\right)}{k\sin\left(\frac{\pi k\log(b)}{\log(c)}\right)}\\
&\quad-\frac{1}{8\pi}\sum_{k=1}^{\infty}\frac{\cos\left(\frac{\pi k\log(d)}{\log(b)}\right)\cos\left(2\pi k\frac{\log(x)}{\log(b)}\right)}{k\sin\left(\frac{\pi k\log(d)}{\log(b)}\right)}-\frac{1}{8\pi}\sum_{k=1}^{\infty}\frac{\cos\left(\frac{\pi k\log(b)}{\log(d)}\right)\cos\left(2\pi k\frac{\log(x)}{\log(d)}\right)}{k\sin\left(\frac{\pi k\log(b)}{\log(d)}\right)}\\
&\quad-\frac{1}{8\pi}\sum_{k=1}^{\infty}\frac{\cos\left(\frac{\pi k\log(d)}{\log(c)}\right)\cos\left(2\pi k\frac{\log(x)}{\log(c)}\right)}{k\sin\left(\frac{\pi k\log(d)}{\log(c)}\right)}-\frac{1}{8\pi}\sum_{k=1}^{\infty}\frac{\cos\left(\frac{\pi k\log(c)}{\log(d)}\right)\cos\left(2\pi k\frac{\log(x)}{\log(d)}\right)}{k\sin\left(\frac{\pi k\log(c)}{\log(d)}\right)}\\
&\quad-\frac{1}{8\pi}\sum_{k=1}^{\infty}\frac{\cos\left(\frac{\pi k\log(b)}{\log(a)}\right)\cos\left(\frac{\pi k\log(c)}{\log(a)}\right)\sin\left(2\pi k\frac{\log(x)}{\log(a)}\right)}{k\sin\left(\frac{\pi k\log(b)}{\log(a)}\right)\sin\left(\frac{\pi k\log(c)}{\log(a)}\right)}\\
&\quad-\frac{1}{8\pi}\sum_{k=1}^{\infty}\frac{\cos\left(\frac{\pi k\log(c)}{\log(a)}\right)\cos\left(\frac{\pi k\log(d)}{\log(a)}\right)\sin\left(2\pi k\frac{\log(x)}{\log(a)}\right)}{k\sin\left(\frac{\pi k\log(c)}{\log(a)}\right)\sin\left(\frac{\pi k\log(d)}{\log(a)}\right)}\\
&\quad-\frac{1}{8\pi}\sum_{k=1}^{\infty}\frac{\cos\left(\frac{\pi k\log(b)}{\log(a)}\right)\cos\left(\frac{\pi k\log(d)}{\log(a)}\right)\sin\left(2\pi k\frac{\log(x)}{\log(a)}\right)}{k\sin\left(\frac{\pi k\log(b)}{\log(a)}\right)\sin\left(\frac{\pi k\log(d)}{\log(a)}\right)}\\
&\quad-\frac{1}{8\pi}\sum_{k=1}^{\infty}\frac{\cos\left(\frac{\pi k\log(a)}{\log(b)}\right)\cos\left(\frac{\pi k\log(c)}{\log(b)}\right)\sin\left(2\pi k\frac{\log(x)}{\log(b)}\right)}{k\sin\left(\frac{\pi k\log(a)}{\log(b)}\right)\sin\left(\frac{\pi k\log(c)}{\log(b)}\right)}\\
&\quad-\frac{1}{8\pi}\sum_{k=1}^{\infty}\frac{\cos\left(\frac{\pi k\log(c)}{\log(b)}\right)\cos\left(\frac{\pi k\log(d)}{\log(b)}\right)\sin\left(2\pi k\frac{\log(x)}{\log(b)}\right)}{k\sin\left(\frac{\pi k\log(c)}{\log(b)}\right)\sin\left(\frac{\pi k\log(d)}{\log(b)}\right)}\\
&\quad-\frac{1}{8\pi}\sum_{k=1}^{\infty}\frac{\cos\left(\frac{\pi k\log(a)}{\log(b)}\right)\cos\left(\frac{\pi k\log(d)}{\log(b)}\right)\sin\left(2\pi k\frac{\log(x)}{\log(b)}\right)}{k\sin\left(\frac{\pi k\log(a)}{\log(b)}\right)\sin\left(\frac{\pi k\log(d)}{\log(b)}\right)}\\
&\quad-\frac{1}{8\pi}\sum_{k=1}^{\infty}\frac{\cos\left(\frac{\pi k\log(a)}{\log(c)}\right)\cos\left(\frac{\pi k\log(b)}{\log(c)}\right)\sin\left(2\pi k\frac{\log(x)}{\log(c)}\right)}{k\sin\left(\frac{\pi k\log(a)}{\log(c)}\right)\sin\left(\frac{\pi k\log(b)}{\log(c)}\right)}\\
&\quad-\frac{1}{8\pi}\sum_{k=1}^{\infty}\frac{\cos\left(\frac{\pi k\log(b)}{\log(c)}\right)\cos\left(\frac{\pi k\log(d)}{\log(c)}\right)\sin\left(2\pi k\frac{\log(x)}{\log(c)}\right)}{k\sin\left(\frac{\pi k\log(b)}{\log(c)}\right)\sin\left(\frac{\pi k\log(d)}{\log(c)}\right)}\\
\end{split}
\end{displaymath}
\begin{displaymath}
\begin{split}
\hspace{1.72cm}&\quad-\frac{1}{8\pi}\sum_{k=1}^{\infty}\frac{\cos\left(\frac{\pi k\log(a)}{\log(c)}\right)\cos\left(\frac{\pi k\log(d)}{\log(c)}\right)\sin\left(2\pi k\frac{\log(x)}{\log(c)}\right)}{k\sin\left(\frac{\pi k\log(a)}{\log(c)}\right)\sin\left(\frac{\pi k\log(d)}{\log(c)}\right)}\\
&\quad-\frac{1}{8\pi}\sum_{k=1}^{\infty}\frac{\cos\left(\frac{\pi k\log(a)}{\log(d)}\right)\cos\left(\frac{\pi k\log(b)}{\log(d)}\right)\sin\left(2\pi k\frac{\log(x)}{\log(d)}\right)}{k\sin\left(\frac{\pi k\log(a)}{\log(d)}\right)\sin\left(\frac{\pi k\log(b)}{\log(d)}\right)}\\
&\quad-\frac{1}{8\pi}\sum_{k=1}^{\infty}\frac{\cos\left(\frac{\pi k\log(b)}{\log(d)}\right)\cos\left(\frac{\pi k\log(c)}{\log(d)}\right)\sin\left(2\pi k\frac{\log(x)}{\log(d)}\right)}{k\sin\left(\frac{\pi k\log(b)}{\log(d)}\right)\sin\left(\frac{\pi k\log(c)}{\log(d)}\right)}\\
&\quad-\frac{1}{8\pi}\sum_{k=1}^{\infty}\frac{\cos\left(\frac{\pi k\log(a)}{\log(d)}\right)\cos\left(\frac{\pi k\log(c)}{\log(d)}\right)\sin\left(2\pi k\frac{\log(x)}{\log(d)}\right)}{k\sin\left(\frac{\pi k\log(a)}{\log(d)}\right)\sin\left(\frac{\pi k\log(c)}{\log(d)}\right)}\\
&\quad+\frac{1}{8\pi}\sum_{k=1}^{\infty}\frac{\cos\left(\frac{\pi k\log(b)}{\log(a)}\right)\cos\left(\frac{\pi k\log(c)}{\log(a)}\right)\cos\left(\frac{\pi k\log(d)}{\log(a)}\right)\cos\left(2\pi k\frac{\log(x)}{\log(a)}\right)}{k\sin\left(\frac{\pi k\log(b)}{\log(a)}\right)\sin\left(\frac{\pi k\log(c)}{\log(a)}\right)\sin\left(\frac{\pi k\log(d)}{\log(a)}\right)}\\
&\quad+\frac{1}{8\pi}\sum_{k=1}^{\infty}\frac{\cos\left(\frac{\pi k\log(a)}{\log(b)}\right)\cos\left(\frac{\pi k\log(c)}{\log(b)}\right)\cos\left(\frac{\pi k\log(d)}{\log(b)}\right)\cos\left(2\pi k\frac{\log(x)}{\log(b)}\right)}{k\sin\left(\frac{\pi k\log(a)}{\log(b)}\right)\sin\left(\frac{\pi k\log(c)}{\log(b)}\right)\sin\left(\frac{\pi k\log(d)}{\log(b)}\right)}\\
&\quad+\frac{1}{8\pi}\sum_{k=1}^{\infty}\frac{\cos\left(\frac{\pi k\log(a)}{\log(c)}\right)\cos\left(\frac{\pi k\log(b)}{\log(c)}\right)\cos\left(\frac{\pi k\log(d)}{\log(c)}\right)\cos\left(2\pi k\frac{\log(x)}{\log(c)}\right)}{k\sin\left(\frac{\pi k\log(a)}{\log(c)}\right)\sin\left(\frac{\pi k\log(b)}{\log(c)}\right)\sin\left(\frac{\pi k\log(d)}{\log(c)}\right)}\\
&\quad+\frac{1}{8\pi}\sum_{k=1}^{\infty}\frac{\cos\left(\frac{\pi k\log(a)}{\log(d)}\right)\cos\left(\frac{\pi k\log(b)}{\log(d)}\right)\cos\left(\frac{\pi k\log(c)}{\log(d)}\right)\cos\left(2\pi k\frac{\log(x)}{\log(d)}\right)}{k\sin\left(\frac{\pi k\log(a)}{\log(d)}\right)\sin\left(\frac{\pi k\log(b)}{\log(d)}\right)\sin\left(\frac{\pi k\log(c)}{\log(d)}\right)}+\frac{1}{2}\chi_{S_{a,b,c,d}}(x)
\end{split}
\end{displaymath}

\noindent and

\begin{displaymath}
\begin{split}
N_{2,3,5,7}(x)&=\frac{\log(x)^{4}}{24\log(2)\log(3)\log(5)\log(7)}+\frac{\log(x)^{3}}{12\log(2)\log(3)\log(5)}+\frac{\log(x)^{3}}{12\log(2)\log(3)\log(7)}\\
&\quad+\frac{\log(x)^{3}}{12\log(2)\log(5)\log(7)}+\frac{\log(x)^{3}}{12\log(3)\log(5)\log(7)}+\frac{\log(2)\log(x)^{2}}{24\log(3)\log(5)\log(7)}\\
&\quad+\frac{\log(3)\log(x)^{2}}{24\log(2)\log(5)\log(7)}+\frac{\log(5)\log(x)^{2}}{24\log(2)\log(3)\log(7)}+\frac{\log(7)\log(x)^{2}}{24\log(2)\log(3)\log(5)}\\
&\quad+\frac{\log(x)^{2}}{8\log(2)\log(3)}+\frac{\log(x)^{2}}{8\log(2)\log(5)}+\frac{\log(x)^{2}}{8\log(2)\log(7)}+\frac{\log(x)^{2}}{8\log(3)\log(5)}+\frac{\log(x)^{2}}{8\log(3)\log(7)}\\
&\quad+\frac{\log(x)^{2}}{8\log(5)\log(7)}+\frac{\log(x)}{8\log(2)}+\frac{\log(x)}{8\log(3)}+\frac{\log(x)}{8\log(5)}+\frac{\log(x)}{8\log(7)}+\frac{\log(2)\log(x)}{24\log(3)\log(5)}\\
\end{split}
\end{displaymath}
\begin{displaymath}
\begin{split}
\hspace{1.75cm}&\quad+\frac{\log(2)\log(x)}{24\log(3)\log(7)}+\frac{\log(2)\log(x)}{24\log(5)\log(7)}+\frac{\log(3)\log(x)}{24\log(2)\log(5)}+\frac{\log(3)\log(x)}{24\log(2)\log(7)}+\frac{\log(3)\log(x)}{24\log(5)\log(7)}\\
&\quad+\frac{\log(5)\log(x)}{24\log(2)\log(3)}+\frac{\log(5)\log(x)}{24\log(2)\log(7)}+\frac{\log(5)\log(x)}{24\log(3)\log(7)}+\frac{\log(7)\log(x)}{24\log(2)\log(3)}+\frac{\log(7)\log(x)}{24\log(2)\log(5)}\\
&\quad+\frac{\log(7)\log(x)}{24\log(3)\log(5)}+\frac{1}{16}+\frac{\log(2)}{48\log(3)}+\frac{\log(2)}{48\log(5)}+\frac{\log(2)}{48\log(7)}+\frac{\log(3)}{48\log(2)}+\frac{\log(3)}{48\log(5)}\\
&\quad+\frac{\log(3)}{48\log(7)}+\frac{\log(5)}{48\log(2)}+\frac{\log(5)}{48\log(3)}+\frac{\log(5)}{48\log(7)}+\frac{\log(7)}{48\log(2)}+\frac{\log(7)}{48\log(3)}+\frac{\log(7)}{48\log(5)}\\
&\quad+\frac{\log(2)\log(3)}{144\log(5)\log(7)}+\frac{\log(2)\log(5)}{144\log(3)\log(7)}+\frac{\log(2)\log(7)}{144\log(3)\log(5)}+\frac{\log(3)\log(5)}{144\log(2)\log(7)}\\
&\quad+\frac{\log(3)\log(7)}{144\log(2)\log(5)}+\frac{\log(5)\log(7)}{144\log(2)\log(3)}-\frac{\log(2)^{3}}{720\log(3)\log(5)\log(7)}-\frac{\log(3)^{3}}{720\log(2)\log(5)\log(7)}\\
&\quad-\frac{\log(5)^{3}}{720\log(2)\log(3)\log(7)}-\frac{\log(7)^{3}}{720\log(2)\log(3)\log(5)}-\frac{1}{8}B^{*}_{1}\left(\left\{\frac{\log(x)}{\log(2)}\right\}\right)\\
&\quad-\frac{1}{8}B^{*}_{1}\left(\left\{\frac{\log(x)}{\log(3)}\right\}\right)-\frac{1}{8}B^{*}_{1}\left(\left\{\frac{\log(x)}{\log(5)}\right\}\right)-\frac{1}{8}B^{*}_{1}\left(\left\{\frac{\log(x)}{\log(7)}\right\}\right)\\
&\quad-\frac{1}{8\pi}\sum_{k=1}^{\infty}\frac{\cos\left(\frac{\pi k\log(3)}{\log(2)}\right)\cos\left(2\pi k\frac{\log(x)}{\log(2)}\right)}{k\sin\left(\frac{\pi k\log(3)}{\log(2)}\right)}-\frac{1}{8\pi}\sum_{k=1}^{\infty}\frac{\cos\left(\frac{\pi k\log(2)}{\log(3)}\right)\cos\left(2\pi k\frac{\log(x)}{\log(3)}\right)}{k\sin\left(\frac{\pi k\log(2)}{\log(3)}\right)}\\
&\quad-\frac{1}{8\pi}\sum_{k=1}^{\infty}\frac{\cos\left(\frac{\pi k\log(5)}{\log(2)}\right)\cos\left(2\pi k\frac{\log(x)}{\log(2)}\right)}{k\sin\left(\frac{\pi k\log(5)}{\log(2)}\right)}-\frac{1}{8\pi}\sum_{k=1}^{\infty}\frac{\cos\left(\frac{\pi k\log(2)}{\log(5)}\right)\cos\left(2\pi k\frac{\log(x)}{\log(5)}\right)}{k\sin\left(\frac{\pi k\log(2)}{\log(5)}\right)}\\
&\quad-\frac{1}{8\pi}\sum_{k=1}^{\infty}\frac{\cos\left(\frac{\pi k\log(7)}{\log(2)}\right)\cos\left(2\pi k\frac{\log(x)}{\log(2)}\right)}{k\sin\left(\frac{\pi k\log(7)}{\log(2)}\right)}-\frac{1}{8\pi}\sum_{k=1}^{\infty}\frac{\cos\left(\frac{\pi k\log(2)}{\log(7)}\right)\cos\left(2\pi k\frac{\log(x)}{\log(7)}\right)}{k\sin\left(\frac{\pi k\log(2)}{\log(7)}\right)}\\
&\quad-\frac{1}{8\pi}\sum_{k=1}^{\infty}\frac{\cos\left(\frac{\pi k\log(5)}{\log(3)}\right)\cos\left(2\pi k\frac{\log(x)}{\log(3)}\right)}{k\sin\left(\frac{\pi k\log(5)}{\log(3)}\right)}-\frac{1}{8\pi}\sum_{k=1}^{\infty}\frac{\cos\left(\frac{\pi k\log(3)}{\log(5)}\right)\cos\left(2\pi k\frac{\log(x)}{\log(5)}\right)}{k\sin\left(\frac{\pi k\log(3)}{\log(5)}\right)}\\
&\quad-\frac{1}{8\pi}\sum_{k=1}^{\infty}\frac{\cos\left(\frac{\pi k\log(7)}{\log(3)}\right)\cos\left(2\pi k\frac{\log(x)}{\log(3)}\right)}{k\sin\left(\frac{\pi k\log(7)}{\log(3)}\right)}-\frac{1}{8\pi}\sum_{k=1}^{\infty}\frac{\cos\left(\frac{\pi k\log(3)}{\log(7)}\right)\cos\left(2\pi k\frac{\log(x)}{\log(7)}\right)}{k\sin\left(\frac{\pi k\log(3)}{\log(7)}\right)}\\
&\quad-\frac{1}{8\pi}\sum_{k=1}^{\infty}\frac{\cos\left(\frac{\pi k\log(7)}{\log(5)}\right)\cos\left(2\pi k\frac{\log(x)}{\log(5)}\right)}{k\sin\left(\frac{\pi k\log(7)}{\log(5)}\right)}-\frac{1}{8\pi}\sum_{k=1}^{\infty}\frac{\cos\left(\frac{\pi k\log(5)}{\log(7)}\right)\cos\left(2\pi k\frac{\log(x)}{\log(7)}\right)}{k\sin\left(\frac{\pi k\log(5)}{\log(7)}\right)}\\
&\quad-\frac{1}{8\pi}\sum_{k=1}^{\infty}\frac{\cos\left(\frac{\pi k\log(3)}{\log(2)}\right)\cos\left(\frac{\pi k\log(5)}{\log(2)}\right)\sin\left(2\pi k\frac{\log(x)}{\log(2)}\right)}{k\sin\left(\frac{\pi k\log(3)}{\log(2)}\right)\sin\left(\frac{\pi k\log(5)}{\log(2)}\right)}\\
\end{split}
\end{displaymath}
\begin{displaymath}
\begin{split}
&\quad-\frac{1}{8\pi}\sum_{k=1}^{\infty}\frac{\cos\left(\frac{\pi k\log(5)}{\log(2)}\right)\cos\left(\frac{\pi k\log(7)}{\log(2)}\right)\sin\left(2\pi k\frac{\log(x)}{\log(2)}\right)}{k\sin\left(\frac{\pi k\log(5)}{\log(2)}\right)\sin\left(\frac{\pi k\log(7)}{\log(2)}\right)}\\
&\quad-\frac{1}{8\pi}\sum_{k=1}^{\infty}\frac{\cos\left(\frac{\pi k\log(3)}{\log(2)}\right)\cos\left(\frac{\pi k\log(7)}{\log(2)}\right)\sin\left(2\pi k\frac{\log(x)}{\log(2)}\right)}{k\sin\left(\frac{\pi k\log(3)}{\log(2)}\right)\sin\left(\frac{\pi k\log(7)}{\log(2)}\right)}\hspace{3.18cm}\\
&\quad-\frac{1}{8\pi}\sum_{k=1}^{\infty}\frac{\cos\left(\frac{\pi k\log(2)}{\log(3)}\right)\cos\left(\frac{\pi k\log(5)}{\log(3)}\right)\sin\left(2\pi k\frac{\log(x)}{\log(3)}\right)}{k\sin\left(\frac{\pi k\log(2)}{\log(3)}\right)\sin\left(\frac{\pi k\log(5)}{\log(3)}\right)}\\
&\quad-\frac{1}{8\pi}\sum_{k=1}^{\infty}\frac{\cos\left(\frac{\pi k\log(5)}{\log(3)}\right)\cos\left(\frac{\pi k\log(7)}{\log(3)}\right)\sin\left(2\pi k\frac{\log(x)}{\log(3)}\right)}{k\sin\left(\frac{\pi k\log(5)}{\log(3)}\right)\sin\left(\frac{\pi k\log(7)}{\log(3)}\right)}\\
&\quad-\frac{1}{8\pi}\sum_{k=1}^{\infty}\frac{\cos\left(\frac{\pi k\log(2)}{\log(3)}\right)\cos\left(\frac{\pi k\log(7)}{\log(3)}\right)\sin\left(2\pi k\frac{\log(x)}{\log(3)}\right)}{k\sin\left(\frac{\pi k\log(2)}{\log(3)}\right)\sin\left(\frac{\pi k\log(7)}{\log(3)}\right)}\\
&\quad-\frac{1}{8\pi}\sum_{k=1}^{\infty}\frac{\cos\left(\frac{\pi k\log(2)}{\log(5)}\right)\cos\left(\frac{\pi k\log(3)}{\log(5)}\right)\sin\left(2\pi k\frac{\log(x)}{\log(5)}\right)}{k\sin\left(\frac{\pi k\log(2)}{\log(5)}\right)\sin\left(\frac{\pi k\log(3)}{\log(5)}\right)}\\
&\quad-\frac{1}{8\pi}\sum_{k=1}^{\infty}\frac{\cos\left(\frac{\pi k\log(3)}{\log(5)}\right)\cos\left(\frac{\pi k\log(7)}{\log(5)}\right)\sin\left(2\pi k\frac{\log(x)}{\log(5)}\right)}{k\sin\left(\frac{\pi k\log(3)}{\log(5)}\right)\sin\left(\frac{\pi k\log(7)}{\log(5)}\right)}\\
&\quad-\frac{1}{8\pi}\sum_{k=1}^{\infty}\frac{\cos\left(\frac{\pi k\log(2)}{\log(5)}\right)\cos\left(\frac{\pi k\log(7)}{\log(5)}\right)\sin\left(2\pi k\frac{\log(x)}{\log(5)}\right)}{k\sin\left(\frac{\pi k\log(2)}{\log(5)}\right)\sin\left(\frac{\pi k\log(7)}{\log(5)}\right)}\\
&\quad-\frac{1}{8\pi}\sum_{k=1}^{\infty}\frac{\cos\left(\frac{\pi k\log(2)}{\log(7)}\right)\cos\left(\frac{\pi k\log(3)}{\log(7)}\right)\sin\left(2\pi k\frac{\log(x)}{\log(7)}\right)}{k\sin\left(\frac{\pi k\log(2)}{\log(7)}\right)\sin\left(\frac{\pi k\log(3)}{\log(7)}\right)}\\
&\quad-\frac{1}{8\pi}\sum_{k=1}^{\infty}\frac{\cos\left(\frac{\pi k\log(3)}{\log(7)}\right)\cos\left(\frac{\pi k\log(5)}{\log(7)}\right)\sin\left(2\pi k\frac{\log(x)}{\log(7)}\right)}{k\sin\left(\frac{\pi k\log(3)}{\log(7)}\right)\sin\left(\frac{\pi k\log(5)}{\log(7)}\right)}\\
&\quad-\frac{1}{8\pi}\sum_{k=1}^{\infty}\frac{\cos\left(\frac{\pi k\log(2)}{\log(7)}\right)\cos\left(\frac{\pi k\log(5)}{\log(7)}\right)\sin\left(2\pi k\frac{\log(x)}{\log(7)}\right)}{k\sin\left(\frac{\pi k\log(2)}{\log(7)}\right)\sin\left(\frac{\pi k\log(5)}{\log(7)}\right)}\\
&\quad+\frac{1}{8\pi}\sum_{k=1}^{\infty}\frac{\cos\left(\frac{\pi k\log(3)}{\log(2)}\right)\cos\left(\frac{\pi k\log(5)}{\log(2)}\right)\cos\left(\frac{\pi k\log(7)}{\log(2)}\right)\cos\left(2\pi k\frac{\log(x)}{\log(2)}\right)}{k\sin\left(\frac{\pi k\log(3)}{\log(2)}\right)\sin\left(\frac{\pi k\log(5)}{\log(2)}\right)\sin\left(\frac{\pi k\log(7)}{\log(2)}\right)}\\
\end{split}
\end{displaymath}
\begin{displaymath}
\begin{split}
\hspace{1.75cm}&\quad+\frac{1}{8\pi}\sum_{k=1}^{\infty}\frac{\cos\left(\frac{\pi k\log(2)}{\log(3)}\right)\cos\left(\frac{\pi k\log(5)}{\log(3)}\right)\cos\left(\frac{\pi k\log(7)}{\log(3)}\right)\cos\left(2\pi k\frac{\log(x)}{\log(3)}\right)}{k\sin\left(\frac{\pi k\log(2)}{\log(3)}\right)\sin\left(\frac{\pi k\log(5)}{\log(3)}\right)\sin\left(\frac{\pi k\log(7)}{\log(3)}\right)}\\
&\quad+\frac{1}{8\pi}\sum_{k=1}^{\infty}\frac{\cos\left(\frac{\pi k\log(2)}{\log(5)}\right)\cos\left(\frac{\pi k\log(3)}{\log(5)}\right)\cos\left(\frac{\pi k\log(7)}{\log(5)}\right)\cos\left(2\pi k\frac{\log(x)}{\log(5)}\right)}{k\sin\left(\frac{\pi k\log(2)}{\log(5)}\right)\sin\left(\frac{\pi k\log(3)}{\log(5)}\right)\sin\left(\frac{\pi k\log(7)}{\log(5)}\right)}\\
&\quad+\frac{1}{8\pi}\sum_{k=1}^{\infty}\frac{\cos\left(\frac{\pi k\log(2)}{\log(7)}\right)\cos\left(\frac{\pi k\log(3)}{\log(7)}\right)\cos\left(\frac{\pi k\log(5)}{\log(7)}\right)\cos\left(2\pi k\frac{\log(x)}{\log(7)}\right)}{k\sin\left(\frac{\pi k\log(2)}{\log(7)}\right)\sin\left(\frac{\pi k\log(3)}{\log(7)}\right)\sin\left(\frac{\pi k\log(5)}{\log(7)}\right)}+\frac{1}{2}\chi_{S_{2,3,5,7}}(x).
\end{split}
\end{displaymath}
\end{corollary}

\begin{proof}
\noindent We have that
\begin{displaymath}
\begin{split}
\sum_{k=1}^{\infty}\frac{\chi_{S_{a,b,c,d}}(k)}{k^{s}}&=\left(\sum_{m_{1}=0}^{\infty}\frac{1}{a^{m_{1}s}}\right)\left(\sum_{m_{2}=0}^{\infty}\frac{1}{b^{m_{2}s}}\right)\left(\sum_{m_{3}=0}^{\infty}\frac{1}{c^{m_{3}s}}\right)\left(\sum_{m_{4}=0}^{\infty}\frac{1}{d^{m_{4}s}}\right)\\
&=\frac{1}{\left(1-e^{-\log(a)s}\right)\left(1-e^{-\log(b)s}\right)\left(1-e^{-\log(c)s}\right)\left(1-e^{-\log(d)s}\right)}.
\end{split}
\end{displaymath}
Therefore, by Perron's formula, we get that
\begin{displaymath}
\begin{split}
N_{a,b,c,d}(x)&=\frac{1}{2\pi i}\int_{\gamma}\frac{x^{s}}{s\left(1-e^{-\log(a)s}\right)\left(1-e^{-\log(b)s}\right)\left(1-e^{-\log(c)s}\right)\left(1-e^{-\log(d)s}\right)}ds\\
&=\Res_{s=0}\left(\frac{x^{s}}{s\left(1-e^{-\log(a)s}\right)\left(1-e^{-\log(b)s}\right)\left(1-e^{-\log(c)s}\right)\left(1-e^{-\log(d)s}\right)}\right)\\
&\quad+\sum_{k=1}^{\infty}\Res_{s=\frac{2\pi ik}{\log(a)}}\left(\frac{x^{s}}{s\left(1-e^{-\log(a)s}\right)\left(1-e^{-\log(b)s}\right)\left(1-e^{-\log(c)s}\right)\left(1-e^{-\log(d)s}\right)}\right)\\
&\quad+\sum_{k=1}^{\infty}\Res_{s=-\frac{2\pi ik}{\log(a)}}\left(\frac{x^{s}}{s\left(1-e^{-\log(a)s}\right)\left(1-e^{-\log(b)s}\right)\left(1-e^{-\log(c)s}\right)\left(1-e^{-\log(d)s}\right)}\right)\\
&\quad+\sum_{k=1}^{\infty}\Res_{s=\frac{2\pi ik}{\log(b)}}\left(\frac{x^{s}}{s\left(1-e^{-\log(a)s}\right)\left(1-e^{-\log(b)s}\right)\left(1-e^{-\log(c)s}\right)\left(1-e^{-\log(d)s}\right)}\right)\\
&\quad+\sum_{k=1}^{\infty}\Res_{s=-\frac{2\pi ik}{\log(b)}}\left(\frac{x^{s}}{s\left(1-e^{-\log(a)s}\right)\left(1-e^{-\log(b)s}\right)\left(1-e^{-\log(c)s}\right)\left(1-e^{-\log(d)s}\right)}\right)\\
&\quad+\sum_{k=1}^{\infty}\Res_{s=\frac{2\pi ik}{\log(c)}}\left(\frac{x^{s}}{s\left(1-e^{-\log(a)s}\right)\left(1-e^{-\log(b)s}\right)\left(1-e^{-\log(c)s}\right)\left(1-e^{-\log(d)s}\right)}\right)\\
&\quad+\sum_{k=1}^{\infty}\Res_{s=-\frac{2\pi ik}{\log(c)}}\left(\frac{x^{s}}{s\left(1-e^{-\log(a)s}\right)\left(1-e^{-\log(b)s}\right)\left(1-e^{-\log(c)s}\right)\left(1-e^{-\log(d)s}\right)}\right)\\
&\quad+\sum_{k=1}^{\infty}\Res_{s=\frac{2\pi ik}{\log(d)}}\left(\frac{x^{s}}{s\left(1-e^{-\log(a)s}\right)\left(1-e^{-\log(b)s}\right)\left(1-e^{-\log(c)s}\right)\left(1-e^{-\log(d)s}\right)}\right)\\
\end{split}
\end{displaymath}
\begin{displaymath}
\begin{split}
\hspace{1.7cm}&\quad+\sum_{k=1}^{\infty}\Res_{s=-\frac{2\pi ik}{\log(d)}}\left(\frac{x^{s}}{s\left(1-e^{-\log(a)s}\right)\left(1-e^{-\log(b)s}\right)\left(1-e^{-\log(c)s}\right)\left(1-e^{-\log(d)s}\right)}\right)\\
&\quad+\frac{1}{2}\chi_{S_{a,b,c,d}}(x),
\end{split}
\end{displaymath}
where $\gamma=\text{line from $1-i\infty$ to $1+i\infty$}$.\\
For the Residues, we have that
\begin{displaymath}
\begin{split}
&\Res_{s=0}\left(\frac{x^{s}}{s\left(1-e^{-\log(a)s}\right)\left(1-e^{-\log(b)s}\right)\left(1-e^{-\log(c)s}\right)\left(1-e^{-\log(d)s}\right)}\right)\\
&=\frac{\log(x)^{4}}{24\log(a)\log(b)\log(c)\log(d)}+\frac{\log(x)^{3}}{12\log(a)\log(b)\log(c)}+\frac{\log(x)^{3}}{12\log(a)\log(b)\log(d)}\\
&\quad+\frac{\log(x)^{3}}{12\log(a)\log(c)\log(d)}+\frac{\log(x)^{3}}{12\log(b)\log(c)\log(d)}+\frac{\log(a)\log(x)^{2}}{24\log(b)\log(c)\log(d)}\\
&\quad+\frac{\log(b)\log(x)^{2}}{24\log(a)\log(c)\log(d)}+\frac{\log(c)\log(x)^{2}}{24\log(a)\log(b)\log(d)}+\frac{\log(d)\log(x)^{2}}{24\log(a)\log(b)\log(c)}\\
&\quad+\frac{\log(x)^{2}}{8\log(a)\log(b)}+\frac{\log(x)^{2}}{8\log(a)\log(c)}+\frac{\log(x)^{2}}{8\log(a)\log(d)}+\frac{\log(x)^{2}}{8\log(b)\log(c)}+\frac{\log(x)^{2}}{8\log(b)\log(d)}\\
&\quad+\frac{\log(x)^{2}}{8\log(c)\log(d)}+\frac{\log(x)}{8\log(a)}+\frac{\log(x)}{8\log(b)}+\frac{\log(x)}{8\log(c)}+\frac{\log(x)}{8\log(d)}+\frac{\log(a)\log(x)}{24\log(b)\log(c)}\\
&\quad+\frac{\log(a)\log(x)}{24\log(b)\log(d)}+\frac{\log(a)\log(x)}{24\log(c)\log(d)}+\frac{\log(b)\log(x)}{24\log(a)\log(c)}+\frac{\log(b)\log(x)}{24\log(a)\log(d)}+\frac{\log(b)\log(x)}{24\log(c)\log(d)}\\
&\quad+\frac{\log(c)\log(x)}{24\log(a)\log(b)}+\frac{\log(c)\log(x)}{24\log(a)\log(d)}+\frac{\log(c)\log(x)}{24\log(b)\log(d)}+\frac{\log(d)\log(x)}{24\log(a)\log(b)}+\frac{\log(d)\log(x)}{24\log(a)\log(c)}\\
&\quad+\frac{\log(d)\log(x)}{24\log(b)\log(c)}+\frac{1}{16}+\frac{\log(a)}{48\log(b)}+\frac{\log(a)}{48\log(c)}+\frac{\log(a)}{48\log(d)}+\frac{\log(b)}{48\log(a)}+\frac{\log(b)}{48\log(c)}\\
&\quad+\frac{\log(b)}{48\log(d)}+\frac{\log(c)}{48\log(a)}+\frac{\log(c)}{48\log(b)}+\frac{\log(c)}{48\log(d)}+\frac{\log(d)}{48\log(a)}+\frac{\log(d)}{48\log(b)}+\frac{\log(d)}{48\log(c)}\\
&\quad+\frac{\log(a)\log(b)}{144\log(c)\log(d)}+\frac{\log(a)\log(c)}{144\log(b)\log(d)}+\frac{\log(a)\log(d)}{144\log(b)\log(c)}+\frac{\log(b)\log(c)}{144\log(a)\log(d)}\\
&\quad+\frac{\log(b)\log(d)}{144\log(a)\log(c)}+\frac{\log(c)\log(d)}{144\log(a)\log(b)}-\frac{\log(a)^{3}}{720\log(b)\log(c)\log(d)}-\frac{\log(b)^{3}}{720\log(a)\log(c)\log(d)}\\
&\quad-\frac{\log(c)^{3}}{720\log(a)\log(b)\log(d)}-\frac{\log(d)^{3}}{720\log(a)\log(b)\log(c)}
\end{split}
\end{displaymath}
and that
\begin{displaymath}
\begin{split}
&\Res_{s=\frac{2\pi ik}{\log(a)}}\left(\frac{x^{s}}{s\left(1-e^{-\log(a)s}\right)\left(1-e^{-\log(b)s}\right)\left(1-e^{-\log(c)s}\right)\left(1-e^{-\log(d)s}\right)}\right)\\
&=-\frac{ib^{\frac{2\pi ik}{\log(a)}}c^{\frac{2\pi ik}{\log(a)}}d^{\frac{2\pi ik}{\log(a)}}x^{\frac{2\pi ik}{\log(a)}}}{2\pi k\left(b^{\frac{2\pi ik}{\log(a)}}-1\right)\left(c^{\frac{2\pi ik}{\log(a)}}-1\right)\left(d^{\frac{2\pi ik}{\log(a)}}-1\right)}\;\;\forall k\in\mathbb{N}
\end{split}
\end{displaymath}
\begin{displaymath}
\begin{split}
&\Res_{s=-\frac{2\pi ik}{\log(a)}}\left(\frac{x^{s}}{s\left(1-e^{-\log(a)s}\right)\left(1-e^{-\log(b)s}\right)\left(1-e^{-\log(c)s}\right)\left(1-e^{-\log(d)s}\right)}\right)\\
&=-\frac{ix^{-\frac{2\pi ik}{\log(a)}}}{2\pi k\left(b^{\frac{2\pi ik}{\log(a)}}-1\right)\left(c^{\frac{2\pi ik}{\log(a)}}-1\right)\left(d^{\frac{2\pi ik}{\log(a)}}-1\right)}\;\;\forall k\in\mathbb{N}.
\end{split}
\end{displaymath}
Exactly similar relations hold also for the other Residues under exchanging $a$ with $b$, $a$ with $c$, and $a$ with $d$ ("permuting $a,b,c,d$").
Using the relations
\begin{displaymath}
\begin{split}
\sin\left(2\pi k\frac{\log(x)}{\log(a)}\right)&=\frac{1}{2}ix^{-\frac{2\pi ik}{\log(a)}}-\frac{1}{2}ix^{\frac{2\pi ik}{\log(a)}}\\
\cos\left(2\pi k\frac{\log(x)}{\log(a)}\right)&=\frac{1}{2}x^{-\frac{2\pi ik}{\log(a)}}+\frac{1}{2}x^{\frac{2\pi ik}{\log(a)}}\\
\cot\left(\frac{\pi k\log(b)}{\log(a)}\right)&=i\frac{b^{\frac{2\pi ik}{\log(a)}}+1}{b^{\frac{2\pi ik}{\log(a)}}-1},
\end{split}
\end{displaymath}
we establish (by expanding everything out) the following identity
\begin{displaymath}
\begin{split}
&\cot\left(\frac{\pi k\log(b)}{\log(a)}\right)\cot\left(\frac{\pi k\log(c)}{\log(a)}\right)\cot\left(\frac{\pi k\log(d)}{\log(a)}\right)\cos\left(2\pi k\frac{\log(x)}{\log(a)}\right)\\
&-\cot\left(\frac{\pi k\log(b)}{\log(a)}\right)\cot\left(\frac{\pi k\log(c)}{\log(a)}\right)\sin\left(2\pi k\frac{\log(x)}{\log(a)}\right)\\
&-\cot\left(\frac{\pi k\log(c)}{\log(a)}\right)\cot\left(\frac{\pi k\log(d)}{\log(a)}\right)\sin\left(2\pi k\frac{\log(x)}{\log(a)}\right)\\
&-\cot\left(\frac{\pi k\log(b)}{\log(a)}\right)\cot\left(\frac{\pi k\log(d)}{\log(a)}\right)\sin\left(2\pi k\frac{\log(x)}{\log(a)}\right)\\
&-\cot\left(\frac{\pi k\log(b)}{\log(a)}\right)\cos\left(2\pi k\frac{\log(x)}{\log(a)}\right)-\cot\left(\frac{\pi k\log(c)}{\log(a)}\right)\cos\left(2\pi k\frac{\log(x)}{\log(a)}\right)\\
&-\cot\left(\frac{\pi k\log(d)}{\log(a)}\right)\cos\left(2\pi k\frac{\log(x)}{\log(a)}\right)+\sin\left(2\pi k\frac{\log(x)}{\log(a)}\right)\\
&=\left(i\frac{b^{\frac{2\pi ik}{\log(a)}}+1}{b^{\frac{2\pi ik}{\log(a)}}-1}\right)\left(i\frac{c^{\frac{2\pi ik}{\log(a)}}+1}{c^{\frac{2\pi ik}{\log(a)}}-1}\right)\left(i\frac{d^{\frac{2\pi ik}{\log(a)}}+1}{d^{\frac{2\pi ik}{\log(a)}}-1}\right)\left(\frac{1}{2}x^{-\frac{2\pi ik}{\log(a)}}+\frac{1}{2}x^{\frac{2\pi ik}{\log(a)}}\right)\\
&-\left(i\frac{b^{\frac{2\pi ik}{\log(a)}}+1}{b^{\frac{2\pi ik}{\log(a)}}-1}\right)\left(i\frac{c^{\frac{2\pi ik}{\log(a)}}+1}{c^{\frac{2\pi ik}{\log(a)}}-1}\right)\left(\frac{1}{2}ix^{-\frac{2\pi ik}{\log(a)}}-\frac{1}{2}ix^{\frac{2\pi ik}{\log(a)}}\right)\\
&-\left(i\frac{c^{\frac{2\pi ik}{\log(a)}}+1}{c^{\frac{2\pi ik}{\log(a)}}-1}\right)\left(i\frac{d^{\frac{2\pi ik}{\log(a)}}+1}{d^{\frac{2\pi ik}{\log(a)}}-1}\right)\left(\frac{1}{2}ix^{-\frac{2\pi ik}{\log(a)}}-\frac{1}{2}ix^{\frac{2\pi ik}{\log(a)}}\right)\\
&-\left(i\frac{b^{\frac{2\pi ik}{\log(a)}}+1}{b^{\frac{2\pi ik}{\log(a)}}-1}\right)\left(i\frac{d^{\frac{2\pi ik}{\log(a)}}+1}{d^{\frac{2\pi ik}{\log(a)}}-1}\right)\left(\frac{1}{2}ix^{-\frac{2\pi ik}{\log(a)}}-\frac{1}{2}ix^{\frac{2\pi ik}{\log(a)}}\right)\\
\end{split}
\end{displaymath}
\begin{displaymath}
\begin{split}
\hspace{1.08cm}&-\left(i\frac{b^{\frac{2\pi ik}{\log(a)}}+1}{b^{\frac{2\pi ik}{\log(a)}}-1}\right)\left(\frac{1}{2}x^{-\frac{2\pi ik}{\log(a)}}+\frac{1}{2}x^{\frac{2\pi ik}{\log(a)}}\right)-\left(i\frac{c^{\frac{2\pi ik}{\log(a)}}+1}{c^{\frac{2\pi ik}{\log(a)}}-1}\right)\left(\frac{1}{2}x^{-\frac{2\pi ik}{\log(a)}}+\frac{1}{2}x^{\frac{2\pi ik}{\log(a)}}\right)\\
&-\left(i\frac{d^{\frac{2\pi ik}{\log(a)}}+1}{d^{\frac{2\pi ik}{\log(a)}}-1}\right)\left(\frac{1}{2}x^{-\frac{2\pi ik}{\log(a)}}+\frac{1}{2}x^{\frac{2\pi ik}{\log(a)}}\right)+\left(\frac{1}{2}ix^{-\frac{2\pi ik}{\log(a)}}-\frac{1}{2}ix^{\frac{2\pi ik}{\log(a)}}\right)\\
&=-\frac{4i\left(x^{-\frac{2\pi ik}{\log(a)}}+b^{\frac{2\pi ik}{\log(a)}}c^{\frac{2\pi ik}{\log(a)}}d^{\frac{2\pi ik}{\log(a)}}x^{\frac{2\pi ik}{\log(a)}}\right)}{\left(b^{\frac{2\pi ik}{\log(a)}}-1\right)\left(c^{\frac{2\pi ik}{\log(a)}}-1\right)\left(d^{\frac{2\pi ik}{\log(a)}}-1\right)}.
\end{split}
\end{displaymath}
This shows that for the first Residues (the "$a$ -Residues"), we have that
\begin{displaymath}
\begin{split}
&\Res_{s=\frac{2\pi ik}{\log(a)}}\left(\frac{x^{s}}{s\left(1-e^{-\log(a)s}\right)\left(1-e^{-\log(b)s}\right)\left(1-e^{-\log(c)s}\right)\left(1-e^{-\log(d)s}\right)}\right)\\
&+\Res_{s=-\frac{2\pi ik}{\log(a)}}\left(\frac{x^{s}}{s\left(1-e^{-\log(a)s}\right)\left(1-e^{-\log(b)s}\right)\left(1-e^{-\log(c)s}\right)\left(1-e^{-\log(d)s}\right)}\right)\\
&=-\frac{i\left(x^{-\frac{2\pi ik}{\log(a)}}+b^{\frac{2\pi ik}{\log(a)}}c^{\frac{2\pi ik}{\log(a)}}d^{\frac{2\pi ik}{\log(a)}}x^{\frac{2\pi ik}{\log(a)}}\right)}{2\pi k\left(b^{\frac{2\pi ik}{\log(a)}}-1\right)\left(c^{\frac{2\pi ik}{\log(a)}}-1\right)\left(d^{\frac{2\pi ik}{\log(a)}}-1\right)}\\
&=\frac{1}{8\pi k}\bigg(\cot\left(\frac{\pi k\log(b)}{\log(a)}\right)\cot\left(\frac{\pi k\log(c)}{\log(a)}\right)\cot\left(\frac{\pi k\log(d)}{\log(a)}\right)\cos\left(2\pi k\frac{\log(x)}{\log(a)}\right)\\
&\hspace{1.6cm}-\cot\left(\frac{\pi k\log(b)}{\log(a)}\right)\cot\left(\frac{\pi k\log(c)}{\log(a)}\right)\sin\left(2\pi k\frac{\log(x)}{\log(a)}\right)\\
&\hspace{1.6cm}-\cot\left(\frac{\pi k\log(c)}{\log(a)}\right)\cot\left(\frac{\pi k\log(d)}{\log(a)}\right)\sin\left(2\pi k\frac{\log(x)}{\log(a)}\right)\\
&\hspace{1.6cm}-\cot\left(\frac{\pi k\log(b)}{\log(a)}\right)\cot\left(\frac{\pi k\log(d)}{\log(a)}\right)\sin\left(2\pi k\frac{\log(x)}{\log(a)}\right)\\
&\hspace{1.6cm}-\cot\left(\frac{\pi k\log(b)}{\log(a)}\right)\cos\left(2\pi k\frac{\log(x)}{\log(a)}\right)-\cot\left(\frac{\pi k\log(c)}{\log(a)}\right)\cos\left(2\pi k\frac{\log(x)}{\log(a)}\right)\\
&\hspace{1.6cm}-\cot\left(\frac{\pi k\log(d)}{\log(a)}\right)\cos\left(2\pi k\frac{\log(x)}{\log(a)}\right)+\sin\left(2\pi k\frac{\log(x)}{\log(a)}\right)\bigg)\;\;\forall k\in\mathbb{N}.
\end{split}
\end{displaymath}
By exchanging the variable $a$ with all other variables $b$, $c$ and $d$ ("permuting $a,b,c,d$"), we get all four Residues (the "$a,b,c,d$ -Residues"), which have all the same structure. Summing everything up, we get our formula for $N_{a,b,c,d}(x)$. Setting $a=2$, $b=3$, $c=5$ and $d=7$, we get also the formula for $N_{2,3,5,7}(x)$.
\end{proof}

\noindent And so on.\\
These formulas are exactly equivalent to the previous mentioned formulas.

\section{The Formula for the Counting Function of the Natural Numbers of the Form $a^{p^{2}}b^{q^{2}}$}

\noindent Let $a,b\in\mathbb{N}$ such that $a<b$ and $\gcd(a,b)=1$.\\
\noindent For $x\in\mathbb{R}_{0}^{+}$, we define the function $N^{(2)}_{a,b}(x)$ by
\begin{displaymath}
\begin{split}
N^{(2)}_{a,b}(x):&=\sum_{\substack{a^{p^{2}}b^{q^{2}}\leq x\\p\in\mathbb{N}_{0},q\in\mathbb{N}_{0}}}1.
\end{split}
\end{displaymath}

\noindent Moreover, we define\begin{displaymath}
\begin{split}
S^{(2)}_{a,b}:&=\left\{a^{p^{2}}b^{q^{2}}:p\in\mathbb{N}_{0},q\in\mathbb{N}_{0}\right\},\\
\chi_{S^{(2)}_{a,b}}(x):&=\begin{cases}1&\text{if $x\in S^{(2)}_{a,b}$}\\0&\text{if $x\notin S^{(2)}_{a,b}$}\end{cases}.
\end{split}
\end{displaymath}

\noindent We have that
\begin{displaymath}
\begin{split}
N^{(2)}_{a,b}(x)&=1+\sum_{k=0}^{\left\lfloor\sqrt{\log_{b}(x)}\right\rfloor}\left\lfloor\sqrt{\log_{a}\left(\frac{x}{b^{k^{2}}}\right)}\right\rfloor+\left\lfloor\sqrt{\log_{b}(x)}\right\rfloor.
\end{split}
\end{displaymath}

\noindent We have also the following
\begin{theorem}(Formula for $N^{(2)}_{a,b}(x)$)\\
\noindent For every real number $x>1$, we have that
\begin{displaymath}
\begin{split}
N^{(2)}_{a,b}(x)&=\frac{\pi\log(x)}{4\sqrt{\log(a)\log(b)}}+\frac{1}{2}\sqrt{\frac{\log(x)}{\log(a)}}+\frac{1}{2}\sqrt{\frac{\log(x)}{\log(b)}}+\frac{1}{4}-\frac{1}{2}B^{*}_{1}\left(\left\{\sqrt{\frac{\log(x)}{\log(a)}}\right\}\right)\\
&\quad-\frac{1}{2}B^{*}_{1}\left(\left\{\sqrt{\frac{\log(x)}{\log(b)}}\right\}\right)+\sqrt{\log(x)}\sum_{n=1}^{\infty}\sum_{m=1}^{\infty}\frac{J_{1}\left(2\pi\sqrt{\frac{n^{2}\log(a)+m^{2}\log(b)}{\log(a)\log(b)}\log(x)}\right)}{\sqrt{n^{2}\log(a)+m^{2}\log(b)}}\\
&\quad+\frac{1}{2}\sqrt{\frac{\log(x)}{\log(a)}}\sum_{k=1}^{\infty}\frac{J_{1}\left(2\pi k\sqrt{\frac{\log(x)}{\log(b)}}\right)}{k}+\frac{1}{2}\sqrt{\frac{\log(x)}{\log(b)}}\sum_{k=1}^{\infty}\frac{J_{1}\left(2\pi k\sqrt{\frac{\log(x)}{\log(a)}}\right)}{k}+\frac{1}{2}\chi_{S^{(2)}_{a,b}}(x).
\end{split}
\end{displaymath}
\end{theorem}

\noindent This formula converges very rapidly.\\

\noindent Setting $a=2$ and $b=3$, we get
\begin{corollary}(Formula for $N^{(2)}_{2,3}(x)$)\\
\noindent For every real number $x>1$, we have that
\begin{displaymath}
\begin{split}
N^{(2)}_{2,3}(x)&=\frac{\pi\log(x)}{4\sqrt{\log(2)\log(3)}}+\frac{1}{2}\sqrt{\frac{\log(x)}{\log(2)}}+\frac{1}{2}\sqrt{\frac{\log(x)}{\log(3)}}+\frac{1}{4}-\frac{1}{2}B^{*}_{1}\left(\left\{\sqrt{\frac{\log(x)}{\log(2)}}\right\}\right)\\
&\quad-\frac{1}{2}B^{*}_{1}\left(\left\{\sqrt{\frac{\log(x)}{\log(3)}}\right\}\right)+\sqrt{\log(x)}\sum_{n=1}^{\infty}\sum_{m=1}^{\infty}\frac{J_{1}\left(2\pi\sqrt{\frac{n^{2}\log(2)+m^{2}\log(3)}{\log(2)\log(3)}\log(x)}\right)}{\sqrt{n^{2}\log(2)+m^{2}\log(3)}}\\
&\quad+\frac{1}{2}\sqrt{\frac{\log(x)}{\log(2)}}\sum_{k=1}^{\infty}\frac{J_{1}\left(2\pi k\sqrt{\frac{\log(x)}{\log(3)}}\right)}{k}+\frac{1}{2}\sqrt{\frac{\log(x)}{\log(3)}}\sum_{k=1}^{\infty}\frac{J_{1}\left(2\pi k\sqrt{\frac{\log(x)}{\log(2)}}\right)}{k}+\frac{1}{2}\chi_{S^{(2)}_{2,3}}(x).
\end{split}
\end{displaymath}
\end{corollary}

\noindent We have that
\begin{displaymath}
\begin{split}
S^{(2)}_{2,3}:&=\left\{2^{p^{2}}3^{q^{2}}:p\in\mathbb{N}_{0},q\in\mathbb{N}_{0}\right\}\\
&=\left\{1,2,3,6,16,48,81,162,\ldots\right\},
\end{split}
\end{displaymath}

\noindent and therefore we get the following table:

\begin{table}[htbp]
\begin{center}
\begin{tabular}{| c | c | c | c |}
\hline
$x$&$N^{(2)}_{2,3}(x)$&$\text{Formula for }N^{(2)}_{2,3}(x)$&$\text{Number of terms $(n,m)$ needed with $k=400$}$\\
\hline
\bfseries $1$&$1$&$1.077194794603379$&$(n,m)=(1,1)\text{\;at $x=1.1$}$\\
\hline
\bfseries $10$&$4$&$4.069103424005291$&$(n,m)=(1,1)$\\
\hline
\bfseries $10^{2}$&$7$&$7.000949506610362$&$(n,m)=(5,5)$\\
\hline
\bfseries $10^{3}$&$9$&$9.086395912838084$&$(n,m)=(3,3)$\\
\hline
\bfseries $10^{4}$&$11$&$11.038613589829053$&$(n,m)=(5,5)$\\
\hline
\bfseries $10^{5}$&$15$&$15.012706923272531$&$(n,m)=(5,5)$\\
\hline
\bfseries $10^{6}$&$17$&$17.046462385363300$&$(n,m)=(5,5)$\\
\hline
\bfseries $10^{7}$&$18$&$18.408421860888305$&$(n,m)=(9,9)$\\
\hline
\bfseries $10^{8}$&$22$&$22.127760008955621$&$(n,m)=(6,6)$\\
\hline
\bfseries $10^{9}$&$24$&$24.034210155019944$&$(n,m)=(8,8)$\\
\hline
\bfseries $10^{10}$&$26$&$26.009844154207983$&$(n,m)=(9,9)$\\
\hline
\bfseries $10^{10^{2}}$&$226$&$226.001668111078420$&$(n,m)=(39,39)$\\
\hline
\bfseries $10^{10^{3}}$&$2122$&$2122.031291011313557$&$(n,m)=(168,168)$\\
\hline
\bfseries $10^{10^{4}}$&$20886$&$20886.032472386492101$&$(n,m)=(400,400)$\\
\hline
\bfseries $10^{10^{5}}$&$207756$&$207756.0303040763527672$&$(n,m)=(1000,1000)$\\
\hline
\bfseries $10^{10^{6}}$&$2074033$&$2074033.0733802760244109$&$(n,m)=(1400,1400)$\\
\hline
\end{tabular}
\caption{Values of $N^{(2)}_{2,3}(x)$}\end{center}
\end{table}

\section{Conclusion}

\noindent We have presented and proved the formulas for the distribution of every smooth number sequence. This article and the proofs of these formulas will soon be published in a Journal.

\bigskip
\hrule
\bigskip

\noindent 2010 {\it Mathematics Subject Classification}: Primary 40A30; Secondary 11Y55.

\noindent\emph{Keywords:} distribution of smooth numbers, distribution of friable numbers, distribution of $2$-smooth numbers (powers of $2$), distribution of $3$-smooth numbers (harmonic numbers), distribution of $5$-smooth numbers (regular numbers or Hamming numbers), distribution of $7$-smooth numbers (Humble numbers or "highly composite numbers"), distribution of all smooth numbers, distribution of all friable numbers, distribution of $p_{n}$-smooth numbers, distribution of the natural numbers of the form $a^{p}$ less than or equal to $x$, distribution of the natural numbers of the form $a^{p}b^{q}$ less than or equal to $x$, distribution of the natural numbers of the form $a^{p}b^{q}c^{l}$ less than or equal to $x$, distribution of the natural numbers of the form $a^{p}b^{q}c^{l}d^{f}$ less than or equal to $x$, distribution of the natural numbers of the form $a_{1}^{q_{1}}a_{2}^{q_{2}}a_{3}^{q_{3}}\cdots a_{n}^{q_{n}}$ less than or equal to $x$, distribution of the natural numbers of the form $2^{p}3^{q}$ less than or equal to $x$, distribution of the natural numbers of the form $2^{p}3^{q}5^{l}$ less than or equal to $x$, distribution of the natural numbers of the form $2^{p}3^{q}5^{l}7^{f}$ less than or equal to $x$, distribution of the natural numbers of the form $2^{q_{1}}3^{q_{2}}5^{q_{3}}\cdots p_{n}^{q_{n}}$ less than or equal to $x$, distribution of the natural numbers of the form $a^{p^{2}}b^{q^{2}}$ less than or equal to $x$, distribution of the natural numbers of the form $2^{p^{2}}3^{q^{2}}$ less than or equal to $x$.


\begin{thebibliography}{9}

\bibitem{1}
http://mathworld.wolfram.com/SmoothNumber.html

\bibitem{2}
https://en.wikipedia.org/wiki/Smooth\_number

\bibitem{3}
https://oeis.org/A003586

\bibitem{4}
B. C. Berndt, Ramanujan's Notebooks, Part IV, \emph{Springer}, 66-69, 1994.

\bibitem{5}
Srinivasa Ramanujan, Collected Papers of Srinivasa Ramanujan, \emph{Oxford University Press}, 2000.

\bibitem{6}
math.stackexchange.com/questions/15966/ramanujans-first-letter-to-hardy-and-the-number-of-$3$-smooth-integers

\bibitem{7}
http://math.stackexchange.com/questions/17267/on-the-consequences-of-an-exact-de-bruijn-function-or-if-ramanujan-had-more?

\bibitem{8}
A. Granville, Smooth numbers: computational number theory and beyond, \emph{Algorithmic Number Theory}, MSRI Publications, Volume 44, 2008.

\bibitem{9}
https://en.wikipedia.org/wiki/Regular\_number

\bibitem{10}
https://oeis.org/A051037

\bibitem{11}
G. H. Hardy, A lattice-point problem, in "Ramanujan: Twelve lectures on subjects
\hspace*{0.1cm} suggested by his life and work", \emph{AMS Chelsea}, page 72, 1999.

\bibitem{12}
G. H. Hardy, J. E. Littlewood, Some problems of Diophantine approximation: The
\hspace*{0.1cm} lattice-points of a right-angled triangle, \emph{Proc. London Math. Soc.}, page 35, 1922.

\bibitem{13}
Emanuele Tron, SNS Pisa, emanuele.tron@sns.it.

\bibitem{14}
http://oeis.org/A002473

\bibitem{15}
cf. Donald Mintz, donald.jmintz@gmail.com.

\bibitem{16}
cf. Donald J. Mintz, $2$,$3$ Sequence as Binary Mixture, \emph{The Fibonacci Quarterly},\\
\hspace*{0.1cm} Volume 19, 351-360, 1981.

\end{thebibliography}
\end{document}